\documentclass[a4paper]{article}
\usepackage[english]{babel}
\usepackage[utf8]{inputenc}
\usepackage{amsmath}
\allowdisplaybreaks
\usepackage{graphicx}
\usepackage{subfigure} 
\usepackage{cite}
\usepackage[colorinlistoftodos]{todonotes}
\usepackage{lineno}
\usepackage[misc]{ifsym} 
\usepackage{amssymb}
\usepackage{multirow}
\usepackage{makecell}
\usepackage{booktabs}
\usepackage{diagbox}
\usepackage{mathrsfs}
\usepackage[colorlinks=true,urlcolor=blue,citecolor=red]{hyperref}
\usepackage{amsthm}
\newtheorem{theorem}{Theorem}[section]
\newtheorem{lemma}[theorem]{Lemma}
\newtheorem{corollary}[theorem]{Corollary}
\newtheorem{remark}[theorem]{Remark}
\newtheorem{definition}[theorem]{Definition}
\newtheorem{proposition}[theorem]{Proposition}

\numberwithin{equation}{section}
\DeclareMathOperator{\dive}{div}
\newcommand\keywords[1]{\par\noindent\textbf{Keywords.} #1}
\newcommand\msc[1]{\par\noindent\textbf{MSC codes.} #1}

\title{A Diffeomorphism Groupoid and Algebroid Framework for Discontinuous Image Registration}
\author{}
\begin{document}
\bibliographystyle{abbrv}
\maketitle
% \linenumbers

\centerline{\scshape Lili Bao\textsuperscript{1}, Bin Xiao\textsuperscript{2}, 
Shihui Ying\textsuperscript{3,4*} and Stefan Sommer\textsuperscript{1*}}

\medskip
{\footnotesize
 \centerline{ \textsuperscript{1}Department of Computer Science, University of Copenhagen}
 \centerline{Universitetsparken 1, Copenhagen 2100, Denmark}

\centerline{\textsuperscript{2}Institute for Medical Imaging Technology, Ruijin Hospital}
 \centerline{Shanghai Jiao Tong University School of Medicine, Shanghai 201800, China}

 \centerline{\textsuperscript{3}Shanghai Institute of Applied Mathematics and Mechanics}
  \centerline{Shanghai University, Shanghai 200444, China}
 
 \centerline{\textsuperscript{4}School of Mechanics and Engineering Science}
  \centerline{Shanghai University, Shanghai 200444, China}

\medskip

 {\small\ttfamily
 \centerline{\texttt{lili.bao@di.ku.dk,\ xb12501@rjh.com.cn,\ shying@shu.edu.cn,\ sommer@di.ku.dk}}
}
}

\begingroup
\renewcommand{\thefootnote}{}
\footnotetext{* Corresponding authors.}
\footnotetext{~Funding: The work of the third author was supported by the National Natural Science Foundation of China (No. 12531019). The fourth author was supported by the Novo Nordisk Foundation grants NNF18OC0052000, NNF24OC0093490, and NNF24OC0089608, the VILLUM FONDEN research grant 40582, and UCPH Data+ Strategy 2023 funds for interdisciplinary research.}

\begin{abstract}
In this paper, we propose a novel mathematical framework for piecewise diffeomorphic image registration that involves discontinuous sliding motion using a diffeomorphism groupoid and algebroid approach. 
The traditional Large Deformation Diffeomorphic Metric Mapping (LDDMM) registration method builds on Lie groups, which assume continuity and smoothness in velocity fields, limiting its applicability in handling discontinuous sliding motion.
To overcome this limitation, we extend the diffeomorphism Lie groups to a framework of discontinuous diffeomorphism Lie groupoids, allowing for discontinuities along sliding boundaries while maintaining diffeomorphism within homogeneous regions. We provide a rigorous analysis of the associated mathematical structures, including Lie algebroids and their duals, and derive specific Euler-Arnold equations to govern optimal flows for discontinuous deformations. 
Numerical tests are performed to validate the efficiency of the proposed approach. 

\end{abstract}


\vspace{0.5em}

\msc{58H05, 68U10, 22A22, 65D18}

\section{Introduction}
\label{sec: Introduction}
Image registration is a fundamental and crucial technique in computer vision and medical image processing. The aim of image registration is to find reasonable spatial deformations between two or more images \cite{oliveira2014medical, viergever2016survey}.
In many applications, maintaining the integrity and consistency of structures is essential, making diffeomorphic deformations highly desirable \cite{marsland2020riemannian}. 
The Large Deformation Diffeomorphic Metric Mapping (LDDMM) framework plays an important role in such diffeomorphic registration as it provides good registration results along with a solid mathematical foundation, allowing meaningful statistics to be computed on the registration results \cite{beg2005computing,sommer2012kernel}.

However, challenges arise when encountering discontinuous deformations, such as the discontinuous sliding motion observed between lung tissues and surrounding structures during breathing \cite{bao2024sliding,bao2024time}. Registration of lungs is an essential part in medical imaging, for example for motion estimation and tumor radiotherapy. The LDDMM method, being designed for smooth and continuous deformations, struggles to capture these discontinuous movements.
% However, in abdominal imaging, discontinuous sliding motions occur during breathing between the lung tissues and the surrounding structures. 
% LDDMM and similar diffeomorphic registration methods are built on the assumption of continuity and smoothness of the deformations. This assumption limits their applicability in dealing with such discontinuous sliding motion.

To address this limitation, we propose to use the mathematical theory proposed by Izosimov and Khesin \cite{izosimov2018vortex, izosimov2024geometry} to model such discontinuous image registration, thereby extending the diffeomorphism group representation of deformation to a more general groupoid based framework. 
In our framework, discontinuities along sliding hypersurfaces are naturally handled using the structure of Lie groupoids and algebroids, while maintaining diffeomorphism in homogeneous regions. Specifically, we construct a discontinuous diffeomorphism groupoid $\mathrm{DDiff}(M) \rightrightarrows \mathrm{V}(M)$ to describe discontinuous deformations. 
We also explore the associated mathematical structures, including Lie algebroids $\text{DVect}(M)$ and their duals $\mathrm{DVect}(M)^*$, and formulate specific Euler-Arnold equations for optimal flows to drive the discontinuous deformation.
The framework and examples of its effect on discontinuous image registration are illustrated in Figure \ref{fig: vortex_sheet_LDDMM}.

\begin{figure}[htbp]
    \centering
    \includegraphics[width=1.0\textwidth]{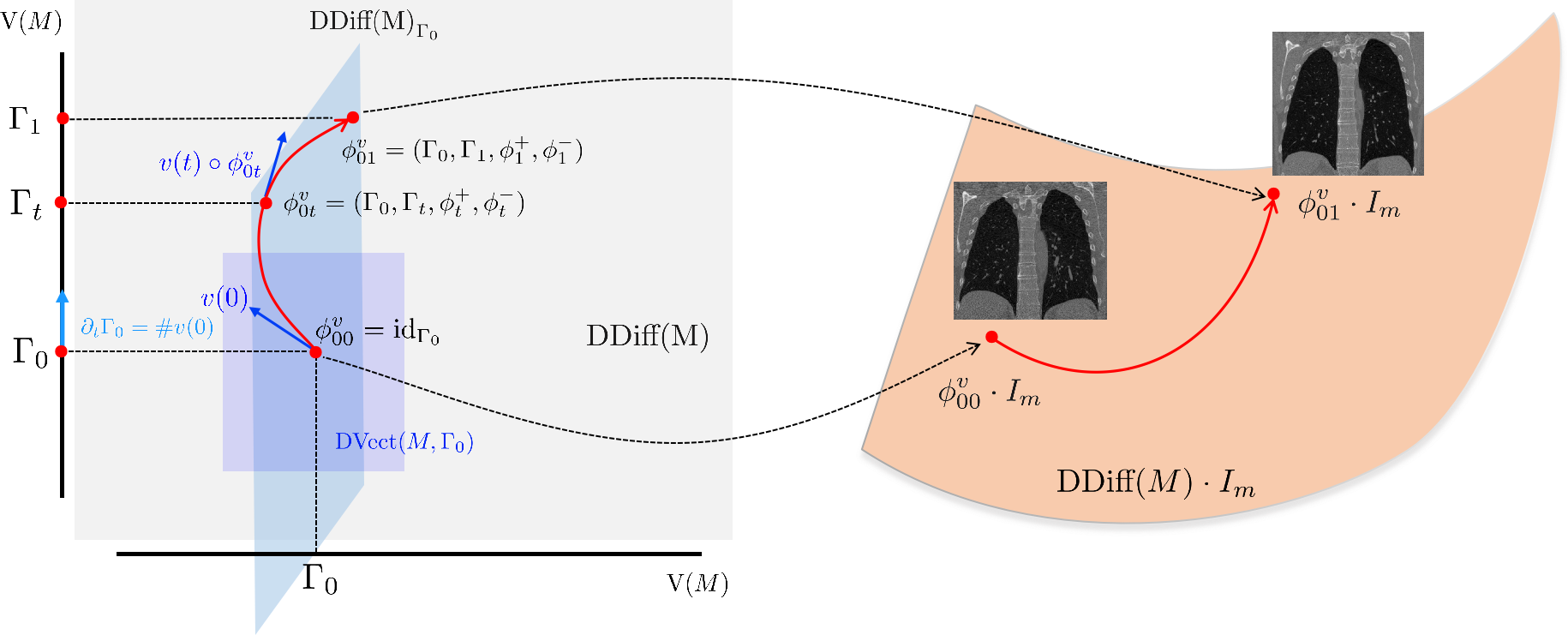}
    \caption{Illustration of the discontinuous image registration based on Lie groupoid and algebroid. (left) Illustration of the geodesic flow on groupoids $\mathrm{DDiff}(M) \rightrightarrows \mathrm{V}(M)$ of discontinuous diffeomorphisms with respect to the metric $\left\langle \cdot, \cdot \right\rangle_{\text{DVect}(M)}$, governed by the Euler-Arnold equation \eqref{eq: Euler-Arnold}. 
    (right) The action of the discontinuous diffeomorphism groupoid on the moving image $I_m$. }
    \label{fig: vortex_sheet_LDDMM}
\end{figure}

% Traditional diffeomorphism based registration methods often struggle to accurately capture and model these discontinuous deformations.
% Thus a significant gap has emerged in the application of diffeomorphism theories to discontinuous image registration.

% To address this issue, in this paper, we propose a new mathematical framework for discontinuous image registration by extending the diffeomorphism group representation of deformation to a more general groupoid based framework. 
\subsection{Related work}
\label{subsec: Related work}
Various approaches to tackle the challenge of sliding motion registration have been explored in the literature. These methods can generally be divided into two categories: displacement-based methods and velocity-based methods, both of which focus on developing elaborate regularization constraints on deformations. 

Various special regularizations based on displacement fields have been proposed to deal with discontinuous deformation. For example, total variation (TV) has been used to describe the discontinuity of the deformation field in \cite{frohn-schauf_multigrid_2008}. 
To balance the smoothness and discontinuity of the deformation field across different image regions, the modified total variation (MTV) method has been proposed in \cite{Chumchob2010A}. Recently a time multiscale regularization strategy has been proposed in \cite{bao2024time} to capture finer scale deformations and preserve sharper structures in images. 
Direction-dependent regularization methods have been employed in \cite{fu_adaptive_2018,schmidt-richberg_estimation_2012}, where the deformation field is decomposed into normal and tangential directions and regularized separately. 

Velocity-based methods model the deformation field through a flow generated by a velocity field, as in the case of diffeomorphic frameworks like the Large Deformation Diffeomorphic Metric Mapping (LDDMM).  To adapt velocity-based approaches for sliding motions, researchers have introduced specialized regularization techniques, such as decomposing the velocity fields into tangential and normal components \cite{risser_dieomorphic_2011, risser_piecewise-diffeomorphic_2013}, and introducing a shear strain rate to control the amount of shear to approximate sliding motion in \cite{mang2016constrained}. More recently, an extension of LDDMM has been proposed that incorporates both zeroth- and first-order momenta with a non-differentiable kernel, offering a novel approach to handling sliding motion \cite{bao2024sliding}. All these methods have shown promise in modeling sliding motion, but a unified framework for discontinuous image registration remains elusive.

This motivates the need for a more general and mathematically rigorous framework that can naturally model discontinuities while preserving the desirable properties of diffeomorphisms in homogeneous regions. In this work, we introduce such a framework by using groupoids instead of smooth Lie groups for discontinuous image registration. 
% This framework extends the diffeomorphism Lie groups to a framework of discontinuous diffeomorphism
% Lie groupoids, allowing for discontinuities along sliding boundaries while maintaining diffeomorphism within homogeneous regions.

This work is inspired by the seminal developments in fluid dynamics and geometric mechanics in the works of Izosimov and Khesin \cite{izosimov2018vortex, khesin2023geometric, izosimov2024geometry}. Their framework, which introduces discontinuous flows and vortex sheet groupoids in the context of incompressible fluid dynamics, serves as the mathematical foundation for our proposed approach. 
Our paper has two primary aims: First, extending the groupoid framework to handle compressible cases needed for registration; and second, applying this extended theory to image registration, enabling the modeling of discontinuous deformations. 
% We adapt and extend these ideas to compressible cases and model discontinuous deformations in image registration. 
For clarity and consistency, we follow the notation and conventions introduced in \cite{izosimov2018vortex}.

\subsection{Content and outline}
\label{subsec: Content and outline}
The organization of the paper is as follows.
% In Section \ref{sec: Contributions}, we highlight the contributions of our work.
In Section \ref{sec: Mathematical Background}, we briefly review some mathematical background of the proposed framework, including the LDDMM approach to image registration and Lie groupoids and Lie algebroids.
In Section \ref{sec: Proposed framework}, the detailed formulation of the proposed framework is explained, we first describe the concept of discontinuous diffeomorphism groupoid and the associated mathematical structures, including Lie algebroids of discontinuous vector fields and their duals. Following this, we derive the Euler-Arnold equations for optimal flows in the present context. Then we show how the proposed groupoid representation can be used for discontinuous image registration problems.
In Section \ref{sec: Experiments}, we explore the applications of the proposed framework to image registration.
We conclude the paper and discuss possible future works in Section \ref{sec: Conclusion}.
% \section{Contributions}
% \label{sec: Contributions}
The paper thus has the following contributions:
\begin{enumerate}
    %% \item We derive the theory of the diffeomorphism groupoid, specifically transitioning from the incompressible to the compressible framework.
    %\item We develop a Lie groupoid extension to model discontinuous diffeomorphisms and provide a detailed account of the theoretical foundation.
    %\item We explore the associated mathematical structures including Lie algebroids and their dual.
    %% \item We define a new metric on the proposed Lie groupoid to get the distance between two images involving discontinuous movement, ensuring diffeomorphism remained across different regions and discontinuous along a hypersurface for modeling sliding motion. 
    %\item We propose the first formulation of the Euler-Poincar\'{e} equations on Lie groupoids with general inertia operators, extending the traditional EPDiff from Lie groups to Lie groupoids (to the best of our knowledge).
    %% derive the Euler-Arnold equations for optimal flows which are fundamental for the theoretical understanding of the discontinuous diffeomorphism groupoid and necessary for practical implementations. 
    %\item We show how the proposed framework can be used for discontinuous image registration, ensuring that diffeomorphisms are preserved across different regions while allowing for discontinuities along a hypersurface for modeling sliding motion. 
    %% \item We develop accompanying numerical methods to enable the theory to be used for image registration.
    \item We show how Lie groupoids can be used in shape analysis to model piecewise diffeomorphic transformations, and we provide a detailed account of the theoretical foundation.
    \item We derive the Euler-Poincar\'e equations on Lie groupoids with general inertia operators in the compressible case, extending the traditional EPDiff equations as used in shape analysis from Lie groups to Lie groupoids.
    \item We show how the proposed framework can be used for discontinuous image registration with numerical schemes, ensuring that diffeomorphisms are preserved across different regions while allowing for discontinuities along a hypersurface for modeling sliding motion.
\end{enumerate}

\section{Mathematical Background}
\label{sec: Mathematical Background}
In this section, we review some basic mathematical concepts relevant to this work, including the well-known LDDMM registration method and concepts from Lie groupoids theory. 
% This background will help contextualize the development of our framework for discontinuous image registration in the following sections.
% based on the flow approach namely Large Deformation Diffeomorphic Metric Mapping (LDDMM) registration and Lie groupoids.

\subsection{LDDMM registration}
\label{subsec: LDDMM registration}
We begin by briefly reviewing the LDDMM approach, which has a solid mathematical foundation through the use of Lie groups.
% In this subsection, we give a brief overview of the LDDMM approach to image registration. 
A detailed treatment can be found in \cite{younes_shapes_2019}, and with a perspective from fluid dynamics in \cite{khesin2008geometry,arnold2014topological}.
The problem of image registration based on diffeomorphisms can be described as follows:
Given two images $I_m$ and $I_f$ defined on image domain $\Omega \subset \mathbb{R}^{d}$, where $I_m$ is the moving image and $I_f$ is the fixed image, we aim to find a reasonable diffeomorphic deformation $\phi: \Omega \rightarrow \mathbb{R}^{d}$ such that the fixed image $I_f$ and the deformed moving image $I_m^* = \phi \cdot I_m$ are as closely as possible, i.e., 
$I_f \approx \phi \cdot I_m= I_m\circ\phi^{-1}$ (see Figure \ref{fig: image registration}). Additionally, we hope such a $\phi$ to be as simple as possible, i.e., as close as possible to the identity deformation.  
Based on this, the mathematical formulation of the registration problem can be given as:
\begin{equation}
    \phi^\ast = \mathop{\arg\min}_{\phi} E(\phi) := E_S(\phi \cdot I_m, I_f) + E_R(\phi),
    \label{eq: functional}
\end{equation}
% \begin{align}
%     \phi^\ast =\mathop{\arg\min}_{\phi} E(\phi):&= E_S(\phi \cdot I_m, I_f)+ E_R(\phi)\nonumber\\ 
%     & = E_S(\phi \cdot I_m, I_f) + d_V^2(\text{Id}, \phi)
% \label{eq: LDDMM functional} 
% \end{align}
where the first term $E_S(\phi \cdot I_m, I_f)$ is the similarity term that measures the similarity between the two images, and the second term $E_R(\phi)$ is the regularization term, ensuring the smoothness of the deformation field.

 \begin{figure}[!htbp]
\centering
\includegraphics[width=0.8\textwidth]{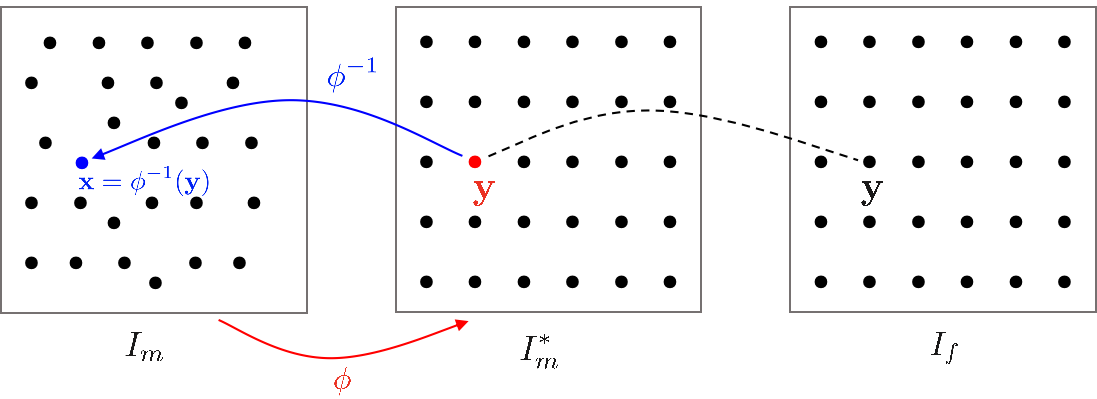}
\caption{\label{fig: image registration}An illustration of image registration. }
\end{figure}

In the LDDMM method, it is assumed that the deformation $\phi = \phi_{01}^v$ is obtained as the endpoint of a flow generated by
a time-dependent velocity field $v(t)$, guided by the following differential equation \cite{dupuis1998variational, pennec2019riemannian}:
\begin{equation}
    \frac{\partial \phi(t, x)}{\partial t} = v(t, \phi(t, x)),\ \phi(0, x) = x, \forall x \in \Omega,
    \label{eq: differential equation}
\end{equation}
achieving registration between the two images. To ensure that the deformation field is as simple and smooth as possible, and close to the identity deformation, the regularization term $E_R(\phi)$ can be defined using the distance given as follows \cite{TrouveA,younes_shapes_2019}:
\begin{equation}
    E_R(\phi) = d_V^2(\text{Id}, \phi)
    = \min_{v(t) \in V, \phi_{01}^v= \phi} \int_0^1 \|v(t)\|_V^2 dt.
    \label{eq: LDDMM regularizer}
\end{equation}
The velocity $v$ lies in the admissible Hilbert space $V$ with associated norm $\Vert \cdot \Vert_V$, and $V$ is usually produced by using a differential operator $\mathcal{I}$ \cite{bauer2023liouville}. The inner product is defined as follows:
\begin{equation}
    \langle u, v \rangle_V := \langle \mathcal{I}u, v \rangle_{L^2} = \langle u, \mathcal{I}v \rangle_{L^2},
    \label{eq: differential operator}
\end{equation}
where $\langle \cdot, \cdot \rangle_{L^2}$ is the $L^2$ inner product of square-integrable vector fields on $\Omega$. This induces the norm $\Vert v \Vert_V := \sqrt{\langle v , v \rangle_V}$. 
Let $\text{Diff}(\Omega)$ denote the Lie group of smooth diffeomorphisms on $\Omega$, then the set of flows at time 1 built from $V$, denoted by $G_V$, is a subgroup of $\text{Diff}(\Omega)$, and is a Lie group \cite{younes_shapes_2019}. The tangent space of $G_V$ at the identity deformation is the space $V = \text{Vect}(\Omega)$ of smooth vector fields on $\Omega$ \cite{khesin2008geometry}. 
Figure \ref{fig: LDDMM} illustrates the LDDMM registration framework.

In fact, the operator $\mathcal{I}$ can  also be viewed as an operator from $V$ to its dual space $V^{\ast}$, namely $\mathcal{I}: V \rightarrow V^{\ast}$. 
In fluid dynamics, $\mathcal{I}$ is commonly referred to as the inertia operator \cite{arnold2014topological}. To connect with the geometric formulation used later, we represent the dual quantity $\mathcal{I}v$ by the 1-form density \cite{holm2005momentum,holm2009geometric}:
\begin{equation*}
    \tilde{m} = \mathcal{I}v = m \otimes \mu = \sum_{i} m_i dx_i \otimes (dx_1 \wedge \cdots \wedge dx_n),
\end{equation*}
where $m = \sum_{i} m_i dx_i$ is the associated 1-form and $\mu$ is the volume form. Following the standard convention in the LDDMM literature, we refer to the associated 1-form $m$ simply as the momentum.
% The operator $\mathcal{I}$ is also referred to as the momentum operator. This is because $\mathcal{I}$ relates the inner product on $V$ to the $L^2$ inner product, and the image of the element $v$ under the operator $\mathcal{I}$, denoted as $\mathcal{I}v$, is called the momentum, denoted as $m$, i.e., $m = \mathcal{I}v$. 
\begin{figure}[!t]
    \centering
    \includegraphics[width=0.8\textwidth]{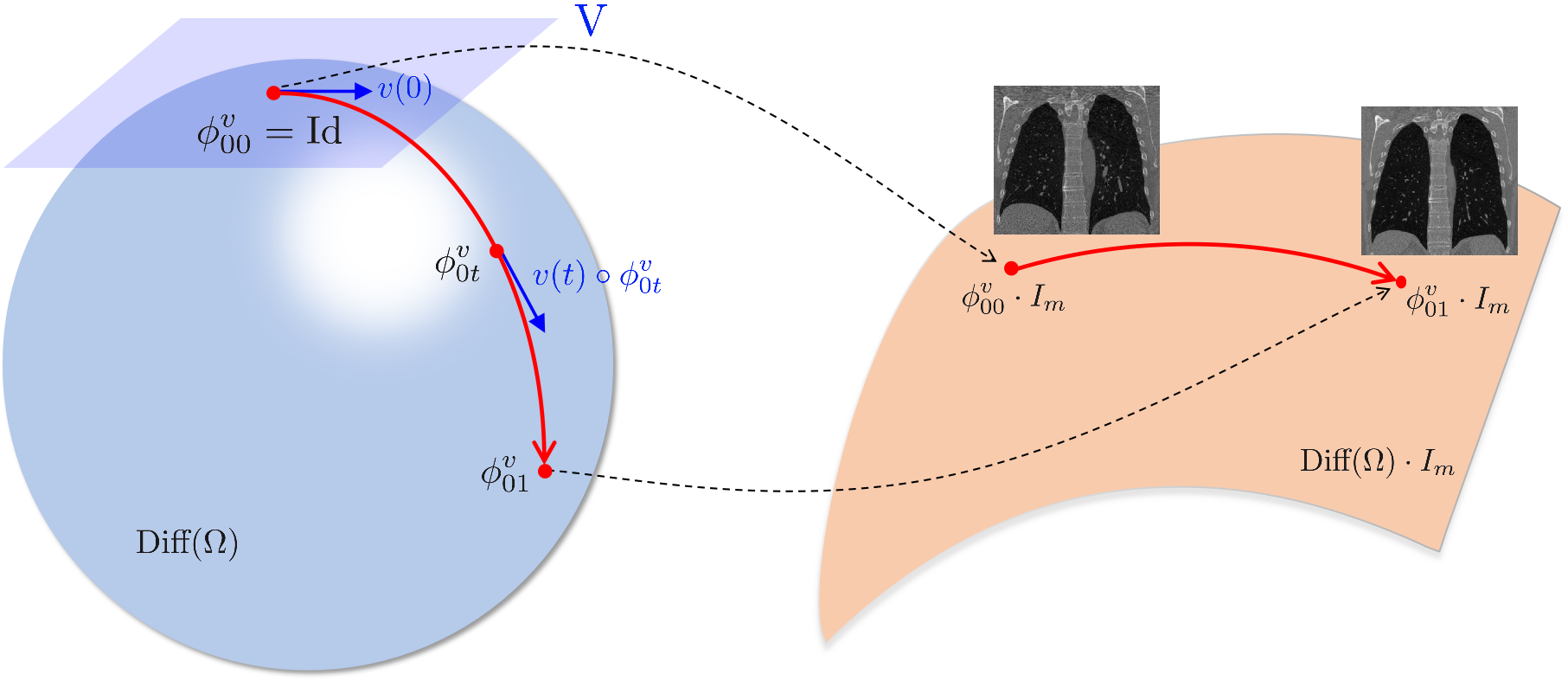}
    \caption{(left) Illustration of the geodesic flow on groups $\mathrm{Diff}(\Omega)$ of diffeomorphisms with respect to the right-invariant metric $\left\langle \cdot, \cdot \right\rangle_{\text{Vect}(\Omega)}$. 
    (right) The action of the diffeomorphism group on the moving image $I_m$. }
    \label{fig: LDDMM}
\end{figure}

 Minimizing the functional in \eqref{eq: functional} with the regularizer in \eqref{eq: LDDMM regularizer} yields a geodesic with the shortest length path in $G_V$, uniquely determined by integrating the Euler-Poincar\'{e} differential equations on the diffeomorphism group (EPDiff) with an initial condition \cite{miller2006geodesic,younes_shapes_2019}:
\begin{align}
        &\frac{\partial m(t)}{\partial t} + v(t) \cdot \nabla m(t) + (\nabla v(t))^{T} \cdot m(t) + m(t) (\dive v(t)) = 0, \label{eq: EPDiff}\\    
        &v(t) = K \ast m(t).    \label{eq: connect}
\end{align}
That is, given an optimal initial momentum $m(0)$, according to the time evolution of momentum \eqref{eq: EPDiff} and the connection between momentum and velocity \eqref{eq: connect}, 
the entire path of velocities $v(t)$ can be recovered, and the corresponding optimal deformation $\varphi$ also can be reconstructed via equation \eqref{eq: differential equation}.
See e.g. \cite{younes_shapes_2019} for more details.

For the convenience of subsequent discussions, we describe an equivalent form of the EPDiff equation \cite{micheli2013sobolev, mumford2013euler}. 
In terms of the 1-form density $\tilde{m}$, 
% Instead of considering the momentum $m(t)$ a vector field, it can be considered a differential form. Let $\tilde{m}$ be a 1-form density \cite{holm2005momentum,holm2009geometric}:
% \begin{equation*}
%     \tilde{m} = m \otimes \mu = \sum_{i} m_i dx_i \otimes (dx_1 \wedge \cdots \wedge dx_n),
% \end{equation*}
equation \eqref{eq: EPDiff} can equivalently be written as
\begin{equation}\label{eq: one-form density EPDiff}
    \partial_t \tilde{m}(t) + \mathcal{L}_{v(t)}\tilde{m}(t) = 0,
\end{equation}
where $\mathcal{L}_{v(t)}\tilde{m}(t)$ denotes the Lie derivative of the 1-form density $\tilde{m}(t)$ along the vector field $v(t)$.

\subsection{Lie groupoid and Lie algebroid}
\label{subsec: Lie Groupoid and Lie Algebroid}
We now give a brief overview of Lie groupoids and their associated Lie algebroids, which will form the foundation for the framework in the next section. More details can be found in \cite{dufour2006poisson, izosimov2018vortex}.

\begin{definition}[Groupoid \cite{izosimov2018vortex}]
\label{def: groupoid}
A groupoid $\mathcal{G} \rightrightarrows B$ consists of a pair of sets, where $B$ is the set of objects and $\mathcal{G}$ is the set of arrows, equipped with the following structure:
\begin{enumerate}
    \item There are two maps, $\mathrm{src}$ and $\mathrm{trg}: \mathcal{G} \rightarrow B$, called the source and target maps.
    \item There is a partially defined binary operation on $\mathcal{G}$, $(g, h) \mapsto gh$, defined for all $g, h \in \mathcal{G}$ such that $\mathrm{src}(g) = \mathrm{trg}(h)$, and it satisfies the following properties:
    \begin{enumerate}
        \item $\mathrm{src}(gh) = \mathrm{src}(h)$, $\mathrm{trg}(gh) = \mathrm{trg}(g)$.
        \item Associativity: $g(hk) = (gh)k$ whenever these expressions are defined.
        \item Identity: For each $x \in B$, there exists an element $\mathrm{id}_x \in \mathcal{G}$ such that $\mathrm{src}(\mathrm{id}_x) = \mathrm{trg}(\mathrm{id}_x)=x$ and for every $g \in \mathcal{G}$, $\mathrm{id}_{\mathrm{trg}(g)} \cdot g = g \cdot \mathrm{id}_{\mathrm{src}(g)} = g$.
        \item Inverse: For each $g \in \mathcal{G}$, there exists an element $g^{-1} \in \mathcal{G}$ such that $\mathrm{src}(g^{-1}) = \mathrm{trg}(g)$, $\mathrm{trg}(g^{-1}) = \mathrm{src}(g)$, $g^{-1}g = \mathrm{id}_{\mathrm{src}(g)}$, and $gg^{-1} = \mathrm{id}_{\mathrm{trg}(g)}$.
    \end{enumerate}
\end{enumerate}
\end{definition}

\begin{figure}[!t]
    \centering
    \includegraphics[width=0.6\textwidth]{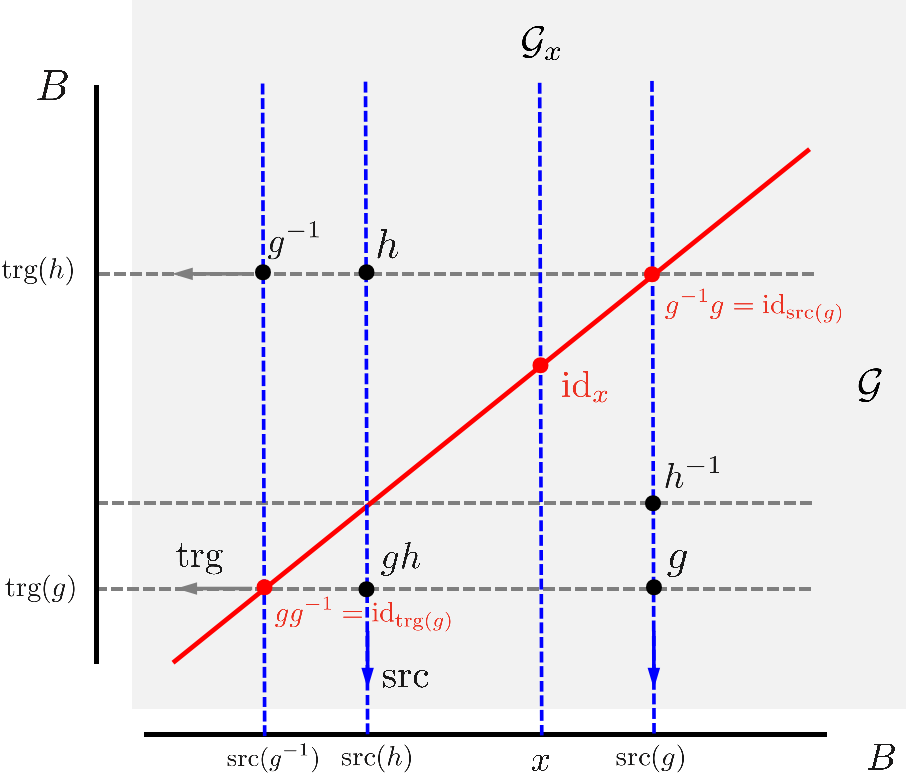}
    \caption{The structure of a groupoid $\mathcal{G} \rightrightarrows B$.}
    \label{fig: Lie groupoid}
\end{figure}

The structure of a groupoid, along with its inverses, identities, and multiplication operations, is illustrated in Figure \ref{fig: Lie groupoid}. In this diagram, the groupoid $\mathcal{G} \rightrightarrows B$ is represented as a square. The vertical projection represents the source map $\mathrm{src}: \mathcal{G} \rightarrow B$, while the horizontal projection represents the target map $\mathrm{trg}: \mathcal{G} \rightarrow B$. Elements on the diagonal where $\mathrm{src} = \mathrm{trg}$ consist of all the identity elements in $\mathcal{G}$. A pair of inverse elements is symmetric with respect to this diagonal.

\begin{definition}[Lie groupoid \cite{izosimov2018vortex}]
    A groupoid $\mathcal{G} \rightrightarrows B$ is called a Lie groupoid if both $\mathcal{G}$ and $B$ are manifolds, the source and target maps are submersions, and the multiplication, identity, and inversion maps are smooth.
\end{definition}

\begin{definition}[Lie algebroid \cite{izosimov2018vortex}]
    A Lie algebroid $\mathcal{A} \rightarrow B$ is a vector bundle $\mathcal{A}$ over a manifold $B$, equipped with a Lie bracket $[\cdot, \cdot]$ on its smooth sections and a vector bundle map $\#: \mathcal{A} \rightarrow TB$, called the anchor map, such that for any smooth sections $\zeta, \eta$ of $\mathcal{A}$ and any smooth function $f$ on $B$, the following identity holds:
    \begin{equation*}
        [\zeta, f\eta] = f[\zeta, \eta] + (\mathcal{L}_{\# \zeta}f)\eta.
    \end{equation*}
    Here, $B$ is called the base manifold, $TB$ is the tangent bundle of $B$, and $\#$ is the anchor map.
\end{definition}

\section{Lie Groupoids for Discontinuous Image Registration}
\label{sec: Proposed framework}
The LDDMM registration method based on Lie group theory treats smooth, continuous registration as a problem of smooth, continuous flow. To describe the discontinuous registration problem with sliding boundaries, we can treat it as a discontinuous flow problem involving vortex sheets, where the flow velocity exhibits discontinuous jumps at the vortex sheets. We will present the main contribution of this paper in this section: introducing the use of the discontinuous diffeomorphism groupoid in shape analysis and image registration, thereby allowing for discontinuities along hypersurfaces, and the associated mathematical structures, including Lie algebroids and their duals, and meanwhile giving the framework for discontinuous image registration using the proposed theory.

Our work builds upon the mathematical framework established by Izosimov and Khesin in 
\cite{izosimov2018vortex, khesin2023geometric} for modeling vortex sheets in incompressible fluid flows, which relies on the Lie groupoid $\mathrm{DSDiff}(M)$ of discontinuous volume-preserving diffeomorphisms and the Lie algebroid $\mathrm{DSVect}(M)$ of divergence-free discontinuous vector fields. On the mathematical side, the novelty of our work lies in extending the framework to the compressible case by using the Lie groupoid $\mathrm{DDiff}(M)$ of discontinuous diffeomorphisms and the Lie algebroid $\mathrm{DVect}(M)$ of discontinuous vector fields, without imposing the volume-preserving and divergence-free conditions. 
In this extended setting, the associated structures are modified accordingly to accommodate compressibility while preserving the fundamental geometric framework. 
% This generalization yields a more comprehensive theory, enabling applications for discontinuous image registration.

We closely follow the structural framework established in \cite[Section 6]{izosimov2018vortex}, maintaining the formulation of key mathematical constructs within our extended setting.
% the contribution differs from them. Whereas \cite{izosimov2018vortex} is focused on incompressible fluid flows with vortex sheets, we aim at compressible cases to create a more comprehensive theory. 
To this end, we
consider an $n$-dimensional Riemannian manifold $M$ equipped with a Riemannian volume form $\mu$. In the next subsection, we will describe the Lie groupoid $\mathrm{DDiff}(M)$ and its associated structures.%, adapted to the compressible case.

\subsection{The Lie groupoid of discontinuous diffeomorphisms}
\label{subsec: The Lie groupoid of discontinuous diffeomorphisms}

Let $\Gamma_0 \subset M$ be a hypersurface (often referred to as a vortex sheet) that divides $M$ into two regions, $D^+_{\Gamma_0}$ and $D^-_{\Gamma_0}$, such that $M = D^+_{\Gamma_0} \cup \Gamma_0 \cup D^-_{\Gamma_0}$.

We begin by defining the full groupoid of discontinuous diffeomorphisms. This groupoid consists of all diffeomorphisms discontinuous along a hypersurface, i.e. all quadruples $(\Gamma_1, \Gamma_2, \phi^+, \phi^-)$, where $\Gamma_1$ and $\Gamma_2$ are hypersurfaces in $M$, each an image of $\Gamma_0$ under some diffeomorphism of $M$, and $\phi^{\pm}: D_{\Gamma_{1}}^{\pm}\rightarrow D_{\Gamma_{2}}^{\pm}$ are diffeomorphisms, where $M = D_{\Gamma_{i}}^{+} \cup \Gamma_{i} \cup D_{\Gamma_{i}}^{-}$ for $i = 1,2$. The groupoid structure is equipped with source and target maps: the source of $(\Gamma_1,\Gamma_2,\phi^+,\phi^-)$ is $\mathrm{src}(\Gamma_1,\Gamma_2,\phi^+,\phi^-) = \Gamma_1$ and the target of $(\Gamma_1,\Gamma_2,\phi^+,\phi^-)$ is $\mathrm{trg}(\Gamma_1,\Gamma_2,\phi^+,\phi^-) = \Gamma_2$. The composition is defined as follows: 
\begin{equation}
    (\Gamma_2,\Gamma_3,\psi^+,\psi^-)(\Gamma_1,\Gamma_2,\phi^+,\phi^-)=(\Gamma_1,\Gamma_3,\psi^+\phi^+,\psi^-\phi^-).
\end{equation}

The base space of the full groupoid of discontinuous diffeomorphisms is the set of images of $\Gamma_0$ under diffeomorphisms in $\mathrm{Diff}(M)$.
% \begin{equation}
%     \mathrm{V}(M) = \{ \Gamma = \phi(\Gamma_0) \mid \phi \in \mathrm{Diff}(M) \}.
% \end{equation}

In this work, we do not study the full groupoid described above. Instead, we focus on a subgroupoid with additional regularity constraints. 
% Specifically, we require that the diffeomorphisms $\phi^+$ and $\phi^-$ are modeled on Sobolev $H^s$ spaces for $s>\frac{N}{2}+1$, where $N = \mathrm{dim}~M$. 
Specifically, we require that the diffeomorphisms $\phi^+$ and $\phi^-$ belong to the subgroup $G_V\subset \mathrm{Diff}(M)$. 
This regularity condition ensures that $\phi^+$ and $\phi^-$ are at least continuously differentiable. For convenience, we denote this subgroupoid as $\mathrm{DDiff}(M)$. 
The base space of the $\mathrm{DDiff}(M)$ thus is defined as follows:
\begin{equation}
    \mathrm{V}(M) = \{ \Gamma = \phi(\Gamma_0) \mid \phi \in G_V \}.
\end{equation}
Elements of $\mathrm{DDiff}(M)$ are thus quadruples $(\Gamma_1, \Gamma_2, \phi^+, \phi^-)$ satisfying the above regularity, with the source and target maps, along with the composition operation (illustrated in Figure \ref{Fig: Compose}) inherited from the full groupoid. 
% The elements of $\mathrm{DDiff}(M)$ are diffeomorphisms discontinuous along a hypersurface, i.e. quadruples $(\Gamma_1, \Gamma_2, \phi^+, \phi^-)$ where $\Gamma_1, \Gamma_2 \in \mathrm{V}(M)$, and $\phi^{\pm}: D_{\Gamma_{1}}^{\pm}\rightarrow D_{\Gamma_{2}}^{\pm}$ are diffeomorphisms and $M = D_{\Gamma_{i}}^{+} \cup \Gamma_{i} \cup D_{\Gamma_{i}}^{-}$. 

% the source of $(\Gamma_1,\Gamma_2,\phi^+,\phi^-)$ is $\mathrm{src}(\Gamma_1,\Gamma_2,\phi^+,\phi^-) = \Gamma_1$ and the target of $(\Gamma_1,\Gamma_2,\phi^+,\phi^-)$ is $\mathrm{trg}(\Gamma_1,\Gamma_2,\phi^+,\phi^-) = \Gamma_2$. The composition in DDiff($M$) is given by the following (see Figure \ref{Fig: Compose}):
% \begin{equation}
%     (\Gamma_2,\Gamma_3,\psi^+,\psi^-)(\Gamma_1,\Gamma_2,\phi^+,\phi^-)=(\Gamma_1,\Gamma_3,\psi^+\phi^+,\psi^-\phi^-).
% \end{equation}

\begin{figure}
\centering
\includegraphics[width=0.8\textwidth]{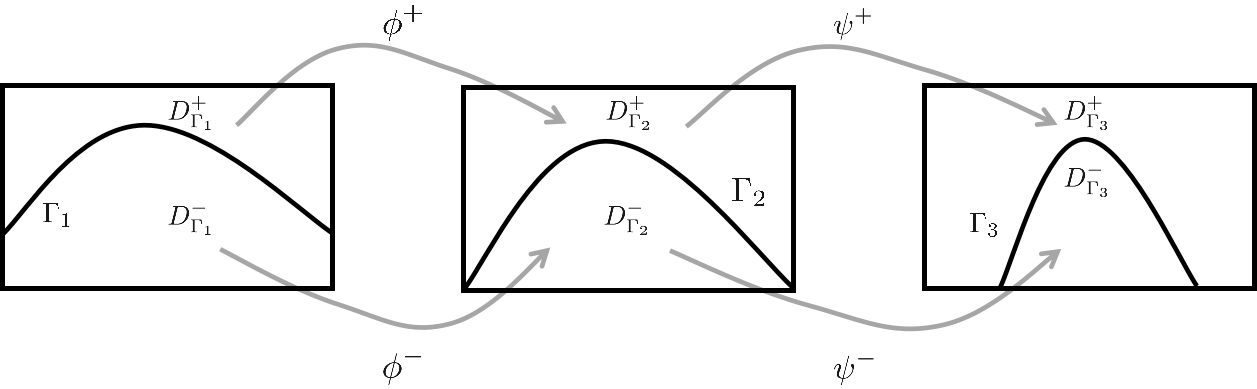}\nonumber
\caption{Illustration of the composition rule of two groupoid elements $(\Gamma_1, \Gamma_2, \phi^+, \phi^-)$ and $(\Gamma_2, \Gamma_3, \psi^+, \psi^-)$ in the Lie groupoid $\mathrm{DDiff}(M)$, resulting in $(\Gamma_1, \Gamma_3, \psi^+ \phi^+, \psi^- \phi^-)$.}
\label{Fig: Compose}
\end{figure}

\begin{definition}
    The source fiber of $\mathrm{DDiff}(M)$ corresponding to $\Gamma \in \mathrm{V}(M)$ is the set $\mathrm{DDiff}(M)_{\Gamma} = \{\phi \in \mathrm{DDiff}(M) \mid \mathrm{src}(\phi) = \Gamma\}$.
\end{definition}

\begin{theorem}
    $\mathrm{DDiff}(M)\rightrightarrows \mathrm{V}(M)$ is a Lie groupoid.
\end{theorem}

\begin{proof}
    It is straightforward to verify that it satisfies each condition in the definition of a groupoid as given in Definition \ref{def: groupoid}, and thus it is a groupoid. To further prove that it is a Lie groupoid, refer to page 476 of \cite{izosimov2018vortex}.
\end{proof}

Figure \ref{fig: vortex_sheet_full_groupoid} illustrates the structure of the Lie groupoid $\mathrm{DDiff}(M) \rightrightarrows \mathrm{V}(M)$. In this diagram, the Lie groupoid $\mathrm{DDiff}(M) \rightrightarrows \mathrm{V}(M)$ is represented as a cube. 
The vertical projection represents the source map $\mathrm{src}: \mathrm{DDiff}(M) \rightarrow \mathrm{V}(M)$, and the horizontal projection represents the target map $\mathrm{trg}: \mathrm{DDiff}(M) \rightarrow \mathrm{V}(M)$. Each element $\Gamma$ in $\mathrm{V}(M)$ corresponds to an identity element $\mathrm{id}_{\Gamma} = (\Gamma, \Gamma, \mathrm{id})$ in $\mathrm{DDiff}(M)$. 
The diagonal where $\mathrm{src} = \mathrm{trg}$ consists of all identity elements in $\mathrm{DDiff}(M)$. The source fiber $\mathrm{DDiff}(M)_{\Gamma}$ can be visualized as a vertical plane passing through $\Gamma$, indicating that the source map value at every point on this plane is $\Gamma$.

\begin{figure}[!htbp]
    \centering
    \includegraphics[width=0.6\textwidth]{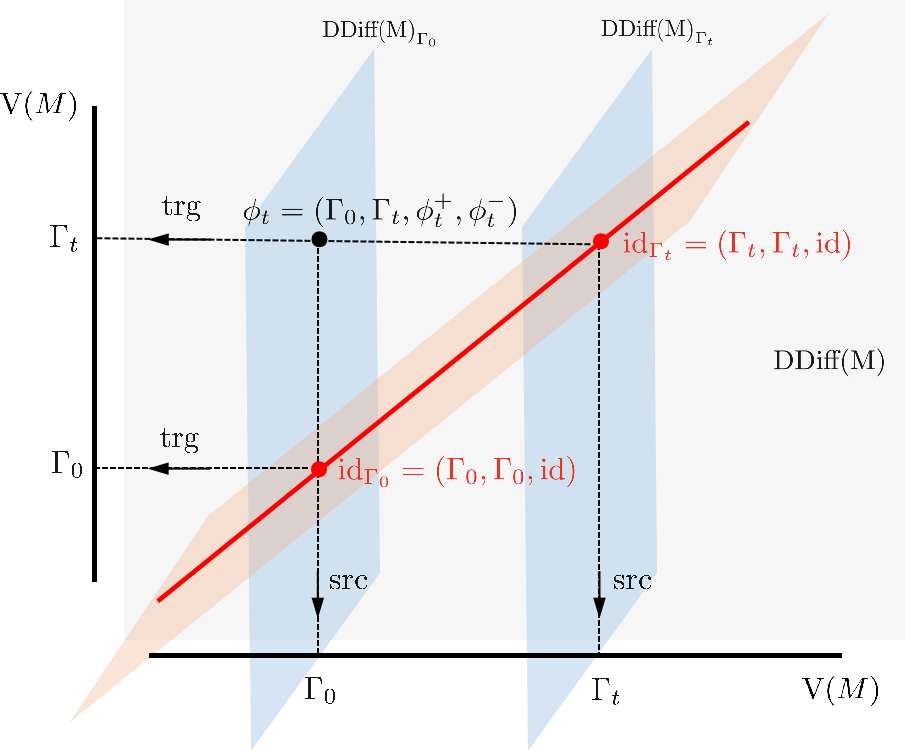}
    \caption{The structure of the Lie groupoid $\mathrm{DDiff}(M) \rightrightarrows \mathrm{V}(M)$.}
    \label{fig: vortex_sheet_full_groupoid}
\end{figure}

\subsection{The Lie algebroid of discontinuous vector fields}
\label{subsec: The Lie algebroid of discontinuous vector fields}

In this subsection, we introduce the Lie algebroid $\mathrm{DVect}(M) \rightarrow \mathrm{V}(M)$ corresponding to the Lie groupoid $\mathrm{DDiff}(M) \rightrightarrows \mathrm{V}(M)$. 

For the sake of convenience, we first define several spaces of tensor fields that are discontinuous along the vortex sheet $\Gamma$. 
All such tensor fields are assumed to have the form $\chi^+_\Gamma \xi^+ + \chi^-_\Gamma \xi^-$, where $\chi^\pm_\Gamma$ are the characteristic functions of $D^\pm_\Gamma$, and $\xi^\pm$ are smooth within $D^\pm_\Gamma$ up to the boundary $\Gamma$. The following spaces are defined:
\begin{align*}
    \mathrm{DC}^{\infty}(M,\Gamma) &= \{\chi_{\Gamma}^{+}f^{+}+\chi_{\Gamma}^{-}f^- | f^{\pm}\in \mathrm{C}^{\infty}(D_{\Gamma}^{\pm})\} &\hspace{0.2cm} &(\text{discontinuous functions})\\
    % \text{DVect}(M,\Gamma) &= \{\chi_{\Gamma}^{+}u^{+}+\chi_{\Gamma}^{-}u^- | u^{\pm}\in \text{Vect}(D_{\Gamma}^{\pm}), u^+|_{\Gamma} - u^-|_{\Gamma} \mathop{//} \Gamma\} &\hspace{0.2cm} &(\text{discontinuous vector fields with continuous normal component}) \\
    \mathrm{DVect}^{'}(M,\Gamma) &= \{\chi_{\Gamma}^{+}v^{+}+\chi_{\Gamma}^{-}v^- | v^{\pm}\in \mathrm{Vect}(D_{\Gamma}^{\pm})\}  &\hspace{0.2cm} &(\text{discontinuous vector fields}) \\
    \mathrm{D}\Omega^1(M,\Gamma) &= \{\chi_{\Gamma}^{+}m^{+}+\chi_{\Gamma}^{-}m^- | m^{\pm}\in \Omega^1(D_{\Gamma}^{\pm})\}  &\hspace{0.2cm} &(\text{discontinuous 1-forms}) \\
    \mathrm{DDense}(M,\Gamma) &=  \{\chi_{\Gamma}^{+}\mu^{+}+\chi_{\Gamma}^{-}\mu^- | \mu^{\pm}\in \mathrm{Dense}(D_{\Gamma}^{\pm})\}  &\hspace{0.2cm} &(\text{discontinuous densities}) 
\end{align*}
Here, $\mathrm{Dense}(D_{\Gamma}^{\pm})$ denotes the space of smooth densities on $D_{\Gamma}^{\pm}$. For the space of discontinuous vector fields $\mathrm{DVect}'(M, \Gamma)$, there is a special subspace of discontinuous vector fields $\mathrm{DVect}(M, \Gamma)$, where the normal components of the vector fields $u^+$ and $u^-$ are continuous across the boundary $\Gamma$:
\begin{equation*}
    \mathrm{DVect}(M, \Gamma) := \{\chi_{\Gamma}^{+}u^{+}+\chi_{\Gamma}^{-}u^- \mid u^{\pm}\in \mathrm{Vect}(D_{\Gamma}^{\pm}), u^+|_{\Gamma} - u^-|_{\Gamma} \text{ is tangent to } \Gamma\}.
\end{equation*}

Then we recall the regularized version $\mathcal{D}^{\mathcal{R}}$ of the differential operator $\mathcal{D}$ on discontinuous tensor fields $\xi$, as introduced in \cite{izosimov2018vortex}: $\mathcal{D}^{\mathcal{R}}\xi = \chi^+_\Gamma \mathcal{D}\xi^+ + \chi^-_\Gamma \mathcal{D}\xi^-$, which allows $\mathcal{D}$ to operate on discontinuous tensor fields by applying the differential operator on each side of the discontinuity separately and then combining the results accordingly.

For a function $f= \chi_{\Gamma}^{+}f^{+}+\chi_{\Gamma}^{-}f^- \in \mathrm{DC}^{\infty}(M, \Gamma) $, the jump across the hypersurface $\Gamma$ is defined as $\text{jump}(f) = f^+|_{\Gamma} - f^-|_{\Gamma}$. 
Below, we present two fundamental lemmas concerning operations on discontinuous functions, adapted from \cite{izosimov2018vortex}, which are essential to the subsequent proof.
The first lemma is a key result from \cite[Lemma 5.2]{izosimov2018vortex}, stated as follows:

% \textcolor{red}{definition of regularized operator}

\begin{lemma}[\cite{izosimov2018vortex}]
    Let $\Gamma_t \in \mathrm{VS}(M)$ be a family of hypersurfaces parametrized by $t\in \mathbb{R}$, $f_t \in \mathrm{DC}^{\infty}(M, \Gamma_t)$ be a smooth family of functions discontinuous across $\Gamma_t$, then
    \begin{equation}
        \frac{d}{dt}\int_{M} f_t\mu = \int_{M} \frac{d^{\mathcal{R}} f_t}{dt}\mu + \int_{\Gamma_t}\text{jump}(f_t)\frac{d\Gamma_t}{dt}.
        \label{eq: discontinuous function evolution}
    \end{equation}
\end{lemma}
% \begin{proof}
%     It follows from
%     \begin{equation}
%         \frac{d}{dt} \int_{D^\pm_{\Gamma_t}} f_t \mu = \int_{D^\pm_{\Gamma_t}} \frac{d f_{t}^{\pm}}{dt} \mu + \int_{\partial D^\pm_{\Gamma_t}} f_{t}^{\pm} \frac{d (\partial D^\pm_{\Gamma_t})}{dt}, 
%     \end{equation} 
%    where $\partial D^\pm_{\Gamma_t} = \pm \Gamma_{t}$.
% \end{proof}

This result, originally formulated for $\Gamma_t \in \mathrm{VS}(M)$, 
extends to $\Gamma_t \in \mathrm{V}(M)$.%, since $\mathrm{VS}(M)\subset \mathrm{V}(M)$.  
The second lemma is a slight modification of \cite[Lemma 5.4]{izosimov2018vortex}, as we no longer assume the divergence-free condition $\dive v^{\pm}=0$.
\begin{lemma}
\label{lemma: jump}
Let $f\in \mathrm{DC}^{\infty}(M, \Gamma)$, and let $v \in \mathrm{DVect}(M,\Gamma)$. Then 
    \begin{equation}
        \int_{M} \mathcal{L}_v^{\mathcal{R}} (f \mu) =  \int_{M}(\mathcal{L}_v^{\mathcal{R}} f)\mu + \int_{M}\dive^{\mathcal{R}} (v)f \mu = \int_{\Gamma}  \text{jump}(f) i_{v}\mu.
        \label{eq: jump equation}
    \end{equation} 
\end{lemma}

\begin{proof}
    \begin{equation*}
        \int_{M} \mathcal{L}_v^{\mathcal{R}} (f \mu) 
        = \int_{D_{\Gamma}^{+}} \mathcal{L}_v (f \mu) + \int_{D_{\Gamma}^{-}} \mathcal{L}_v (f \mu) 
    \end{equation*}

    \begin{align*}
            \int_{D_{\Gamma}^{+}} \mathcal{L}_v (f \mu) 
            &= \int_{D_{\Gamma}^{+}} (\mathcal{L}_v f)\mu + \dive (v)f \mu 
            = \int_{D_{\Gamma}^{+}}i_v df\mu + fd i_v\mu \nonumber\\
            & = \int_{D_{\Gamma}^{+}} d(fi_v \mu) 
             = \int_{\partial D_{\Gamma}^{+}}fi_v \mu
    \end{align*}
Since for $\alpha \in \Omega ^1(M), \mu \in \Omega ^n(M), v \in \mathrm{Vect}(M)$, we have \cite{arnold2014topological}
\begin{equation*}
    \alpha \wedge i_v \mu = \left\langle \alpha, v \right\rangle \mu = i_v \alpha \mu,
    d(f w)= fd(w) + d(f)\wedge w.
\end{equation*}
Then the third equation arises from
\begin{equation*}
    \int_{D_{\Gamma}^{+}} d(fi_v \mu) =  \int_{D_{\Gamma}^{+}} f d i_v\mu + df\wedge i_v\mu =\int_{D_{\Gamma}^{+}}  f d i_v\mu + i_vdf \mu
\end{equation*}    
and the last equation comes from the Stokes formula.

Analogously we have
\begin{equation*}
    \int_{D_{\Gamma}^{-}} \mathcal{L}_v (f \mu)= \int_{\partial D_{\Gamma}^{-}}fi_v \mu
\end{equation*}
where $\partial D_{\Gamma}^{+} = \Gamma, \partial D_{\Gamma}^{-} = -\Gamma$, so finally we get
\begin{equation*}
    \int_{M} \mathcal{L}_v^{\mathcal{R}} (f \mu) = \int_{\Gamma} \text{jump}(f) i_v\mu.
\end{equation*}
\end{proof}

Now we can get the structure of the Lie algebroid $\mathrm{DVect}(M) \rightarrow \mathrm{V}(M)$ corresponding to the Lie groupoid $\mathrm{DDiff}(M) \rightrightarrows \mathrm{V}(M)$, adapted from \cite[Theorem 6.5]{izosimov2018vortex}. 

\begin{theorem}
    The Lie algebroid associated with the Lie groupoid $\mathrm{DDiff}(M)$, denoted as $\mathrm{DVect}(M)$, defines the structure of discontinuous vector fields on the manifold $M$. Below is the detailed structure of $\mathrm{DVect}(M)$:
% The above Lie algebroid DVect($M$) has the following structure:
    \begin{enumerate}
        \item The fiber of $\mathrm{DVect}(M)$ over $\Gamma$ is the space $\mathrm{DVect}(M,\Gamma)$ which consists of discontinuous vector fields of the form  $u=\chi_{\Gamma}^{+}u^{+}+\chi_{\Gamma}^{-}u^-,$ where $\chi_{\Gamma}^{\pm}$ are the indicator functions of the domain $D_{\Gamma}^{\pm}$, $u^{\pm}$ are smooth vector fields on $D_{\Gamma}^{\pm}$, 
        the normal component of $u$ on $\Gamma$ is continuous, i.e. $(u^+|_{\Gamma} - u^-|_{\Gamma}) \mathop{//} \Gamma$.

        \item The anchor map $\#:\mathrm{DVect}(M,\Gamma)\rightarrow T_{\Gamma}\mathrm{V}(M)$ is given by the projection of $u^+|_{\Gamma}$ or $u^-|_{\Gamma}$ to their normal components, i.e. $\# u =\partial_t \Gamma$.

        \item Let $\zeta, \eta$ be sections of $\mathrm{DVect}(M)$, the algebroid bracket is defined as:
        \begin{equation}
            [\zeta, \eta](\Gamma) =[\zeta(\Gamma), \eta(\Gamma)]^{\mathcal{R}} +\mathcal{L}_{\#\zeta(\Gamma)}\eta-\mathcal{L}_{\#\eta(\Gamma)}\zeta, 
            \label{eq: algebroid bracket}
        \end{equation} 
        where $[u,v]^{\mathcal{R}}$ is the regularized Lie bracket:
        \begin{equation*}
            [u,v]^{\mathcal{R}} = [\chi_{\Gamma}^{+}u^{+}+\chi_{\Gamma}^{-}u^-, \chi_{\Gamma}^{+}v^{+}+\chi_{\Gamma}^{-}v^-]^{\mathcal{R}} = \chi_{\Gamma}^{+}[u^+,v^+] + \chi_{\Gamma}^{-}[u^-,v^-],
        \end{equation*}
        The derivative $\mathcal{L}_{\#\zeta(\Gamma)}\eta$ stands for the regularized derivative of $\eta$ in the direction $\#\zeta(\Gamma_0)$:
        \begin{equation*}
            \mathcal{L}_{\#\zeta(\Gamma)}\eta = \frac{d^{\mathcal{R}}}{dt}\bigg|_{t=0}\eta(\Gamma_t),
        \end{equation*}
        $\Gamma_t$ is any smooth curve with $\Gamma_0 = \Gamma$ and the tangent vector at $\Gamma_0$ given by $\#\zeta(\Gamma_0)$.  
    \end{enumerate}
\end{theorem}
\begin{proof}
    The proof follows a similar structure to Theorem 6.5 in \cite{izosimov2018vortex}.
\end{proof}

\begin{remark}
    The fibers of $\mathrm{DVect}(M)$ over $\Gamma$ are not closed under the Lie bracket of vector fields. Specifically, the regularized Lie bracket $[\zeta(\Gamma),\eta(\Gamma)]^{\mathcal{R}}$ does not belong to $\mathrm{DVect}(M,\Gamma)$, even though both $\zeta(\Gamma)$ and $\eta(\Gamma)$ have continuous normal components. 
    % but their regularized Lie bracket doesn't have this property. 
    However, the terms $\mathcal{L}_{\#\zeta(\Gamma)}\eta$ and $\mathcal{L}_{\#\eta(\Gamma)}\zeta$, despite their discontinuous normal components, can rectify the discontinuity in the first term $[\zeta(\Gamma),\eta(\Gamma)]^{\mathcal{R}}$. This compensatory ensures that the algebroid bracket $[\zeta,\eta](\Gamma)$ is indeed an element of $\mathrm{DVect}(M,\Gamma)$.
\end{remark}

\subsection{The tangent space to the algebroid and its dual}
\label{subsec: The tangent space to the algebroid of discontinuous fields and its dual}
In this subsection, we describe the tangent and dual spaces of the Lie algebroid $\mathrm{DVect}(M)$. 
The construction of the tangent space closely follows that of \cite{izosimov2018vortex}, as it represents the space of possible accelerations for discontinuous velocity fields. However, the dual space differs due to the absence of the divergence-free condition, leading to a modified formulation.
% Since $\mathrm{DVect}(M)$ represents the possible velocity space of a fluid with vortex sheets, its tangent space can be thought of as the space of possible accelerations. 

This distinction in the dual space also affects the subsequent structures. In particular, the Poisson bracket on the dual algebroid and the Euler-Arnold equation for Lie groupoids, which are introduced in the following subsections, differ from those in \cite{izosimov2018vortex}, as they inherently involve the modified dual space. 
The specific definitions are provided below.

\begin{definition}[Tangent space to the Lie algebroid]
    Let $u_t = \chi_{\Gamma_t}^{+} u_t^{+} + \chi_{\Gamma_t}^{-} u_t^- \in \mathrm{DVect}(M)$ be a smooth curve in the space $\mathrm{DVect}(M)$. Then the tangent vector to $u_t$ at $t = t_0$ is a pair $(v, \xi) \in \mathrm{DVect}'(M, \Gamma_{t_0}) \oplus T_{\Gamma_{t_0}} \mathrm{V}(M)$, defined by
    \begin{equation}
        v = \left.\frac{d^{\mathcal{R}}}{dt}\right|_{t=t_0} u_t, \quad \xi = \left.\frac{d}{dt}\right|_{t=t_0} \Gamma_t,
    \end{equation}
    where the space
    \begin{equation*}
        \mathrm{DVect}'(M, \Gamma) = \left\{\chi_{\Gamma}^{+} u^{+} + \chi_{\Gamma}^{-} u^- \mid u^{\pm} \in \mathrm{Vect}(D_{\Gamma}^{\pm})\right\}
    \end{equation*}
    is the space of vector fields that are discontinuous at $\Gamma$, without the requirement for the normal components of the vector fields to be continuous at $\Gamma$.
\end{definition}

Next, we consider the smooth dual space of the Lie algebroid $\mathrm{DVect}(M)$, which can be viewed as the momentum space of a fluid with vortex sheets.
Let
\begin{align}
    \mathrm{D}\Omega^1(M) &= \bigcup_{\Gamma \in \mathrm{V}(M)} \mathrm{D}\Omega^1(M, \Gamma), \\
    \mathrm{DDense}(M) &= \bigcup_{\Gamma \in \mathrm{V}(M)} \mathrm{DDense}(M, \Gamma).
\end{align}

As the dual of $\mathrm{DVect}(M)$, we define the dual vector bundle as follows:
\begin{definition}[Dual of the Lie algebroid]
    The smooth dual vector bundle $\mathrm{DVect}(M)^*$ is the space
    \begin{equation}
        \mathrm{DVect}(M)^* := \mathrm{D}\Omega^1(M) \otimes \mathrm{DDense}(M),
        % \text{or we can also write it as } \text{D}\Omega^1(M) \otimes \text{D}\Omega^n(M), \text{where } n = \text{dim } M.
     \label{eq: splitting dual algebroid}
    \end{equation}
    where $\otimes$ denotes the tensor product.
\end{definition}

The pairing between a 1-form density $\tilde{m} \in \mathrm{DVect}(M, \Gamma)^*= \mathrm{D}\Omega^1(M, \Gamma) \otimes \mathrm{DDense}(M, \Gamma) $ and a vector field $u \in \mathrm{DVect}(M, \Gamma)$ is given by the following formula \cite{khesin2008geometry, khesin2021geometric}:
\begin{equation*}
    \left\langle \tilde{m}, u \right\rangle = \left\langle m \otimes \mu, u \right\rangle := \int_{M} (i_u m) \mu,
\end{equation*}
where $\mu$ is the volume form and $m$ is a 1-form.

Next, we describe the tangent spaces of this dual.
\begin{definition}[Tangent vector of a discontinuous 1-form]
    Let $m_t = \chi_{\Gamma_t}^{+} m_t^{+} + \chi_{\Gamma_t}^{-} m_t^- \in \mathrm{D}\Omega^1(M)$ be a smooth curve of 1-forms defined on the manifold $M$ and discontinuous along the hypersurface $\Gamma_t$. Then, at time $t = t_0$, the tangent vector to $m_t$ is a pair consisting of the following two components:
    \begin{equation}
        \frac{d^{\mathcal{R}}}{dt}\bigg|_{t=0} m_t \in \mathrm{D}\Omega^1(M, \Gamma_{t_0}), \quad \frac{d}{dt}\bigg|_{t=0} \Gamma_t \in T_{\Gamma_{t_0}} \mathrm{V}(M).
    \end{equation}
\end{definition}

This structure allows us to describe the tangent space of $\mathrm{DVect}(M)^*$.
Given any $\tilde{m} \in \mathrm{DVect}(M, \Gamma)^*$, its tangent space has a splitting
    \begin{equation}
        T_{\tilde{m}} \mathrm{DVect}(M)^* \simeq \mathrm{DVect}(M, \Gamma)^* \oplus T_{\Gamma} \mathrm{V}(M).
        \label{eq: splitting tangent of dual algebroid}
    \end{equation}
To see this, let $\tilde{m}_t$ be any smooth curve in $\mathrm{DVect}(M)^*$ corresponding to the curve $\Gamma_t \in \mathrm{V}(M)$, where $\Gamma_0 = \Gamma$ and $\tilde{m}_0 = m \otimes \mu$. Then, for this curve $\tilde{m}_t$, we can associate it with the following pair:
    \begin{equation*}
        \frac{d^{\mathcal{R}}}{dt}\bigg|_{t=0} \tilde{m}_t 
        % = \Big(\frac{d^{\mathcal{R}}}{dt}\bigg|_{t=0} m_t\Big)\otimes \mu
        \in \mathrm{DVect}(M, \Gamma)^*, \quad \frac{d}{dt}\bigg|_{t=0} \Gamma_t \in T_{\Gamma} \mathrm{V}(M).
    \end{equation*}
Then this correspondence establishes an isomorphism between $T_{\tilde{m}} \mathrm{DVect}(M)^*$ and $\mathrm{DVect}(M, \Gamma)^* \oplus T_{\Gamma} \mathrm{V}(M)$.
    
% \begin{corollary}
%     For any $\tilde{m} \in \mathrm{DVect}(M, \Gamma)^*$, its tangent space has a splitting
%     \begin{equation}
%         T_{\tilde{m}} \mathrm{DVect}(M)^* \simeq \mathrm{DVect}(M, \Gamma)^* \oplus T_{\Gamma} \mathrm{V}(M).
%         \label{eq: splitting tangent of dual algebroid}
%     \end{equation}
%     % This means that for any smooth curve $\tilde{m}_t$ in $\mathrm{DVect}(M)^*$, we can associate it with the following pair:
%     % Thus, this curve can be associated with a pair
%     % \begin{equation*}
%     %      \frac{d^{\mathcal{R}}}{dt}\bigg|_{t=0} \tilde{m}_t 
%     %      % = \Big(\frac{d^{\mathcal{R}}}{dt}\bigg|_{t=0} m_t\Big)\otimes \mu
%     %      \in \mathrm{DVect}(M, \Gamma_t)^*, \frac{d}{dt}\bigg|_{t=0} \Gamma_t \in T_{\Gamma}\mathrm{V}(M).
%     % \end{equation*}
% \end{corollary}
% \begin{proof}
%     Let $\tilde{m}_t$ be any smooth curve in $\mathrm{DVect}(M)^*$ corresponding to the curve $\Gamma_t \in \mathrm{V}(M)$, where $\Gamma_0 = \Gamma$ and $\tilde{m}_0 = m \otimes \mu$. Then, for this curve $\tilde{m}_t$, we can associate it with the following pair:
%     \begin{equation*}
%         \frac{d^{\mathcal{R}}}{dt}\bigg|_{t=0} \tilde{m}_t 
%         % = \Big(\frac{d^{\mathcal{R}}}{dt}\bigg|_{t=0} m_t\Big)\otimes \mu
%         \in \mathrm{DVect}(M, \Gamma)^*, \quad \frac{d}{dt}\bigg|_{t=0} \Gamma_t \in T_{\Gamma} \mathrm{V}(M).
%     \end{equation*}
%     Therefore, through this correspondence, we obtain an isomorphism between $T_{\tilde{m}} \mathrm{DVect}(M)^*$ and $\mathrm{DVect}(M, \Gamma)^* \oplus T_{\Gamma} \mathrm{V}(M)$.
% \end{proof}

\subsection{Poisson bracket on the dual algebroid}
\label{subsec: Poisson bracket on the dual algebroid}

First, we introduce the cotangent space to the base space $\mathrm{V}(M)$, denoted as $T_{\Gamma}^*\mathrm{V}(M)$, and then describe the cotangent space to the dual Lie algebroid $\mathrm{DVect}(M)^*$, denoted as $T_{\tilde{m}}^*\mathrm{DVect}(M)^*$.

\begin{definition}[Cotangent space to the base space]
    The cotangent space $T_{\Gamma}^*\mathrm{V}(M)$ is the function space $\mathrm{C}^{\infty}(\Gamma)\otimes \mathrm{Dense}(\Gamma)$, where these functions are defined on the hypersurface $\Gamma$ and take values in the density space. The pairing between $\tilde{f} \in \mathrm{C}^{\infty}(\Gamma)\otimes \mathrm{Dense}(\Gamma)$ and the highest-degree form $\xi \in T_{\Gamma}\mathrm{V}(M) = \Omega^{n-1}(\Gamma)$ (where $n= \mathrm{dim}~ M$) is given by
    \begin{equation*}
        \left\langle \tilde{f}, \xi \right\rangle = \int_{\Gamma} f \xi.
    \end{equation*}
\end{definition}

Now, by dualizing the splitting \eqref{eq: splitting tangent of dual algebroid}, we obtain the cotangent space to the dual Lie algebroid $\text{DVect}(M)^*$.

\begin{definition}[Cotangent space to the dual algebroid]
    Let $\tilde{m} \in \mathrm{DVect}(M, \Gamma)^*$. Then the smooth cotangent space of $\mathrm{DVect}(M)^*$ at $\tilde{m}$ is given by
    \begin{equation}
        T_{\tilde{m}}^* \mathrm{DVect}(M)^* = \mathrm{DVect}(M, \Gamma) \oplus T_{\Gamma}^* \mathrm{V}(M).
        \label{eq: splitting cotangent of dual algebroid}
    \end{equation}
\end{definition}

Furthermore, we define a differentiable function on the dual Lie algebroid $\mathrm{DVect}(M)^*$. A function is differentiable if it has a differential belonging to the cotangent space, as follows:

\begin{definition}[Differentiable function on the dual Lie algebroid]
    A function $\mathcal{F}: \mathrm{DVect}(M)^* \rightarrow \mathbb{R}$ is differentiable if there exists a section $d\mathcal{F}$ of the smooth cotangent bundle $T^* \mathrm{DVect}(M)^*$ such that for any smooth curve $\tilde{m}_t$ in $\mathrm{DVect}(M)^*$, the following holds:
    \begin{equation}
        \frac{d}{dt}\mathcal{F}(\tilde{m}_t) = \left\langle d\mathcal{F}(\tilde{m}_t), (\partial_t^{R}\tilde{m}_t, \partial_t \Gamma_t) \right\rangle .
       \label{eq: function differential1} 
    \end{equation}
\end{definition}

Based on equation \eqref{eq: splitting cotangent of dual algebroid}, the differential $d\mathcal{F}(\tilde{m})$ for $\tilde{m} \in \mathrm{DVect}(M, \Gamma)^*$ can be decomposed as follows:
\begin{equation*}
    d\mathcal{F}(\tilde{m}) = (d^F \mathcal{F}(\tilde{m}), d^B \mathcal{F}(\tilde{m})),
\end{equation*}
where 
\begin{equation*}
    d^F \mathcal{F}(\tilde{m}) \in \mathrm{DVect}(M, \Gamma), \quad d^B \mathcal{F}(\tilde{m}) \in T_{\Gamma}^* \mathrm{V}(M) = \mathrm{C}^{\infty}(\Gamma) \otimes \mathrm{Dense}(\Gamma).
\end{equation*}
Thus, equation \eqref{eq: function differential1} can be expressed as
\begin{equation}
    \frac{d}{dt}\mathcal{F}(\tilde{m}_t) = \left\langle d^F \mathcal{F}(\tilde{m}_t), \partial_t^{\mathcal{R}} \tilde{m}_t \right\rangle + \left\langle d^B \mathcal{F}(\tilde{m}_t), \partial_t \Gamma_t \right\rangle.
    \label{eq: function differential}
\end{equation}

We adapt and generalize the Poisson bracket formulation from \cite[Theorem 6.25]{izosimov2018vortex} to accommodate the compressible setting, retaining the structure of the original proof, as detailed in the following theorem.

\begin{theorem}[Poisson bracket on the dual algebroid]
    Let $\mathcal{F}_1, \mathcal{F}_2: \mathrm{DVect}(M)^* \rightarrow \mathbb{R}$ be differentiable functions, and assume that $d\mathcal{F}_i(\tilde{m}) = (d^F \mathcal{F}_i(\tilde{m}), d^B \mathcal{F}_i(\tilde{m})) = (v_i, \tilde{n}_i)$. Then their Poisson bracket is given by
    \begin{align}
        &\{\mathcal{F}_1, \mathcal{F}_2\}(\tilde{m}) = \mathcal{P}_{\tilde{m}}(d\mathcal{F}_1, d\mathcal{F}_2) \nonumber\\
        &= \mathcal{P}_{\tilde{m}}((v_1, \tilde{n}_1), (v_2, \tilde{n}_2)) \nonumber\\
        &= \int_{M} m([v_1, v_2]^{\mathcal{R}})\mu + \int_{\Gamma} \bigg( n_2 - \text{jump}(i_{v_2} m) \bigg) i_{v_1} \mu - \int_{\Gamma} \bigg( n_1 - \text{jump}(i_{v_1} m) \bigg) i_{v_2} \mu. \label{eq:PoissonBracket}
    \end{align}
\end{theorem}

\begin{proof}
    Suppose section $A$ of $\text{DVect}(M)^*$ is an arbitrary extension of $\tilde{m}$ such that $A(\Gamma) = \tilde{m}$, $U_i$ is a section of $\text{DVect}(M)$ given by $U_i = d^F \mathcal{F}_i(A)$ such that $U_i(\Gamma) = d^F \mathcal{F}_i(A(\Gamma)) = d^F \mathcal{F}_i(\tilde{m})$.
According to the definition of the Poisson bracket on the dual of a Lie algebroid \cite{boucetta2011riemannian,izosimov2018vortex} and the algebroid bracket in equation \eqref{eq: algebroid bracket}, we can write 
\begin{align}
        &\{\mathcal{F}_1, \mathcal{F}_2\}(A(\Gamma))
        =\{\mathcal{F}_1, \mathcal{F}_2\}(\tilde{m})\nonumber\\
        % &= \mathcal{P}(d\mathcal{F}_1, d\mathcal{F}_2)(A(\Gamma)) \nonumber\\
        % &= \mathcal{P}(d\mathcal{F}_1, d\mathcal{F}_2)(\tilde{m}) \nonumber\\
        % &= \mathcal{P}_{\tilde{m}}((v_1, \tilde{n}_1), (v_2, \tilde{n}_2)) \nonumber\\
        &=  \left\langle \tilde{m}, [U_1, U_2](\Gamma) \right\rangle  + \#U_1(\Gamma)\cdot (\mathcal{F}_2(A)- \left\langle A, U_2 \right\rangle )-\#U_2(\Gamma)\cdot (\mathcal{F}_1(A)- \left\langle A, U_1 \right\rangle )\nonumber \\
        &=  \left\langle \tilde{m}, [U_1(\Gamma), U_2(\Gamma)]^{\mathcal{R}}+\#U_1(\Gamma)\cdot U_2 - \#U_2(\Gamma)\cdot U_1  \right\rangle  + \#U_1(\Gamma)\cdot (\mathcal{F}_2(A)- \left\langle A, U_2 \right\rangle ) \nonumber\\
        &\quad \quad -\#U_2(\Gamma)\cdot (\mathcal{F}_1(A)- \left\langle A, U_1 \right\rangle )\nonumber \\
        &=  \left\langle \tilde{m}, [U_1(\Gamma), U_2(\Gamma)]^{\mathcal{R}} \right\rangle + \left\langle \tilde{m}, \#U_1(\Gamma)\cdot U_2 \right\rangle  -  \left\langle \tilde{m}, \#U_2(\Gamma)\cdot U_1 \right\rangle   \nonumber\\
        &\quad\quad + \#U_1(\Gamma)\cdot (\mathcal{F}_2(A)- \left\langle A, U_2 \right\rangle )-\#U_2(\Gamma)\cdot (\mathcal{F}_1(A)- \left\langle A, U_1 \right\rangle )\nonumber \\
        &=  \left\langle \tilde{m}, [U_1(\Gamma), U_2(\Gamma)]^{\mathcal{R}} \right\rangle  + \#U_1(\Gamma)\cdot (\mathcal{F}_2(A))- \#U_1(\Gamma)\cdot \left\langle A, U_2 \right\rangle + \left\langle \tilde{m}, \#U_1(\Gamma)\cdot U_2 \right\rangle  \nonumber\\
        &\quad\quad - (\#U_2(\Gamma)\cdot (\mathcal{F}_1(A))- \#U_2(\Gamma)\cdot \left\langle A, U_1 \right\rangle + \left\langle \tilde{m}, \#U_2(\Gamma)\cdot U_1 \right\rangle) \nonumber\\
        &= \left\langle \tilde{m}, [U_1(\Gamma), U_2(\Gamma)]^{\mathcal{R}} \right\rangle + S_{12}-S_{21} 
    \end{align}
where
\begin{align}
    S_{12} &= \#U_1(\Gamma)\cdot (\mathcal{F}_2(A))- \#U_1(\Gamma)\cdot \left\langle A,U_2 \right\rangle + \left\langle \tilde{m},\#U_1(\Gamma)\cdot U_2 \right\rangle, \label{eq: S12}\\
    S_{21} &= \#U_2(\Gamma)\cdot (\mathcal{F}_1(A))- \#U_2(\Gamma)\cdot \left\langle A,U_1 \right\rangle + \left\langle \tilde{m},\#U_2(\Gamma)\cdot U_1 \right\rangle, \label{eq: S21} 
\end{align}

Now let's first compute $S_{12}$.
Take any curve $\Gamma_1(t)\in \text{V}(M)$ such that $\Gamma_1(0) = \Gamma$, the tangent vector to $\Gamma_1(t)$ at $\Gamma$ is $\#U_1(\Gamma)$, then
\begin{align}
    &\#U_1(\Gamma)\cdot (\mathcal{F}_2(A))= \nabla_{\#U_1(\Gamma)}\mathcal{F}_2(A)\nonumber\\
    &= \frac{d}{dt}\bigg|_{t=0}  \mathcal{F}_2(A(\Gamma_1(t))) \nonumber\\
    &= \left\langle d\mathcal{F}_2(A(\Gamma_1(t))), (\partial_t^{R}A(\Gamma_1(t)), \partial_t \Gamma_1(t)) \right\rangle \bigg|_{t=0} \nonumber\\
    &= \left\langle (d^F \mathcal{F}_2(A(\Gamma_1(t))), d^B \mathcal{F}_2(A(\Gamma_1(t)))), (\partial_t^{R}A(\Gamma_1(t)), \partial_t \Gamma_1(t)) \right\rangle \bigg|_{t=0} \nonumber\\
    &= \left\langle  d^F \mathcal{F}_2(A(\Gamma_1(t))) , \partial_t^{R}A(\Gamma_1(t)) \right\rangle \bigg|_{t=0}+ \left\langle   d^B \mathcal{F}_2(A(\Gamma_1(t))), \partial_t \Gamma_1(t)) \right\rangle\bigg|_{t=0} \nonumber\\
    &= \left\langle U_2(\Gamma), \frac{d^{\mathcal{R}}}{dt}\bigg|_{t=0} A(\Gamma_1(t)) \right\rangle + \left\langle d^B \mathcal{F}_2(\tilde{m}),  \#U_1(\Gamma)\right\rangle \label{eq: 1}.
\end{align}
Further, let $m_1(t)$ be any curve in $\Omega^1(M)$ lifting $A(\Gamma_1(t))$ and such that $m_1(0) = m$. Then using \eqref{eq: discontinuous function evolution} we get
\begin{align}
    &\#U_1(\Gamma)\cdot \left\langle A,U_2 \right\rangle 
    = \nabla_{\#U_1(\Gamma)} \left\langle A,U_2 \right\rangle \nonumber\\
    &= \frac{d}{dt}\bigg|_{t=0} \left\langle A(\Gamma_1(t)),U_2(\Gamma_1(t)) \right\rangle \nonumber\\
    &= \frac{d}{dt}\bigg|_{t=0} \int_M i_{U_2(\Gamma_1(t))} m_1(t) \mu \nonumber\\
    &= \frac{d^{\mathcal{R}}}{dt}\bigg|_{t=0}\int_M i_{U_2(\Gamma_1(t))} m_1(t) \mu + \int_{\Gamma_t} \text{jump}(i_{U_2(\Gamma_1(t))} m_1(t)) \frac{d\Gamma_1(t)}{dt}\bigg|_{t=0}\nonumber\\
    &= \frac{d^{\mathcal{R}}}{dt}\bigg|_{t=0} \left\langle A(\Gamma_1(t)),U_2(\Gamma_1(t)) \right\rangle + \int_{\Gamma} \text{jump}(i_{U_2(\Gamma)} m) \#U_1(\Gamma) \nonumber\\
    &= \left\langle  \frac{d^{\mathcal{R}}}{dt}\bigg|_{t=0}A(\Gamma_1(t)) , U_2(\Gamma)  \right\rangle + \left\langle  \tilde{m} , \frac{d^{\mathcal{R}}}{dt}\bigg|_{t=0} U_2(\Gamma_1(t)) \right\rangle +\int_{\Gamma} \text{jump}(i_{U_2(\Gamma)} m) \#U_1(\Gamma) \nonumber\\
    &= \left\langle  \frac{d^{\mathcal{R}}}{dt}\bigg|_{t=0}A(\Gamma_1(t)) , U_2(\Gamma)  \right\rangle + \left\langle  \tilde{m} , \#U_1(\Gamma) \cdot U_2 \right\rangle + \int_{\Gamma} \text{jump}(i_{U_2(\Gamma)} m) \#U_1(\Gamma). \label{eq: 2}
\end{align}
Substituting \eqref{eq: 1} and \eqref{eq: 2} into \eqref{eq: S12}, we can get
\begin{equation}
    S_{12} = \left\langle d^B \mathcal{F}_2(\tilde{m}),  \#U_1(\Gamma)\right\rangle - \int_{\Gamma} \text{jump}(i_{U_2(\Gamma)} m) \#U_1(\Gamma),
\end{equation}
Analogously, we can get 
\begin{equation}
    S_{21} = \left\langle d^B \mathcal{F}_1(\tilde{m}),  \#U_2(\Gamma)\right\rangle - \int_{\Gamma} \text{jump}(i_{U_1(\Gamma)} m) \#U_2(\Gamma),
\end{equation}
Supposing $d\mathcal{F}_i(\tilde{m}) = (d^F \mathcal{F}_i(\tilde{m}), d^B \mathcal{F}_i(\tilde{m})) =(U_{i}(\Gamma),d^B \mathcal{F}_i(\tilde{m}))= (v_i, \tilde{n}_i)$, then
\begin{align}
            &\{\mathcal{F}_1, \mathcal{F}_2\}(A(\Gamma))
            =\{\mathcal{F}_1, \mathcal{F}_2\}(\tilde{m})
            = \mathcal{P}_{\tilde{m}}(d\mathcal{F}_1, d\mathcal{F}_2) \nonumber\\
            &= \mathcal{P}_{\tilde{m}}((v_1, \tilde{n}_1), (v_2, \tilde{n}_2)) \nonumber\\
            &= \left\langle \tilde{m}, [U_1(\Gamma), U_2(\Gamma)]^{\mathcal{R}} \right\rangle + S_{12}-S_{21} \nonumber\\
            &= \int_{M} m([v_1, v_2]^{\mathcal{R}})\mu + \left\langle \tilde{n}_2, \# v_1 \right\rangle  - \int_{\Gamma} \text{jump}(i_{v_2} m) \#v_1 \nonumber\\
            &\quad -\bigg(\left\langle \tilde{n}_1, \# v_2 \right\rangle  - \int_{\Gamma} \text{jump}(i_{v_1} m) \#v_2\bigg)\nonumber\\    
            &= \int_{M} m([v_1, v_2]^{\mathcal{R}})\mu + \int_{\Gamma} n_2 i_{v_1}\mu  - \int_{\Gamma} \text{jump}(i_{v_2} m)  i_{v_1}\mu \nonumber\\
            &\quad
            -\bigg( \int_{\Gamma} n_1 i_{v_2}\mu  - \int_{\Gamma} \text{jump}(i_{v_1} m) i_{v_2}\mu \bigg) \nonumber\\
            &= \int_{M} m([v_1, v_2]^{\mathcal{R}})\mu + \int_{\Gamma}\bigg( n_2- \text{jump}(i_{v_2} m)\bigg)i_{v_1}\mu  - \int_{\Gamma} \bigg( n_1  -  \text{jump}(i_{v_1} m) \bigg)i_{v_2}\mu .\nonumber 
    \end{align}
\end{proof}

\begin{remark}
    As for the Lie-Poisson bracket on the dual Lie algebra $\mathfrak{g}^*$, it is defined as following \cite{vizman2008geodesic}: 
    \begin{equation}\label{eq: Lie Poisson bracket}
        \{\mathcal{F}_1, \mathcal{F}_2\}(\tilde{m}) = 
        \left \langle\tilde{m}, [d\mathcal{F}_1, d \mathcal{F}_2]\right \rangle = \left \langle\tilde{m}, [v_1, v_2]\right \rangle = \int_{M} m([v_1, v_2])\mu, ~~~\mathcal{F}_1, \mathcal{F}_2\in C^{\infty}(\mathfrak{g}^*).
    \end{equation}
    By comparing equation \eqref{eq:PoissonBracket} and \eqref{eq: Lie Poisson bracket}, we can see that the Poisson bracket on the dual Lie algebroid \eqref{eq:PoissonBracket} is quite similar to the Lie-Poisson bracket on the dual Lie algebra \eqref{eq: Lie Poisson bracket} but with several boundary terms.
\end{remark}

\begin{corollary}
    Alternatively, the Poisson bracket can also be written in the following form:
    \begin{align}\label{eq: Lie Poisson bracket divergence}
       \mathcal{P}_{\tilde{m}}((v_1, \tilde{n}_1), (v_2, \tilde{n}_2)) &= -\int_{M}d^{\mathcal{R}} m(v_1, v_2)\mu -\int_{\Gamma} n_1 i_{v_2}\mu + \int_{\Gamma} n_2 i_{v_1}\mu \nonumber \\
       &\quad - \int_{M}\dive^{\mathcal{R}} (v_1)i_{v_2} m  \mu + \int_{M}\dive^{\mathcal{R}} (v_2)i_{v_1} m \mu.
    \end{align}    
\end{corollary}

\begin{proof}
    According to the Lemma \ref{lemma: jump} we have
    \begin{equation}
       \int_{\Gamma}  \text{jump}(f) i_{v}\mu = \int_{M}(\mathcal{L}_v^{\mathcal{R}} f)\mu + \int_{M}\dive^{\mathcal{R}} (v)f \mu = \int_{M} \mathcal{L}_v^{\mathcal{R}} (f \mu), 
    \end{equation} 
    then 

    \begin{align*}
        \int_{\Gamma} \text{jump}(i_{v_2} m) i_{v_1}\mu 
        &= \int_{M}(\mathcal{L}_{v_1}^{\mathcal{R}} i_{v_2} m)\mu + \int_{M}\dive^{\mathcal{R}} (v_1)i_{v_2} m \mu
        % = \textcolor{red}{\mathcal{L}_{v_1}^{\mathcal{R}} \left\langle \tilde{m}, v_2\right\rangle} + \int_{M}\dive^{\mathcal{R}} (v_1)i_{v_2} m \mu
        % = \# v_1 \cdot  \left\langle \tilde{m},v_2\right\rangle + \int_{M}\dive^{\mathcal{R}} (v_1)i_{v_2} m  \mu
        , \nonumber \\
        \int_{\Gamma} \text{jump}(i_{v_1} m) i_{v_2}\mu 
        &= \int_{M}(\mathcal{L}_{v_2}^{\mathcal{R}} i_{v_1} m)\mu + \int_{M}\dive^{\mathcal{R}} (v_2)i_{v_1} m \mu 
        % = \textcolor{red}{\mathcal{L}_{v_2}^{\mathcal{R}} \left\langle \tilde{m}, v_1\right\rangle} + \int_{M}\dive^{\mathcal{R}} (v_2)i_{v_1} m\mu
        % = \# v_2 \cdot  \left\langle \tilde{m},v_1\right\rangle + \int_{M}\dive^{\mathcal{R}} (v_2)i_{v_1} m \mu
        . \nonumber 
    \end{align*}
    By using the formula
    $i_{[v_1, v_2]}= [\mathcal{L}_{v_1}, i_{v_2}]$,
    % \begin{equation*}
    %   \int_{M} m([v_1, v_2]^{\mathcal{R}})\mu = \left\langle \tilde{m},[v_1, v_2]^{\mathcal{R}} \right\rangle 
    %   =-d^{\mathcal{R}}  \tilde{m}(v_1, v_2) + \# v_1 \cdot  \left\langle \tilde{m},v_2\right\rangle  - \# v_2 \cdot  \left\langle \tilde{m},v_1\right\rangle, 
    % \end{equation*}
then we obtain
    \begin{align}
        &\mathcal{P}_{\tilde{m}}((v_1, \tilde{n}_1),(v_2, \tilde{n}_2))\nonumber\\
        & = \int_{M} m([v_1, v_2]^{\mathcal{R}})\mu + \int_{\Gamma}\bigg( n_2- \text{jump}(i_{v_2} m)\bigg)i_{v_1}\mu  - \int_{\Gamma} \bigg( n_1  -  \text{jump}(i_{v_1} m) \bigg)i_{v_2}\mu \nonumber\\
        % &= -d^{\mathcal{R}}  \tilde{m}(v_1, v_2) +\int_{\Gamma} n_2 i_{v_1}\mu -\int_{\Gamma} n_1 i_{v_2}\mu   \nonumber\\
        % & = -\int_{M}d^{\mathcal{R}} m(v_1,v_2)\mu -\int_{\Gamma} n_1 i_{v_2}\mu + \int_{\Gamma} n_2 i_{v_1}\mu \nonumber
        % &= -d^{\mathcal{R}}  \tilde{m}(v_1, v_2) +\int_{\Gamma} n_2 i_{v_1}\mu -\int_{\Gamma} n_1 i_{v_2}\mu - \int_{M}\dive^{\mathcal{R}} (v_1)i_{v_2} m  \mu + \int_{M}\dive^{\mathcal{R}} (v_2)i_{v_1} m  \mu\nonumber\\
        &= -\int_{M}d^{\mathcal{R}} m(v_1, v_2)\mu -\int_{\Gamma} n_1 i_{v_2}\mu + \int_{\Gamma} n_2 i_{v_1}\mu \nonumber \\
        &\quad - \int_{M}\dive^{\mathcal{R}} (v_1)i_{v_2} m  \mu + \int_{M}\dive^{\mathcal{R}} (v_2)i_{v_1} m \mu.
    \end{align}
\end{proof}

\begin{remark}
    Unlike \cite[Theorem 6.25]{izosimov2018vortex}, which applies to the divergence-free setting, the formulation \eqref{eq: Lie Poisson bracket divergence} includes additional terms involving the divergence operator. This modification arises because we do not impose the divergence-free constraint, leading to a more general expression for the Poisson bracket that accounts for compressible cases. 
    Furthermore, this distinction extends to the Hamiltonian operator introduced in the following Corollary \ref{cor: Hamiltonian operator}, which significantly differs from its counterpart in \cite{izosimov2018vortex}.
\end{remark}

\begin{proposition}
\label{pro: dual of the anchor}
    Let $\tilde{n} \in T_{\Gamma}^*\text{V}(M)(\text{C}^{\infty}(\Gamma)\otimes \text{Dense}(\Gamma)).$ Then its image under the map $\#^{\ast}: T_{\Gamma}^*\text{V}(M)\rightarrow \text{DVect}(M, \Gamma)^*$ is given by
    \begin{equation}
        \#^{\ast} \tilde{n} := \{(d^{\mathcal{R}} h + T^{\mathcal{R}}h)\otimes \mu|\text{~for any~} h\in \text{DC}^{\infty}(M), \text{jump}(h) = n\},
    \end{equation}
    and the operator $T$ is defined by the condition that $\int_{M}\dive^{\mathcal{R}}(v)f\mu = \langle T^{\mathcal{R}} f\otimes \mu, v\rangle$ for all $v\in \text{DVect}(M), f\in\text{DC}^{\infty}(M)$.
\end{proposition}

\begin{proof}
    Let $\tilde{n} \in T_{\Gamma}^*\text{V}(M), v \in \text{DVect}(M, \Gamma)$, then we have 
    \begin{equation}
    <\#^{\ast} \tilde{n},v > = <\tilde{n}, \# v> = \int_{\Gamma} n i_{v}\mu.
    \end{equation}
    By Lemma \ref{lemma: jump}, for any $h\in \text{DC}^{\infty}(M)$ such that $\text{jump}(h) = n$, then we can rewite the above integral as 
    \begin{align*}
           <\#^{\ast} \tilde{n},v > &= \int_{\Gamma} n i_{v}\mu
           = \int_{\Gamma} \text{jump}(h)i_{v}\mu \\
           &= \int_{M}(\mathcal{L}_{v}^{\mathcal{R}} h)\mu + \int_{M}\dive^{\mathcal{R}}(v)h\mu 
           =  \int_{M} (i_{v} d^{\mathcal{R}} h)\mu + \int_{M}\dive^{\mathcal{R}}(v)h\mu \\
           &=<d^{\mathcal{R}} h \otimes \mu, v> + < T^{\mathcal{R}}h\otimes \mu, v>. 
    \end{align*}
    Then we can get $\#^{\ast} \tilde{n} := (d^{\mathcal{R}} h + T^{\mathcal{R}}h)\otimes \mu$.
\end{proof}

\begin{remark}
    If there is no sliding boundary, then it turns to the continuous case, so we have $\int_{M}\dive(v)f\mu = \left \langle-df\otimes \mu, v\right \rangle$, for all $v\in \text{Vect}(M), f\in\text{C}^{\infty}(M)$.
\end{remark}

\begin{corollary}\label{cor: Hamiltonian operator}
    The Hamiltonian operator
    \begin{equation*}
        \mathcal{P}_{\tilde{m}}^{\#}: T_{\tilde{m}}^* \text{DVect}(M)^* \rightarrow T_{\tilde{m}} \text{DVect}(M)^* 
    \end{equation*}
    corresponding to the Poisson bracket on $\text{DVect}(M)^*$ is given by
    \begin{equation}
    \label{eq: Hamiltonian operator}
        (v, \tilde{n})\mapsto \big((-i_{v}d^{\mathcal{R}} m - \dive^{\mathcal{R}}(v)\cdot m - d^{\mathcal{R}} i_{v}m\big)\otimes \mu + \#^{\ast} \widetilde{\text{jump}(i_{v}m)}-\#^{\ast} \tilde{n}, \#v\big),
    \end{equation}
    where $\#^{\ast}: T_{\Gamma}^*\text{V}(M)\rightarrow \text{DVect}(M, \Gamma)^*$ is the dual of the anchor map, explicitly given by
    \begin{equation}
    \label{eq: the dual of anchor}
        \#^{\ast} \tilde{n} := \{(d^{\mathcal{R}} h + T^{\mathcal{R}} h)\otimes \mu| \text{~for any~} h\in \text{DC}^{\infty}(M), \text{jump}(h) = n\}.
    \end{equation}
\end{corollary}

\begin{proof}
    To derive the explicit form of the Hamiltonian operator $\mathcal{P}_{\tilde{m}}^{\#}$ associated with the Poisson bracket on $\text{DVect}(M)^*$, we evaluate the pairing $\langle (v_2, \tilde{n}_2) , \mathcal{P}_{\tilde{m}}^{\#}(v_1, \tilde{n}_1)\rangle $:
    \begin{align*}
        &\langle (v_2, \tilde{n}_2) , \mathcal{P}_{\tilde{m}}^{\#}(v_1, \tilde{n}_1)\rangle = \mathcal{P}_{\tilde{m}}((v_1, \tilde{n}_1), (v_2, \tilde{n}_2)) \\
        &= -\int_{M}d^{\mathcal{R}} m(v_1, v_2)\mu - \int_{M}\dive^{\mathcal{R}} (v_1)i_{v_2} m  \mu + \int_{M}\dive^{\mathcal{R}} (v_2)i_{v_1} m \mu \\
        & \quad -\int_{\Gamma} n_1 i_{v_2}\mu + \int_{\Gamma} n_2 i_{v_1}\mu .
    \end{align*}
    Consider any $h\in\text{DC}^\infty(M)$ satisfying $\text{jump}(h) = n_1$. This allows us to rewrite the boundary integral as follows:
    \begin{align*}
        \int_{\Gamma} n_1 i_{v_2}\mu  
        &= \int_{\Gamma} \text{jump}(h)i_{v_2}\mu \\
        &= \int_{M}(\mathcal{L}_{v_2}^{\mathcal{R}} h)\mu + \int_{M}\dive^{\mathcal{R}}(v_2)h\mu\\
        &= \int_{M} (i_{v_2} d^{\mathcal{R}} h)\mu + \int_{M}\dive^{\mathcal{R}}(v_2)h\mu.
    \end{align*}
% Here $n_1,n_2 \in \text{C}^\infty(\Gamma)$, take any $f\in \text{DC}^\infty(M)$ with $\text{jump}(f) = n_1$. If $\tilde{n}_1 = \text{jump}(i_{v_1}m)\otimes \text{Dens}(\Gamma)$, then take $f= i_{v_1}m$,
% then according to \eqref{eq: jump equation} we have 
% \begin{equation*}
%     \int_{\Gamma} n_1 i_{v_2}\mu  = \int_{\Gamma} \text{jump}(i_{v_1}m)i_{v_2}\mu = \int_{M}(\mathcal{L}_{v_2}^{\mathcal{R}} i_{v_1}m)\mu + \int_{M}\dive^{\mathcal{R}}(v_2)i_{v_1}m\mu=  \int_{M} (i_{v_2} d^{\mathcal{R}} i_{v_1}m)\mu + \int_{M}\dive^{\mathcal{R}}(v_2)i_{v_1}m\mu.
% \end{equation*}
    Substituting this into the Poisson bracket expression,  we obtain 
    \begin{align}
        & \mathcal{P}_{\tilde{m}}((v_1, \tilde{n}_1), (v_2, \tilde{n}_2))\nonumber\\
        &= \int_{M} i_{v_2}(-i_{v_1}d^{\mathcal{R}} m)\mu-  \int_{M} (i_{v_2} d^{\mathcal{R}} h)\mu - \int_{M}\dive^{\mathcal{R}}(v_2)h\mu \nonumber\\
        &\quad +\int_{\Gamma} n_2 i_{v_1}\mu - \int_{M}\dive^{\mathcal{R}} (v_1)i_{v_2} m  \mu + \int_{M}\dive^{\mathcal{R}} (v_2)i_{v_1} m \mu  \nonumber\\
        & = \int_{M} i_{v_2}(-i_{v_1}d^{\mathcal{R}} m )\mu + \left \langle-\#^{\ast} \tilde{n}_1, v_2 \right \rangle +\int_{\Gamma} n_2 i_{v_1}\mu \nonumber\\
        &\quad + \int_{M} i_{v_2}(-\dive^{\mathcal{R}} (v_1)m )\mu+ \int_{M}\dive^{\mathcal{R}} (v_2)i_{v_1} m \mu
        \nonumber\\
        & = \int_{M} i_{v_2}(-i_{v_1}d^{\mathcal{R}} m - \dive^{\mathcal{R}}(v_1)\cdot m)\mu +\int_{M}\dive^{\mathcal{R}}(v_2)i_{v_1} m\mu\nonumber\\
        &\quad + \left \langle-\#^{\ast} \tilde{n}_1, v_2 \right \rangle+\int_{\Gamma} n_2 i_{v_1}\mu \nonumber\\
        & = \int_{M} i_{v_2}\big(-i_{v_1}d^{\mathcal{R}} m - \dive^{\mathcal{R}}(v_1)\cdot m\big)\mu +\bigg(-\int_M i_{v_2}d^{\mathcal{R}} i_{v_1}m \mu + \int_{\Gamma}\text{jump}(i_{v_1}m)i_{v_2}\mu\bigg)\nonumber\\
        &\quad  + \left \langle-\#^{\ast} \tilde{n}_1, v_2 \right \rangle+\int_{\Gamma} n_2 i_{v_1}\mu \nonumber\\
        & = \int_{M} i_{v_2}\bigg(-i_{v_1}d^{\mathcal{R}} m - \dive^{\mathcal{R}}(v_1)\cdot m - d^{\mathcal{R}} i_{v_1}m\bigg)\mu + \int_{\Gamma}\text{jump}(i_{v_1}m)i_{v_2}\mu\nonumber\\
        &\quad + \left \langle-\#^{\ast} \tilde{n}_1, v_2 \right \rangle+\int_{\Gamma} n_2 i_{v_1}\mu \nonumber\\
        & = \left\langle \bigg(-i_{v_1}d^{\mathcal{R}} m - \dive^{\mathcal{R}}(v_1)\cdot m - d^{\mathcal{R}} i_{v_1}m\bigg)\otimes \mu, v_2  \right\rangle +\left\langle \#^{\ast} \widetilde{\text{jump}(i_{v_1}m)}, v_2 \right\rangle \nonumber\\
         &\quad + \left \langle-\#^{\ast} \tilde{n}_1, v_2 \right \rangle+ \int_{\Gamma} n_2 i_{v_1}\mu \nonumber\\
        & = \left\langle \bigg(-i_{v_1}d^{\mathcal{R}} m - \dive^{\mathcal{R}}(v_1)\cdot m - d^{\mathcal{R}} i_{v_1}m\bigg)\otimes \mu + \#^{\ast} \widetilde{\text{jump}(i_{v_1}m)}-\#^{\ast} \tilde{n}_1
        , v_2  \right\rangle \nonumber\\
        &\quad + \left\langle \#v_1 , \tilde{n}_2 \right\rangle.
        \label{eq: Poisson bracket 1}
    \end{align}   
    From \eqref{eq: Poisson bracket 1}, we could give an explicit formula about the Hamiltonian operator $\mathcal{P}_{\tilde{m}}^{\#}$ on $T_{\tilde{m}}^* \text{DVect}(M)^*$:
    \begin{align}
        \mathcal{P}_{\tilde{m}}^{\#}((v, \tilde{n})) 
        & =\bigg( \big(-i_{v}d^{\mathcal{R}} m - \dive^{\mathcal{R}}(v)\cdot m - d^{\mathcal{R}} i_{v}m\big)\otimes \mu + \#^{\ast} \widetilde{\text{jump}(i_{v}m)}-\#^{\ast} \tilde{n}, \#v\bigg).
        \label{eq: Hamiltonian operator1}
    \end{align}
\end{proof}

\begin{remark}
    In the absence of sliding boundaries, it turns into a smooth and continuous case. In this case, the Lie groupoid $\text{DDiff}(M)$ reduces to the diffeomorphism Lie subgroup $G_V\subseteq\text{Diff}(M)$. The corresponding Hamiltonian operator $\mathcal{P}_{\tilde{m}}^{\#}$ on the cotangent space of the dual Lie algebra $T_{\tilde{m}}^* \text{Vect}(M)^*$ is given by
    \begin{align}
        \mathcal{P}_{\tilde{m}}^{\#}(v)
        % &= -\mathcal{L}_v \tilde{m} = -\mathcal{L}_v (m \otimes \mu) \nonumber\\
        % &= -(\mathcal{L}_v m \otimes \mu + \text{div}(v) m \otimes \mu) \nonumber \\
        % &=  -(\mathcal{L}_v m + \text{div}(v) m) \otimes \mu \nonumber \\
        &= -(i_v dm + d i_v m + \text{div}(v) m) \otimes \mu.\nonumber
    \end{align}
\end{remark}

\subsection{Euler-Arnold equation for Lie groupoids}
\label{subsec: Euler-Arnold equation for Lie groupoids}

In this section, we deduce the extremal equations, namely the Euler-Arnold equations, for the Lie groupoid $\mathrm{DDiff}(M)\rightrightarrows \mathrm{V}(M)$. 
This derivation extends the framework developed in \cite{izosimov2018vortex} for the incompressible setting by adapting it to the compressible case, where volume preservation and divergence-free conditions are no longer imposed. 
Consequently, the associated Euler-Arnold equations are modified accordingly.

Let $\mathcal{I}: \text{DVect}(M) \rightarrow \text{DVect}(M)^*$ be an inertia operator with a proper differential operator defined as $\mathcal{I}(u) = \mathcal{I}^{\mathcal{R}}(u) + \mathcal{I}_{\Gamma}(u)$, where $\mathcal{I}^{\mathcal{R}}$ is the regular part, acting withn each subdomain, and $\mathcal{I}_{\Gamma}$ is an interface operator enforcing jump conditions across $\Gamma$, ensuring the self-adjointess of $\mathcal{I}$. 
The smooth dual space $\text{DVect}(M)^*$ of $\text{DVect}(M)$ can be generated by the $\mathfrak{b}-\text{map}$ $\text{DVect}(M)^*:=\{\mathcal{I}(u):~u\in \text{DVect}(M)\}$. 
Then for $u \in \text{DVect}(M, \Gamma)$, we have $\mathcal{I}(u) = u^{\mathfrak{b}} \otimes \mu = m \otimes \mu = \tilde{m}$, where $m = u^{\mathfrak{b}} \in \text{D}\Omega^1(M, \Gamma)$ and $u^{\mathfrak{b}}$ denotes the corresponding 1-form, and the inverse map is denoted by $\sharp$. 

The inertia operator defines a metric on $\text{DVect}(M)$ given by the following: for $u, v \in \text{DVect}(M, \Gamma)$,
\begin{equation}
    \left\langle u, v \right\rangle_{\text{DVect}(M, \Gamma)}
    := \left\langle \mathcal{I}(u), v \right\rangle
    = \left\langle m \otimes \mu, v \right\rangle
    = \int_M (i_v u^{\mathfrak{b}}) \mu = \int_M (i_v m) \mu.
    \label{eq: metric on algebroid}
\end{equation}

The inverse of the inertia operator $\mathcal{I}$, where $\mathcal{I}^{-1}(\tilde{m}) = m^\sharp = u$, exists under proper conditions. In the smooth case, $\mathcal{I}^{-1}$ is represented as a convolution with a Green's function $u  = G * \tilde{m}$, where the corresponding Green's function is known in closed form. 
In the discontinuous case, $\mathcal{I}^{-1}$ can still be understood as the Green kernel associated with the piecewise differential operator together with the interface conditions, but the corresponding kernel is generally not known explicitly. Accordingly, in our implementation, its action is realized numerically through a finite element discretization, rather than by evaluating a closed-form kernel. See Section \ref{sec: Experiments} for implementation details.

% We discretize the velocity field in a finite-element space $V_h$ on a mesh fitted to the prescribed sliding boundary $\Gamma$. For a given discrete momentum $m_h$, the corresponding velocity $u_h \in V_h$ is recovered from the weak problem
% \[
% a_h(u_h,w_h)=\langle m_h,w_h\rangle,\qquad \forall w_h\in V_h,
% \]
% where $a_h(\cdot,\cdot)$ is the bilinear form induced by the inertia operator, including the bulk regularization in each subdomain and the interface contribution across $\Gamma$. After assembly, this yields a sparse linear system $A_h u_h=b_h$, which is solved in FEniCS. Thus, unlike standard kernel-based LDDMM, the inverse inertia operator is realized implicitly through an FEM solve rather than by evaluating a closed-form kernel.

The inner product \eqref{eq: metric on algebroid} defines a kinetic energy given by 
\begin{equation}
    H(v) = \frac{1}{2}  \|v\|_{\text{DVect}(M,\Gamma)}^2 = \frac{1}{2}  \left\langle v,v\right\rangle_{\text{DVect}(M,\Gamma)}. 
    \label{eq: algebroid functional}
\end{equation}
% In terms of momentum $m$ and velocity $v$, they have the following relation:
Since the inertia operator is invertible, we can also define a dual metric on $\text{DVect}(M)^*$:
\begin{align}
    \left\langle \tilde{m}_1, \tilde{m}_2 \right\rangle_{\text{DVect}(M, \Gamma)^*} 
    &:= \left\langle \mathcal{I}^{-1}(\tilde{m}_1), \tilde{m}_2 \right\rangle \nonumber\\
    &= \left\langle \mathcal{I}^{-1}(\tilde{m}_1), \mathcal{I}^{-1}(\tilde{m}_2) \right\rangle_{\text{DVect}(M, \Gamma)} \nonumber\\
    &= \left\langle u_1, u_2 \right\rangle_{\text{DVect}(M, \Gamma)}
    \label{eq: metric on dual algebroid}
\end{align}
where $\tilde{m}_1, \tilde{m}_2$ belong to the same fiber in $\text{DVect}(M)^*$ with $\tilde{m}_1 = \mathcal{I}(u_1), \tilde{m}_2 = \mathcal{I}(u_2)$, 
yielding the corresponding Hamiltonian on $\text{DVect}(M)^*$:
\begin{equation}
    H(\tilde{m}) = \frac{1}{2}  \left\langle \tilde{m},\tilde{m} \right\rangle_{\text{DVect}(M,\Gamma)}^*. 
    \label{eq: dual algebroid functional}
\end{equation}
% Because the inertia operator $\mathcal{I}$ is a differential operator and in this case, the inverse is usually given in terms of the convolution with the Green’s function $G$:
% \begin{equation}
%     v  = G * \tilde{m}.
% \end{equation}

\begin{remark}
An example of a self-adjoint inertia operator $\mathcal{I}$ in the discontinuous setting is
\[
\mathcal{I}(u) = (\mathrm{Id} - \Delta)^{\mathcal{R}} u + \mathcal{I}_{\Gamma}(u),
\]
where the interface term is defined by
\[
\left\langle \mathcal{I}_{\Gamma}(u), v \right\rangle = \int_{\Gamma} \left( v^+ \cdot (\nabla u^+ \cdot n) - v^- \cdot (\nabla u^- \cdot n) \right) \mu,
\]
with homogeneous boundary conditions on $\partial M$.
\end{remark}

Now we calculate the Euler-Arnold equations corresponding to the energy functional \eqref{eq: algebroid functional} under the metric \eqref{eq: metric on algebroid} defined on $\text{DVect}(M)$.

\begin{theorem}[Euler-Arnold equations on the dual Lie algebroid]
    The Euler-Arnold equations for the energy functional \eqref{eq: algebroid functional} corresponding to the metric \eqref{eq: metric on algebroid} on $\text{DVect}(M)$ written in terms of the 1-form density $\tilde{m} \in \text{DVect}(M)^*$ reads
    \begin{equation}
    \label{eq: Euler-Arnold}
    \begin{cases}
        \partial_t^{\mathcal{R}} \tilde{m} + \big(i_{v}d^{\mathcal{R}} m + \frac{1}{2} d^{\mathcal{R}} i_{v}m + \frac{1}{2} \text{div}^{\mathcal{R}}(v) \cdot m \big) \otimes \mu = 0, \\
        \partial_t \Gamma = \#v,
    \end{cases}
    \end{equation}
    where $m$ is the 1-form, and $v = m^{\sharp}$ is the corresponding velocity field.
\end{theorem}

\begin{proof}
    Since
\begin{equation*}
    H(\tilde{m}) = \frac{1}{2}  \left\langle \tilde{m},\tilde{m} \right\rangle_{\text{DVect}(M,\Gamma)}^* 
    = \frac{1}{2} \left\langle  v, \tilde{m} \right\rangle = \frac{1}{2} \int_{M}(m \cdot v)\mu.
\end{equation*}
let $\tilde{m}_t$ be an arbitrary smooth curve in $\text{DVect}(M)^*$ with $\tilde{m}_{t=0} = \tilde{m}$, then using \eqref{eq: discontinuous function evolution} we get
\begin{align}
    \frac{d}{dt}\bigg|_{t=0} H(\tilde{m}_t) 
    &= \frac{1}{2}  \frac{d}{dt}\bigg|_{t=0}\int_{M}(m_t \cdot v_t)\mu \nonumber\\
    &= \frac{1}{2}  \frac{d^{\mathcal{R}}}{dt}\bigg|_{t=0}\int_{M}(m_t \cdot v_t)\mu  + \frac{1}{2} \int_{\Gamma_t} \text{jump}(m_t \cdot v_t) \frac{d\Gamma_t}{dt}\bigg|_{t=0} \nonumber\\
    &= \left\langle v, \frac{d^{\mathcal{R}}}{dt}\bigg|_{t=0} \tilde{m}_t  \right\rangle + \left\langle  \frac{1}{2} ~ \text{jump}(m \cdot v) \otimes \text{Dense}(\Gamma),  \frac{d\Gamma_t}{dt}\bigg|_{t=0}\right\rangle.
\end{align}
According to the definition of the differential of a function in \eqref{eq: function differential}, we have
\begin{equation*}
    d^F H(\tilde{m}_t) = v,  ~~ d^B H(\tilde{m}_t) = \frac{1}{2} ~ \text{jump}(m \cdot v) \otimes \text{Dense}(\Gamma).
\end{equation*}

Using the Poisson bracket on the dual Lie algebroid \eqref{eq: Poisson bracket 1} and the Hamiltonian operator \eqref{eq: Hamiltonian operator}, 
% let $\tilde{n}= \frac{1}{2} ~ \text{jump}~i_vm \otimes \text{Dense}(\Gamma) = \frac{1}{2}~ \widetilde{\text{jump}(i_{v}m)}$, so $h =\frac{1}{2} i_v m $.
% According to the Proposition \ref{pro: dual of the anchor}, we can obtain
we can get the Euler-Arnold equation in the dual Lie algebroid.

For this case, $\tilde{n}$ in the equation \eqref{eq: Hamiltonian operator} is 
$\tilde{n}= \frac{1}{2} ~ \text{jump}~i_vm \otimes \text{Dense}(\Gamma) = \frac{1}{2}~ \widetilde{\text{jump}(i_{v}m)}$, and $h$ in the equation \eqref{eq: the dual of anchor} is $h =\frac{1}{2} i_v m $, then according to the Proposition \ref{pro: dual of the anchor}, we can get
\begin{align}
     \left\langle\#^{\ast} \widetilde{\text{jump}(i_{v}m)}-\#^{\ast} \tilde{n}, v\right\rangle 
    % & = \left\langle\#^{\ast} \big(\widetilde{\text{jump}(i_{v}m)}-\tilde{n}\big), v\right\rangle \nonumber\\
    % & = \left\langle\#^{\ast} (\frac{1}{2} \widetilde{\text{jump}(i_{v}m)}), v\right\rangle \nonumber\\
     &=     \left\langle \frac{1}{2} ~\widetilde{\text{jump}(i_{v}m)}, \#v\right\rangle \nonumber\\
     & = \frac{1}{2}\int_{\Gamma}  \text{jump}(i_{v}m) i_v\mu \nonumber\\
     & = \frac{1}{2}(\int_{M}(\mathcal{L}_{v}^{\mathcal{R}} i_{v}m)\mu + \int_{M}\dive^{\mathcal{R}}(v)i_{v}m\mu \nonumber\\
    % &= \frac{1}{2}( \int_{M} (i_{v} d^{\mathcal{R}} i_{v}m)\mu + \int_{M}i_{v}(\dive^{\mathcal{R}}(v)m)\mu )\nonumber\\
    & = \frac{1}{2}\left\langle (d^{\mathcal{R}} i_{v}m+ \dive^{\mathcal{R}}(v)m)\otimes\mu, v \right\rangle, 
\end{align}
so we can deduce that 
\begin{equation}
    \#^{\ast} \widetilde{\text{jump}(i_{v}m)}-\#^{\ast} \tilde{n} = \frac{1}{2}(d^{\mathcal{R}} i_{v}m+ \dive^{\mathcal{R}}(v)m)\otimes\mu.
\end{equation}
Then using the Hamiltonian operator \eqref{eq: Hamiltonian operator}, we have 
\begin{align}
        \mathcal{P}_{\tilde{m}}^{\#}((v, \tilde{n})) 
        & =\bigg( \big(-i_{v}d^{\mathcal{R}} m - \dive^{\mathcal{R}}(v)\cdot m - d^{\mathcal{R}} i_{v}m\big)\otimes \mu + \#^{\ast} \widetilde{\text{jump}(i_{v}m)}-\#^{\ast} \tilde{n}, \#v\bigg)\nonumber\\
        & = \bigg( \big(-i_{v}d^{\mathcal{R}} m - \dive^{\mathcal{R}}(v)\cdot m - d^{\mathcal{R}} i_{v}m\big)\otimes \mu \nonumber\\ 
        &\quad + \frac{1}{2}(d^{\mathcal{R}} i_{v}m+ \dive^{\mathcal{R}}(v)m)\otimes\mu, \#v\bigg)\nonumber\\
        & = \bigg( \big(-i_{v}d^{\mathcal{R}} m - \frac{1}{2}\dive^{\mathcal{R}}(v)\cdot m - \frac{1}{2}d^{\mathcal{R}} i_{v}m\big)\otimes \mu, \#v\bigg).
\end{align}
So finally we get the Euler-Arnold equation on the dual Lie algebroid:
\begin{equation}
    \begin{cases}
        \partial_t^{\mathcal{R}} \tilde{m}  + \big(i_{v}d^{\mathcal{R}} m + \frac{1}{2} d^{\mathcal{R}} i_{v}m +\frac{1}{2}\dive^{\mathcal{R}}(v)\cdot m  \big)\otimes \mu =0, \\
        \partial_t\Gamma = \#v.
    \end{cases}
\end{equation}
\end{proof}

\subsection{Groupoid representation for discontinuous image registration}
\label{subsec: Groupoid representation for discontinuous image registration}

In this section, building on the groupoid representation theory introduced in the previous subsections, we extend the LDDMM method to address discontinuous image registration problems with sliding boundaries, resulting in an image registration method based on the discontinuous diffeomorphism Lie groupoid.

Specifically, consider two images $I_m: \Omega_{m}\rightarrow \mathbb{R}$ and $I_f: \Omega_{f}\rightarrow \mathbb{R}$ defined on the domains $\Omega_{m} \subset \mathbb{R}^{2}$ and $\Omega_{f} \subset \mathbb{R}^{2}$, where $I_m$ is the moving image and $I_f$ is the fixed image.
Assume that the sliding boundary $\Gamma$ is precomputed based on a prior segmentation on the space of the fixed image. The aim of 
discontinuous image registration is to find a reasonable discontinuous diffeomorphic deformation $\phi: \Omega_{f}\rightarrow \mathbb{R}^{2}$ such that warps the moving image to align it with the fixed image, i.e. $I_f \approx I_m^* = \phi \cdot I_m = I_m \circ \phi^{-1}$  \cite{modersitzki2004numerical}.
% and $\Gamma \approx \phi \cdot \Gamma_0 = \Gamma_0 \circ \phi^{-1} \in \text{VS}(M)$, where $\Gamma_0 \in \Omega_f$.
In fact, given $\Gamma$, the registration problem seeks a discontinuous deformation $\phi = (\Gamma, \Gamma_1, \phi^+, \phi^-)$ located on the source fiber $\text{DDiff}(\Omega_f)_{\Gamma}$ corresponding to $\Gamma$.

% Specifically, consider two images $I_m$ and $I_f$ defined on the image domain $\Omega$, where $I_m$ is the moving image and $I_f$ is the fixed image. Assume that each image has a hypersurface $\Gamma_0$ and $\Gamma_1$, respectively, such that when registering $I_m$ to $I_f$, the deformation $\phi$ is discontinuous along $\Gamma_0$, but diffeomorphic on either side of $\Gamma_0$. Thus, the discontinuous registration problem can be described as finding a reasonable discontinuous diffeomorphic deformation $\phi = (\Gamma_0, \Gamma_1, \phi^+, \phi^-) \in \text{DDiff}(M)$ (using $M$ to represent the image domain $\Omega$ for consistency with the previous subsections), such that $I_f \approx \phi \cdot I_m = I_m \circ \phi^{-1}$ and $\Gamma_1 = \phi \cdot \Gamma_0 \in \text{VS}(M)$. Additionally, we seek such a $\phi$ to be as simple as possible, i.e., as close as possible to the identity deformation $\text{Id}_{\Gamma_0} = (\Gamma_0, \Gamma_0, \text{Id}, \text{Id})$.
% In fact, given $\Gamma_0$, the registration problem seeks a $\phi = (\Gamma_0, \Gamma_1, \phi^+, \phi^-)$ located on the source fiber $\text{DDiff}(M)_{\Gamma_0}$ corresponding to $\Gamma_0$.

Similar to the LDDMM method, we assume the deformation $\phi = \phi_{01}^v$ is obtained as the endpoint of a flow generated by the flow equation \eqref{eq: differential equation}, and this deformation achieves the registration between the two images. The velocity $v: \Omega_{f}\rightarrow \mathbb{R}^{2}$ is located in the Lie algebroid $\text{DVect}(\Omega_{f})$ corresponding to $\text{DDiff}(\Omega_{f})$, and more specifically, it is located on the fiber $\mathrm{DVect}(\Omega_{f}, \Gamma)$ at $\Gamma$. Therefore, to ensure the discontinuous deformation field is as simple as possible and as close as possible to the identity deformation $\text{Id}_{\Gamma}$, we use the metric defined in \eqref{eq: metric on algebroid} on $\text{DVect}(\Omega_{f})$.
Thus, the regularization term for the deformation field can be defined as
\begin{equation*}
    E_R(\phi) = \min_{v(t) \in \text{DVect}(\Omega_{f}, \Gamma_0),\phi_{01}^v= \phi}  \int_0^1 \|v(t)\|_{\text{DVect}(\Omega_{f}, \Gamma)}^2 dt.
\end{equation*}
Therefore, the registration method based on the discontinuous diffeomorphism Lie groupoid can be described as follows:

\begin{definition}[Discontinuous Image registration based on Lie groupoid]
    Given two images $I_m$ and $I_f$ defined on the domains $\Omega_{m}$ and $\Omega_{f}$, respectively, and the sliding boundary $\Gamma$ on the fixed image $I_f$, if a time-dependent vector field $v(t) (t \in [0, 1]) \in \text{DVect}(\Omega_{f}, \Gamma)$ can be found to minimize the following energy functional
    \begin{equation}
        E(v) = E_S(\phi \cdot I_m, I_f) + \frac{1}{2} \int_0^1 \|v(t)\|_{\text{DVect}(\Omega_{f}, \Gamma)}^2 dt,
        \label{eq: groupoid energy}
    \end{equation}
    where $\phi_{0t}^v$ is the flow generated by $v(t)$, that is,
    \begin{equation}
        \partial_t \phi_{0t}^v(x) = v(t, \phi_{0t}^v(x)), \ \phi_{00}^v(x) = x,
        \label{eq: differential equation2}
    \end{equation}
    then $\phi = \phi_{01}^v$ is the optimal discontinuous diffeomorphic deformation to the discontinuous image registration problem.
\end{definition}

% Analogous to the construction of the space $V$ in LDDMM, an appropriate kernel function $K$ is also chosen to construct the space $\text{DVect}(M, \Gamma_0)$, so that the velocity $v$ and the momentum $m$ satisfy $v(t) = K \ast m(t)$, i.e., $v^+(t) = K^+ \ast m^+(t), v^-(t) = K^- \ast m^-(t)$. 

Figure \ref{fig: vortex_sheet_LDDMM} illustrates the proposed framework of discontinuous image registration based on Lie groupoid and algebroid.
Using the Euler-Arnold equations \eqref{eq: Euler-Arnold} derived in the previous subsection \ref{subsec: Euler-Arnold equation for Lie groupoids}, we can obtain the solution to the optimization problem \eqref{eq: groupoid energy}, thereby achieving the optimal discontinuous deformation.

\begin{remark}
    In the absence of a sliding boundary, the situation turns into a smooth and continuous case that is LDDMM. The Euler-Arnold equations \eqref{eq: Euler-Arnold} reduce to
    \begin{equation*}
        \partial_t \tilde{m} + (i_v dm + d i_v m + \text{div}(v) m) \otimes \mu = 0,
    \end{equation*}
    which is equivalent to the EPDiff equation in LDDMM \eqref{eq: one-form density EPDiff}:
    $\partial_t \tilde{m} + \mathcal{L}_v \tilde{m} = 0$.
\end{remark}

% \subsection{Conclusion}
% In this chapter, we have detailed the theory and method of image registration based on the discontinuous diffeomorphism Lie groupoid. By extending the Large Deformation Diffeomorphic Metric Mapping (LDDMM) method to include scenarios with discontinuities along hypersurfaces, we proposed a new theoretical and methodological framework for more accurately addressing image registration problems involving sliding motions.

% First, we defined and theoretically explained the discontinuous diffeomorphism Lie groupoid, followed by the introduction of related mathematical structures, including the Lie algebroid of discontinuous vector fields, the dual Lie algebroid, and the Euler-Arnold evolution equations. This comprehensive mathematical theory provides a new perspective and solution approach for image registration problems involving sliding motions. Finally, we presented a registration framework suitable for sliding motion based on the discontinuous diffeomorphism Lie groupoid.

% Although this chapter has made significant theoretical advancements, applying these theories to real image data presents some challenges. These challenges mainly include the complexity of sliding boundaries in images and the difficulty of translating complex mathematical structures into effective numerical computation algorithms.

% Looking ahead, we will continue to refine and deepen the research on this theoretical framework, develop corresponding computational methods, and further improve the efficiency and accuracy of registration algorithms.

\section{Experiments}
\label{sec: Experiments}
In this section, we test the proposed registration framework \eqref{eq: groupoid energy} and \eqref{eq: differential equation2} on both synthetic examples and real lung images. 
The results in this section are specific to $d=2$, and a three-dimensional implementation is an ongoing work.
The implementation combines the mermaid library (\url{https://github.com/uncbiag/mermaid}), which contains various image registration methods, with FEniCS (\url{https://fenicsproject.org}), a platform for solving partial differential equations using the finite element method (FEM) \cite{langtangen2017solving}. We use automatic differentiation to eliminate manual gradient computation. The code used for the numerical experiments is available at \url{https://github.com/baolily/groupoid-based-registration}. 

In particular, unlike in standard kernel based LDDMM, the inverse inertia operator is realized implicitly through an FEM solve rather than by evaluating a closed-form kernel. 
% The action of the inverse inertia operator is computed in FEniCS through a finite-element solve. 
To recover the velocity $u$ from the momentum $m$, we discretize the velocity field in a finite-element space $V_h$ on a mesh fitted to the prescribed sliding boundary $\Gamma$. The continuous inertia operator $\mathcal{I}$  then induces a discrete operator $\mathcal{I}_h$. Given a discrete momentum $m_h$, the corresponding velocity $u_h \in V_h$ is obtained by solving $
\langle \mathcal{I}_h(u_h), w_h \rangle = \langle \tilde m_h, w_h \rangle, \forall w_h \in V_h.$ 
After finite element discretization, this variational problem is assembled into a sparse linear system and solved in FEniCS.
% After choosing a basis of $V_h$, this yields a sparse linear system
% \[
% A_h u_h = b_h,
% \]
% where $A_h$ is the matrix representation of $\mathcal{I}_h$. 
% This equation is solved in FEniCS. 
% Hence, unlike in standard kernel-based LDDMM, the inverse inertia operator is not evaluated through a closed-form kernel, but implicitly through an FEM solve.

We test the proposed model on images exhibiting sliding motions and compare the results with those obtained by registration using Total Variation (TV) regularization \cite{frohn-schauf_multigrid_2008}, which also focuses on discontinuous registration, as well as with the LDDMM method.
By including the TV method in our comparison, we aim to demonstrate the advantages and effectiveness of the proposed method compared to the existing method that tackles discontinuous registration differently.
For both examples, the segmentations of the sliding boundaries are given by construction.
% We first segment the fixed images to define the sliding boundaries. 
For all experiments, we use the local normalized cross-correlation (LNCC) \cite{Baig_lncc_2012} as similarity measure due to its robustness to local intensity variations. 

We first test our proposed model on synthetic rectangle images, which slide against each other: the upper region is moved to the right while the lower region is moved to the left, as shown in Figures \ref{fig: rectangle template} and \ref{fig: rectangle reference}.  
Figures \ref{fig: rectangle deformed TV} and \ref{fig: rectangle deformed grid TV} show the deformed moving image and deformed grid obtained with TV regularization. 
The TV method is able to capture the sliding motion and preserve the discontinuity across the interface between the two regions.
% Although the TV method can capture the sliding motion, it produces unrealistic deformation as can be seen in Figure \ref{fig: rectangle deformed grid TV}.
Figures \ref{fig: rectangle deformed gaussian} and \ref{fig: rectangle deformed grid gaussian} show the deformed moving image and deformed grid obtained with the LDDMM method. It is clear that the LDDMM produces incorrect motion estimates when sliding motion is present, since its smooth deformation assumption leads to a blurred boundary. In contrast, the proposed diffeomorphism groupoid method can effectively preserve the motion discontinuities at the sliding boundary, as shown in Figures \ref{fig: rectangle deformed vortexsheet} and \ref{fig: rectangle deformed grid vortexsheet}.  

\begin{figure}[!t]
\centering
\subfigure[]{
\includegraphics[height = 2.7cm]{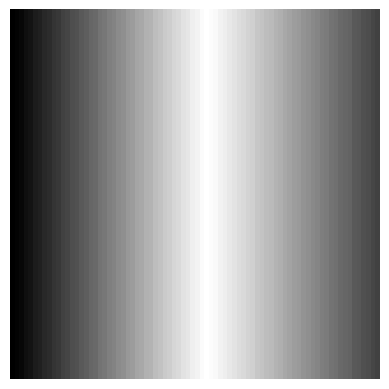}
\label{fig: rectangle template}}
\subfigure[]{
\includegraphics[height = 2.7cm]{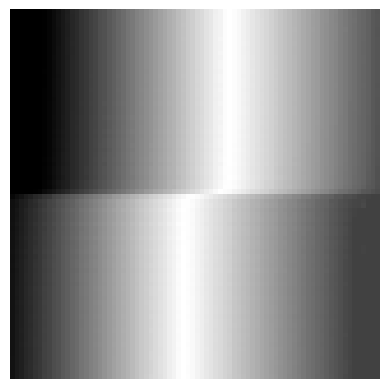}
\label{fig: rectangle deformed TV}}
\subfigure[]{
\includegraphics[height = 2.7cm]{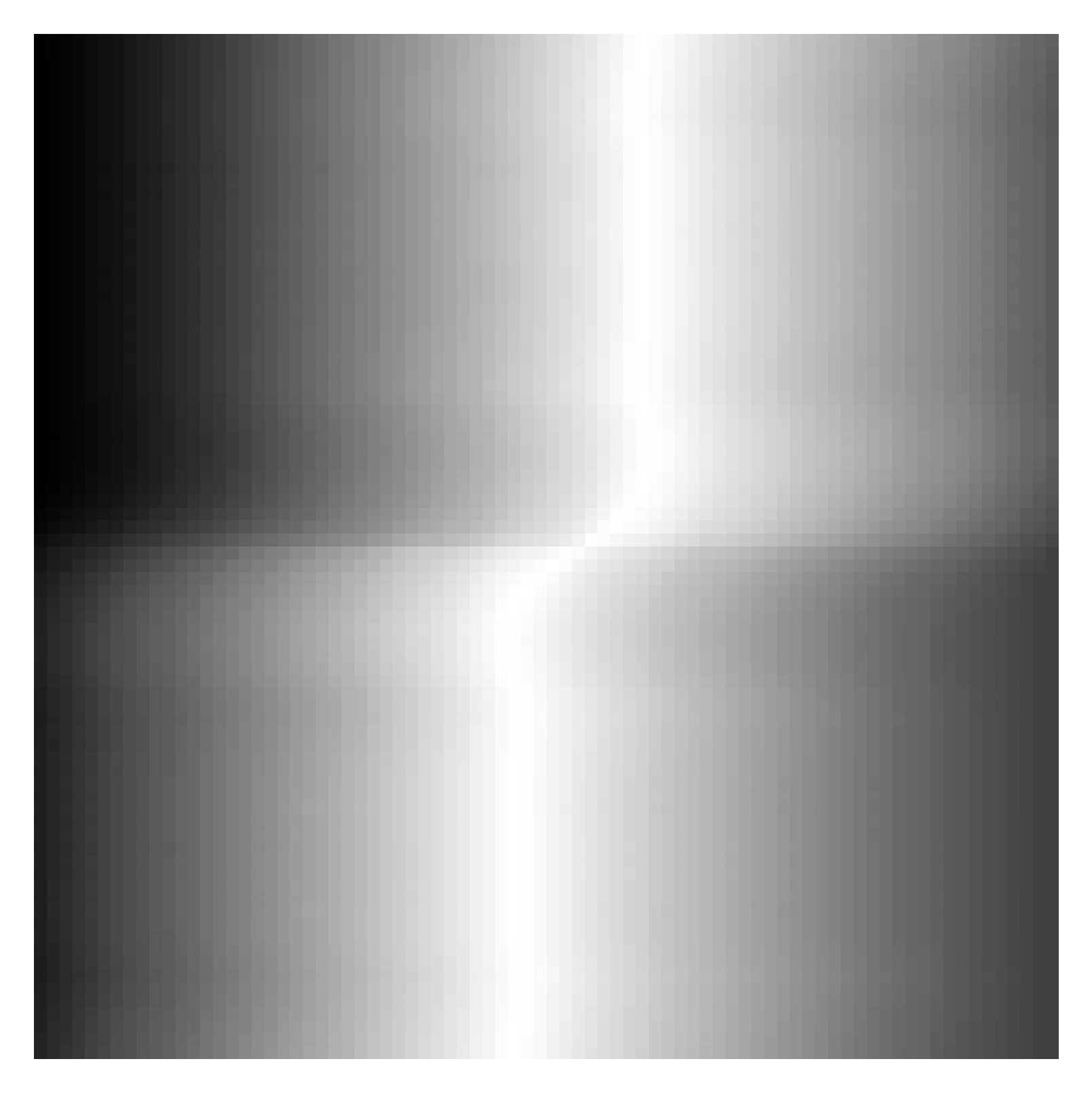}
\label{fig: rectangle deformed gaussian}}
\subfigure[]{
\includegraphics[height = 2.7cm]{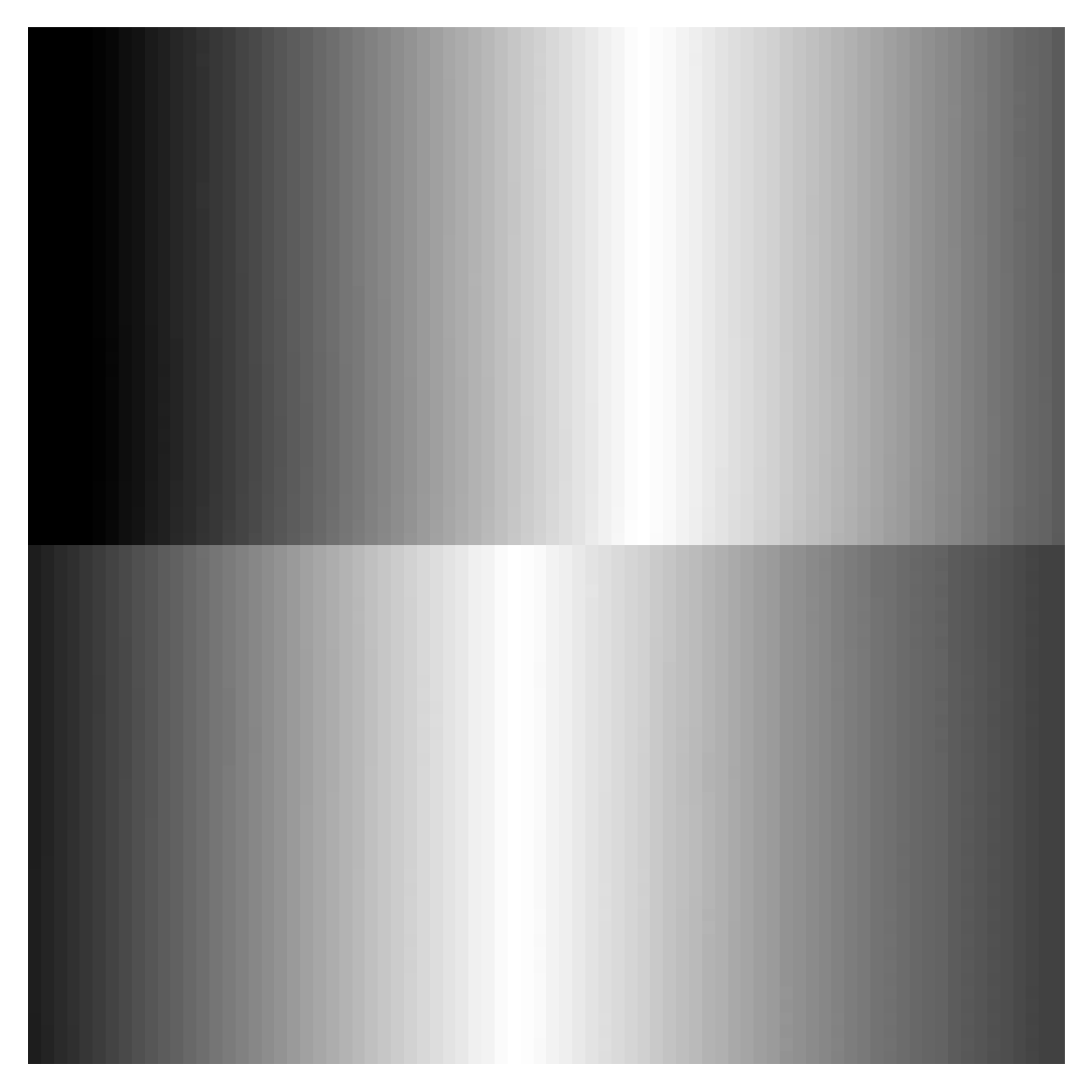}
\label{fig: rectangle deformed vortexsheet}}

\subfigure[]{
\includegraphics[height = 2.7cm]{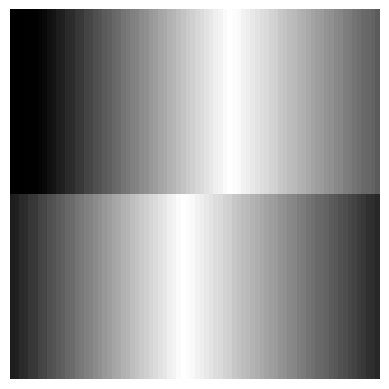}
\label{fig: rectangle reference}}
\subfigure[]{
\includegraphics[height = 2.7cm]{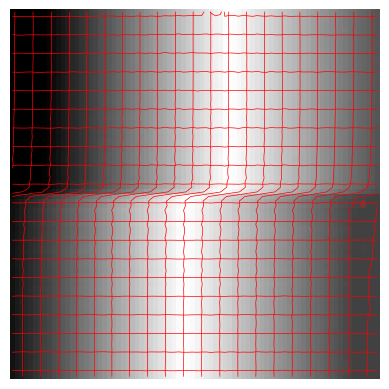}
\label{fig: rectangle deformed grid TV}}
\subfigure[]{
\includegraphics[height = 2.7cm]{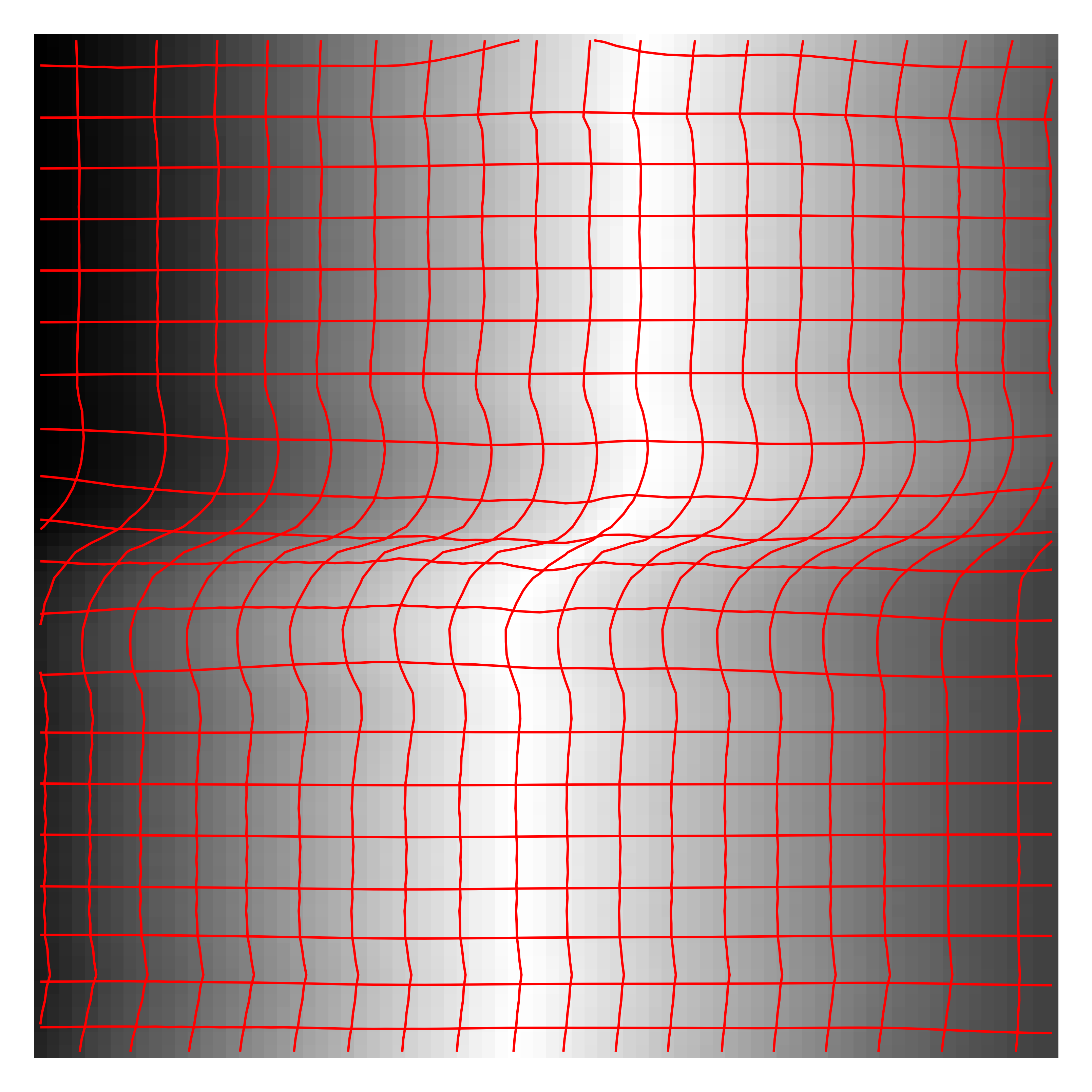}
\label{fig: rectangle deformed grid gaussian}}
\subfigure[]{
\includegraphics[height = 2.7cm]{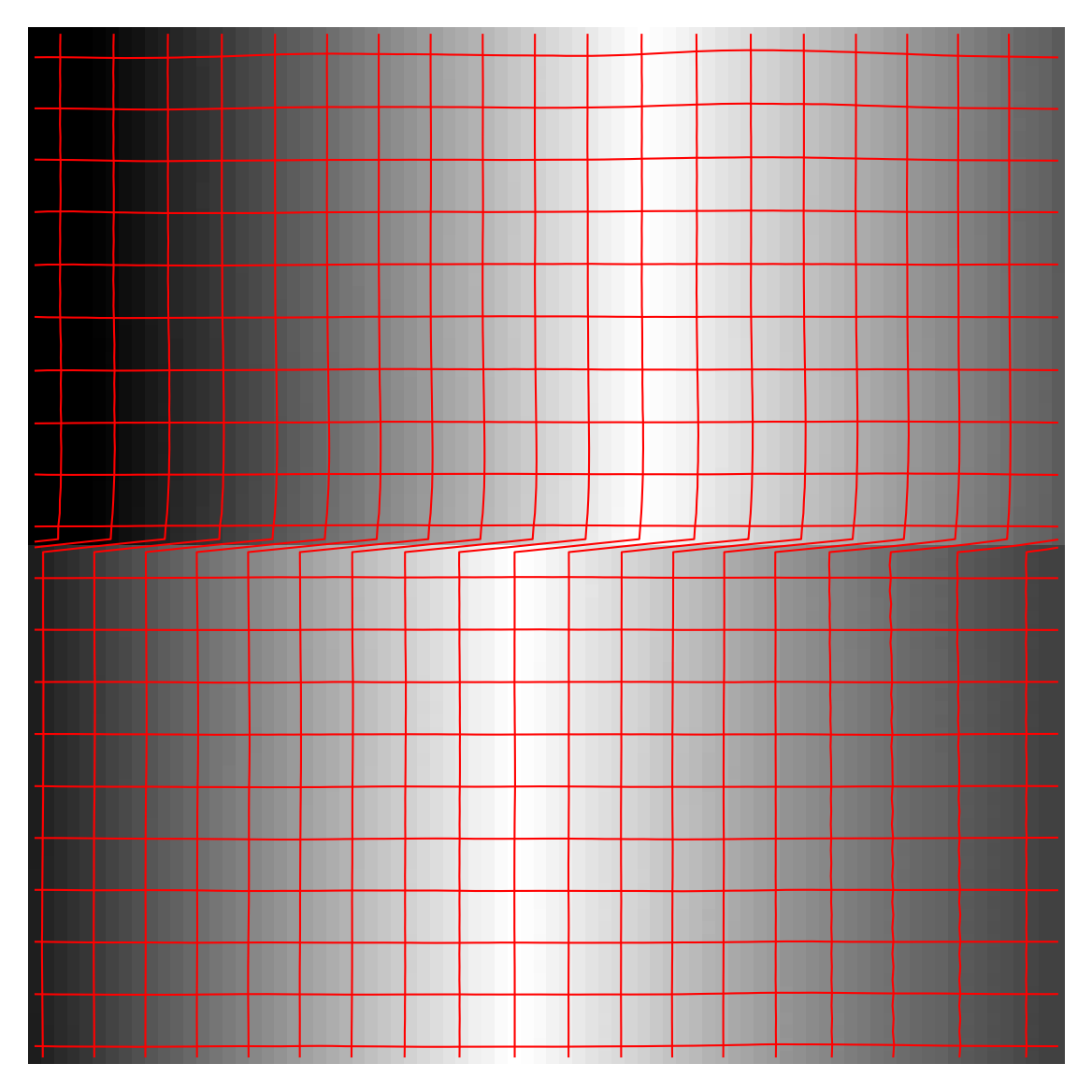}
\label{fig: rectangle deformed grid vortexsheet}}

\caption{\label{fig: rectangle result}Registration results on synthetic rectangle images with sliding motion. \subref{fig: rectangle template} and \subref{fig: rectangle reference} are the moving image and fixed image; 
\subref{fig: rectangle deformed TV}-\subref{fig: rectangle deformed vortexsheet} are the deformed moving images using TV, LDDMM, and the proposed diffeomorphism groupoid, respectively;
\subref{fig: rectangle deformed grid TV}-\subref{fig: rectangle deformed grid vortexsheet} are the deformed grids obtained by TV, LDDMM, and the proposed diffeomorphism groupoid, respectively.}
\end{figure} 

\begin{figure}[htbp]
\centering
\subfigure[]{
\includegraphics[height = 2.6cm]{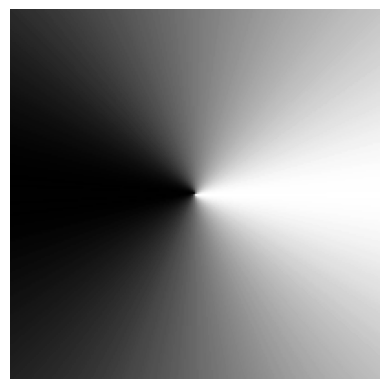}
\label{fig: wheel template}}
\subfigure[]{
\includegraphics[height = 2.6cm]{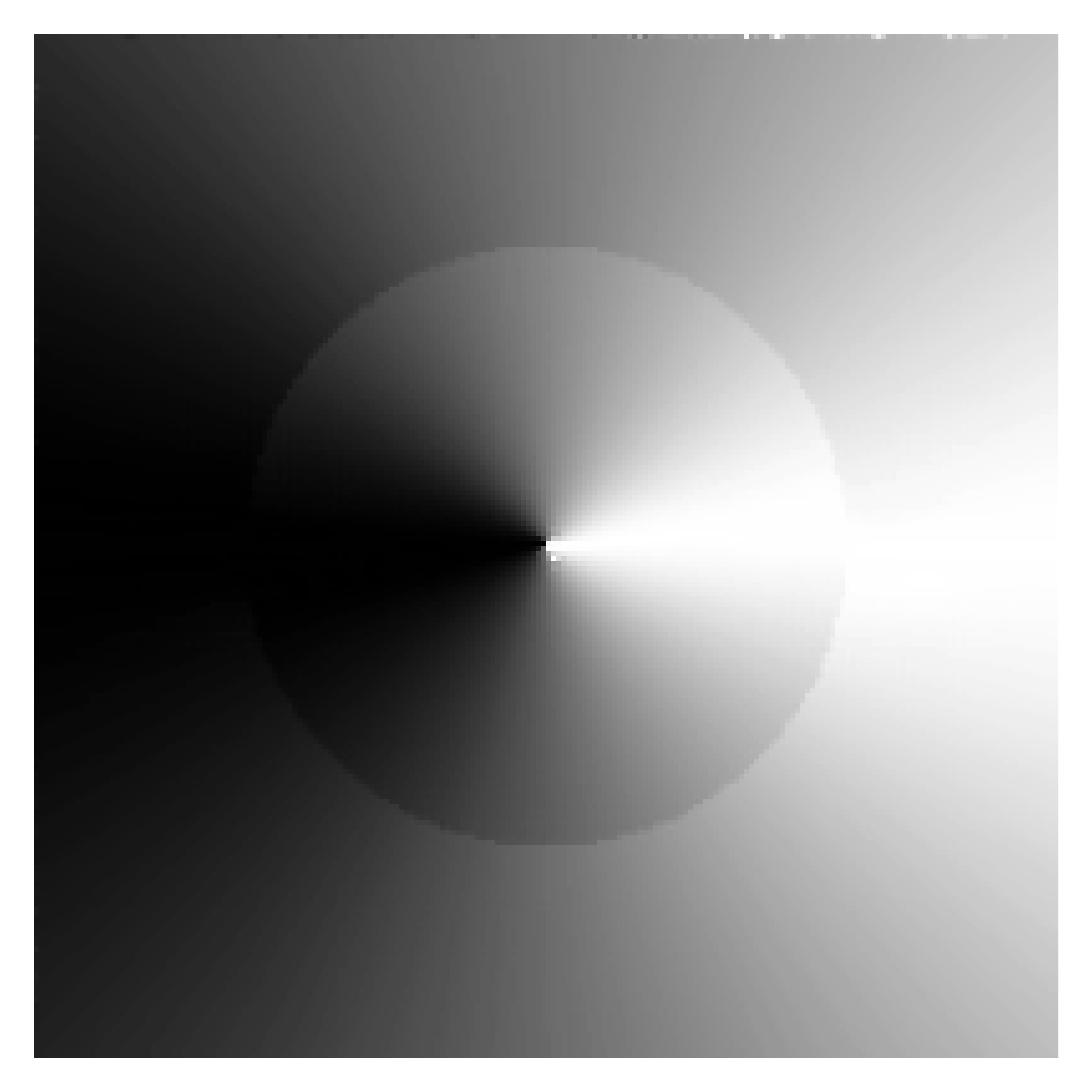}
\label{fig: wheel deformed TV}}
\subfigure[]{
\includegraphics[height = 2.6cm]{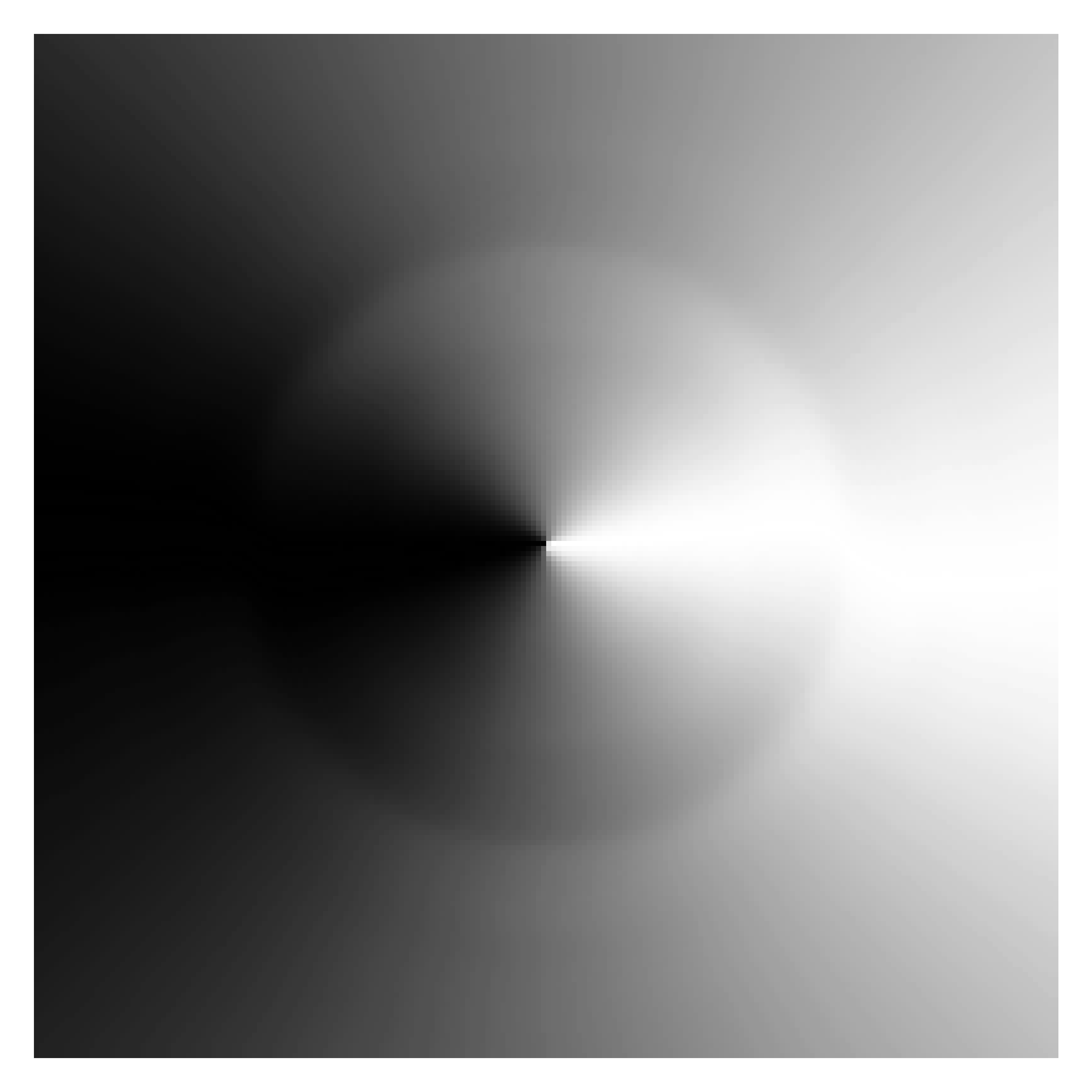}
\label{fig: wheel deformed gaussian}}
\subfigure[]{
\includegraphics[height = 2.6cm]{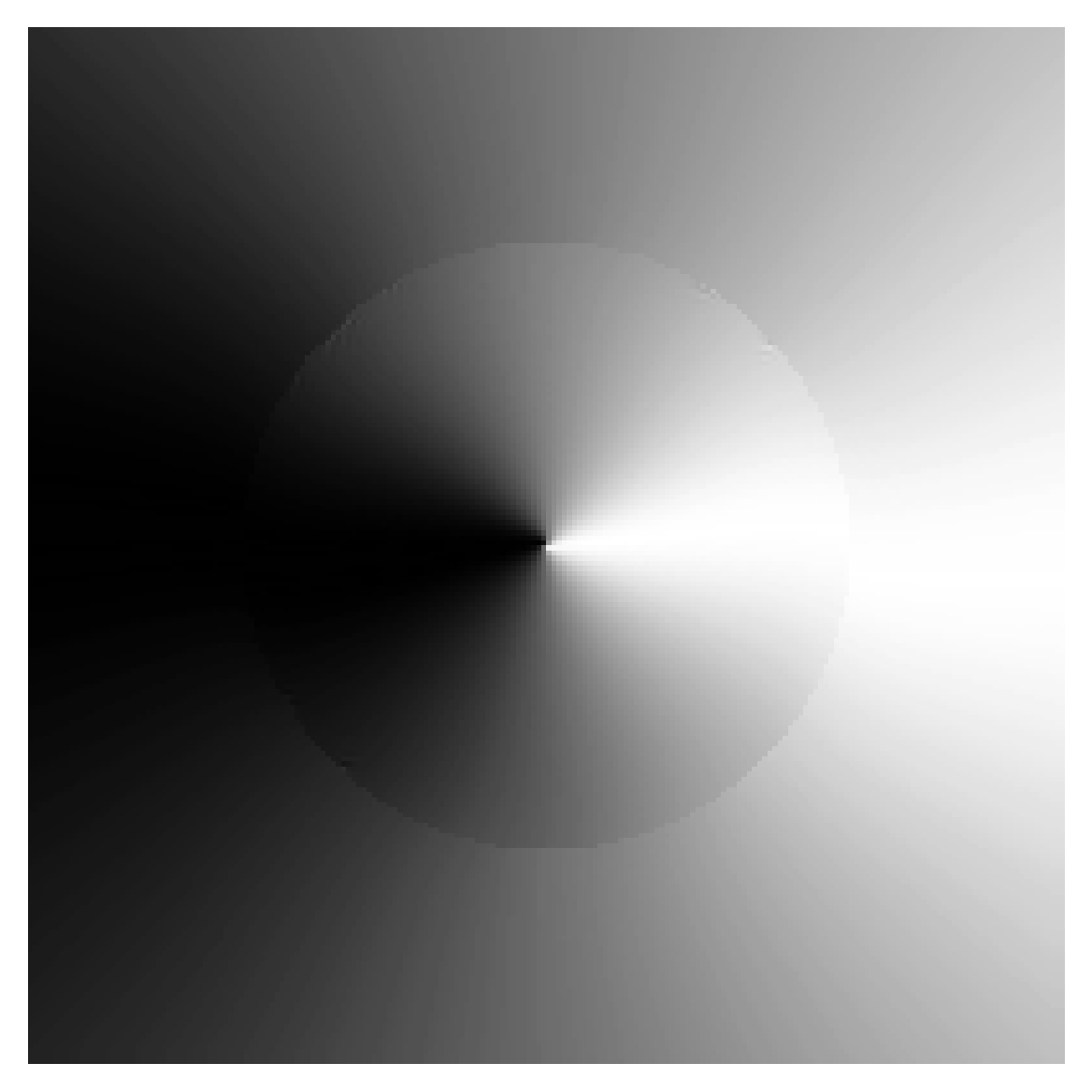}
\label{fig: wheel deformed vortexsheet}}

\subfigure[]{
\includegraphics[height = 2.6cm]{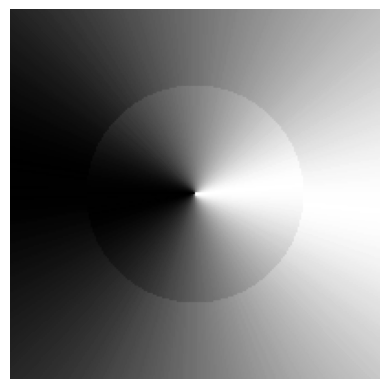}
\label{fig: wheel reference}}
\subfigure[]{
\includegraphics[height = 2.6cm]{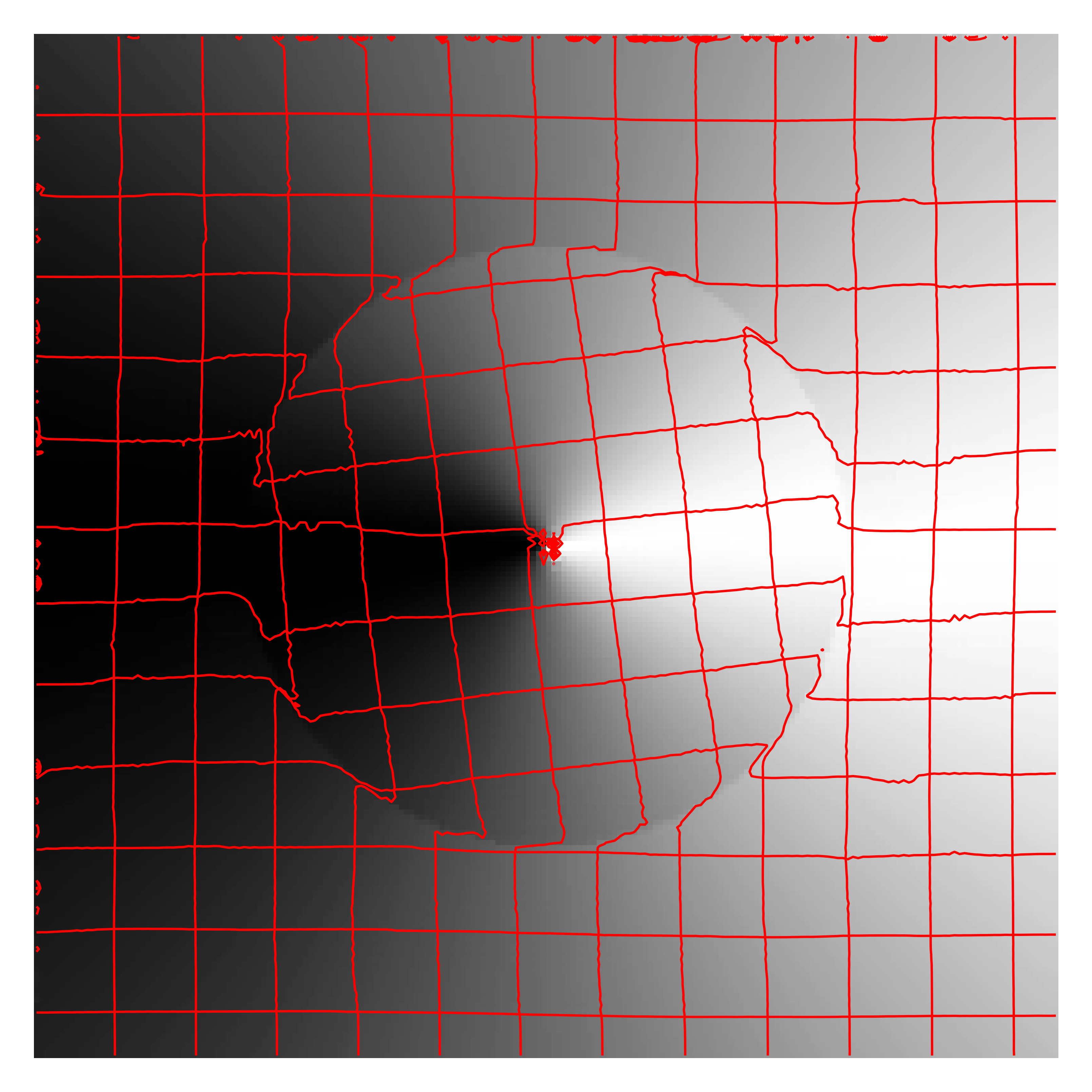}
\label{fig: wheel deformed grid TV}}
\subfigure[]{
\includegraphics[height = 2.6cm]{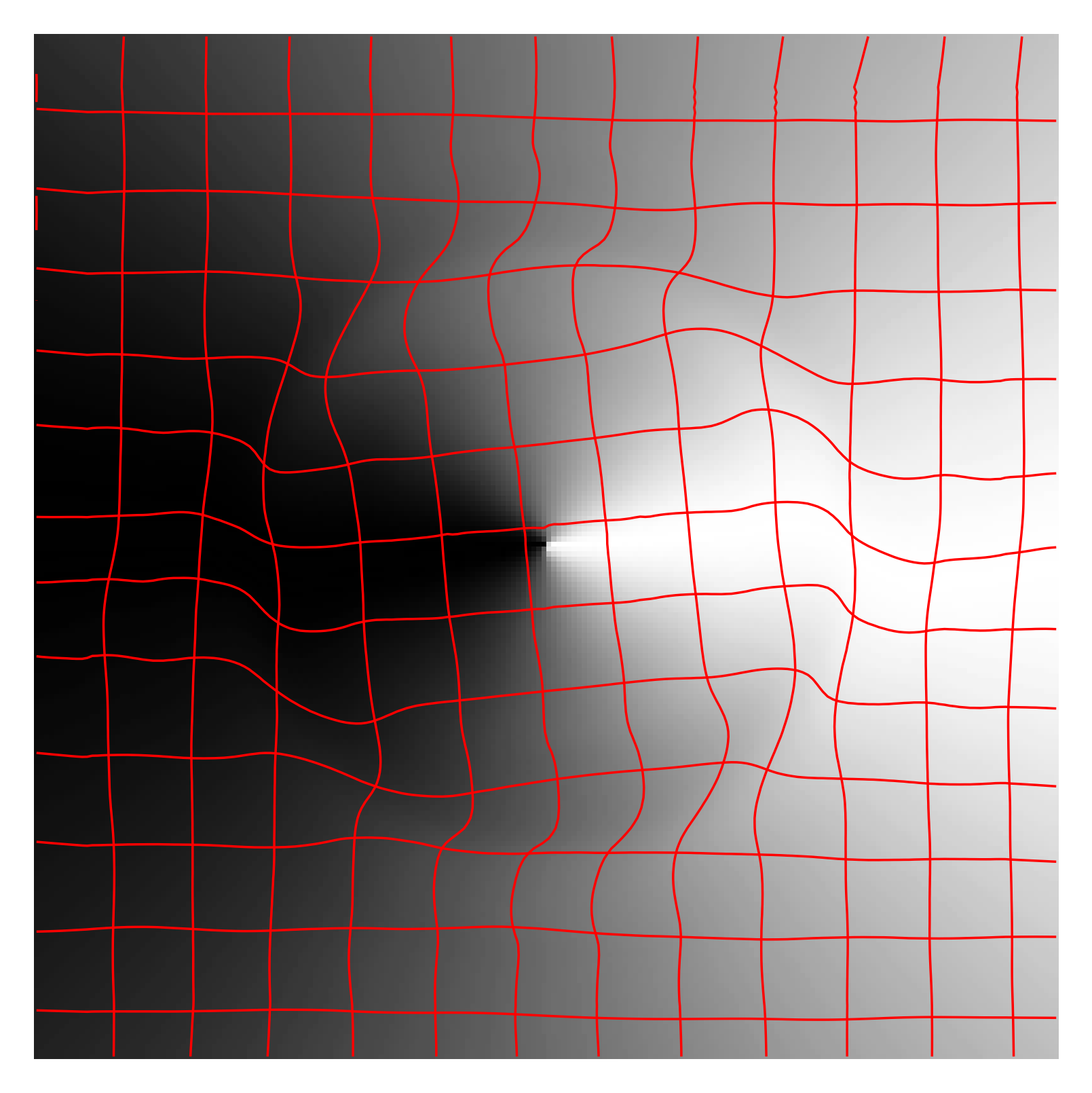}
\label{fig: wheel deformed grid gaussian}}
\subfigure[]{
\includegraphics[height = 2.6cm]{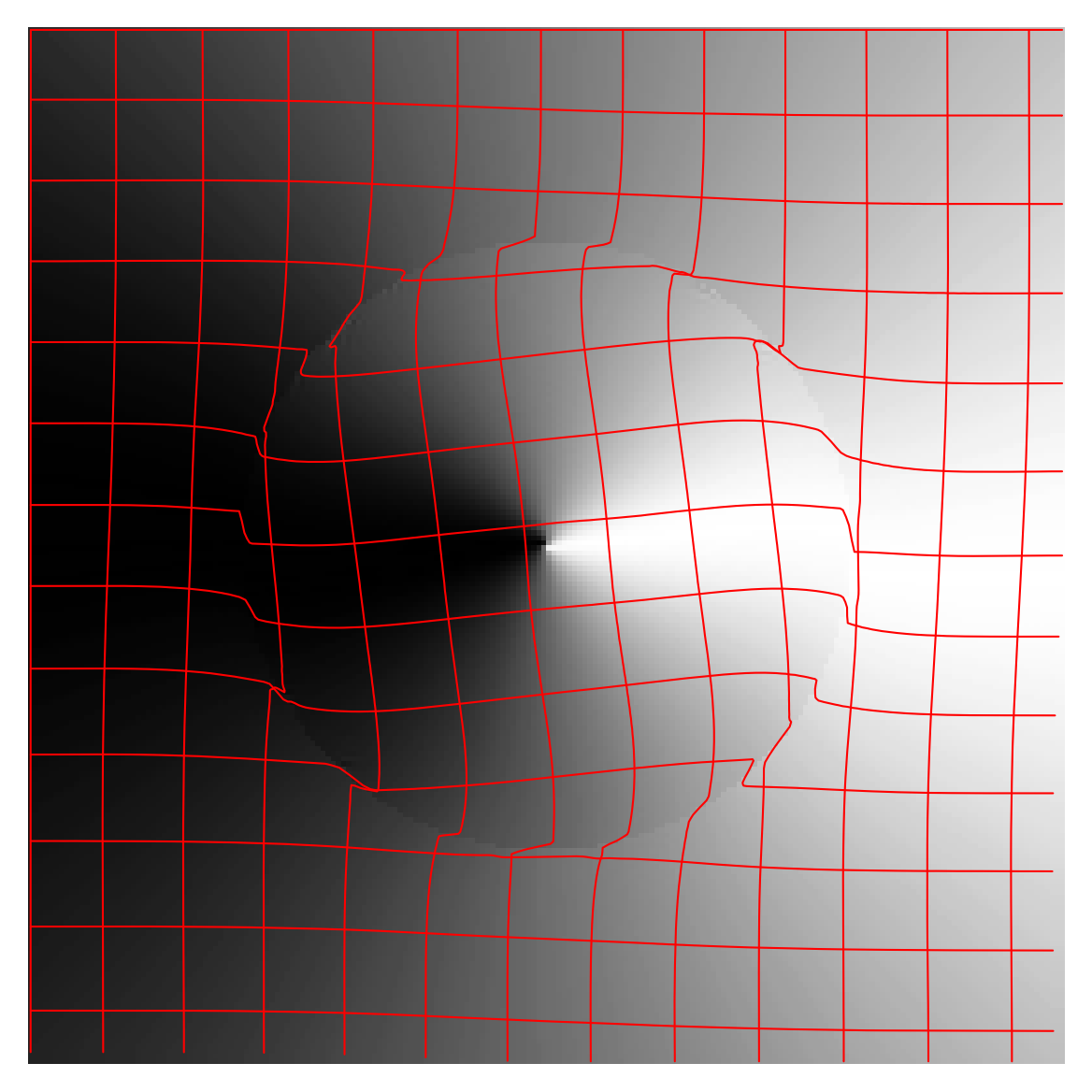}
\label{fig: wheel deformed grid vortexsheet}}
\caption{\label{fig: wheel result}Registration results on synthetic wheel images with sliding motion. \subref{fig: wheel template} and \subref{fig: wheel reference} are the moving image and fixed image; 
\subref{fig: wheel deformed TV}-\subref{fig: wheel deformed vortexsheet} are the deformed moving images using TV, LDDMM, and the proposed diffeomorphism groupoid, respectively;
 \subref{fig: wheel deformed grid TV}-\subref{fig: wheel deformed grid vortexsheet} are the deformed grids obtained by TV, LDDMM, and the proposed diffeomorphism groupoid, respectively.}
\end{figure}

The second evaluation was performed on a more complex sliding motion registration. The fixed image is generated by rotating the inner and outer circular regions of the moving image by 5 degrees in opposite directions, as shown in Figures \ref{fig: wheel template} and \ref{fig: wheel reference}. 
When using the LDDMM method, the inner and outer parts are aligned smoothly as shown in Figures \ref{fig: wheel deformed gaussian} and \ref{fig: wheel deformed grid gaussian}.
In contrast, our proposed diffeomorphism groupoid method can accurately preserve the boundary and deal with the discontinuous deformation, as can be seen in Figures \ref{fig: wheel deformed vortexsheet} and \ref{fig: wheel deformed grid vortexsheet}. Remarkably, our proposed method can achieve comparable results to the TV method, as shown in Figures \ref{fig: wheel deformed TV} and \ref{fig: wheel deformed grid TV}.

To measure the quality of the registration, we consider the following quantities:
\begin{itemize}
    \item Relative Sum of Squared Differences ($Re\_SSD$) to measure the relative residual, which is defined by 
        \begin{equation*}
	       Re\_SSD =\frac{\parallel I_m^* - I_f \parallel^{2}}{\parallel I_m - I_f \parallel^{2}}.
         \end{equation*} 
    \item Normalized Correlation Coefficient ($NCC$) is defined by 
	\begin{equation*}
	   NCC(I_f, I_m^*)=\frac{\left<I_f-\overline{I_f}, ~ I_m^*-\overline{I_m^*}\right>}{\parallel I_f-\overline{I_f} \parallel ~\parallel I_m^*-\overline{I_m^*} \parallel}, 
        \end{equation*}
        \begin{equation*}
        \overline{I_f}=\frac{\int_{\Omega_f}I_f(\textbf{x})d\textbf{x}}{|\Omega_f|}, ~~ \overline{I_m^*}=\frac{\int_{\Omega_f}I_m^*d\textbf{x}}{|\Omega_f|}, 
        \end{equation*}
and $|\Omega_f|$ denotes the number of pixels in the fixed image domain. 
    \item Structural Similarity ($SSIM$) \cite{wang2004image} to evaluate the similarity of the images.
\end{itemize}

\begin{table}[!htbp]
	\centering
	\caption{Comparison of $Re\_SSD(\%)$, $NCC$ and $SSIM$ by different methods. Bold values are the best values.}
	\label{tab: ReSSD}
        \setlength{\tabcolsep}{5mm}{
	\begin{tabular}{ccccc}
		\toprule	
            Data &Methods & $Re\_SSD(\%)$  & $NCC$   & $SSIM$ \\
		\hline\hline
		\multirow{5}{*}{Rectangle}& Before &100 &0.9228 &0.8071 \\
		&TV &7.78 &0.9949 &0.9219 \\
		&LDDMM &10.98 &0.9934 &0.9396 \\
		&Proposed  &\bf{3.37} &\bf{0.9986} &\bf{0.9925} \\
            \midrule
            \multirow{5}{*}{Wheel}& Before  &100 &0.9957 &0.9690 \\
		&TV &61.09 &0.9974 &0.9720 \\
		&LDDMM  &58.11 &0.9975 &0.9788 \\
		&Proposed   &\bf{31.38} &\bf{0.9986} &\bf{0.9878} \\
		\bottomrule
	\end{tabular}
        }
\end{table}

\begin{figure}[!t]
\centering
\subfigure[]{
\includegraphics[height = 1.9cm]{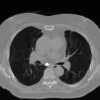}
\label{fig: lung z template}}
\subfigure[]{
\includegraphics[height = 1.9cm]{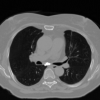}
\label{fig: lung z reference}}
\subfigure[]{
\includegraphics[height = 1.9cm]{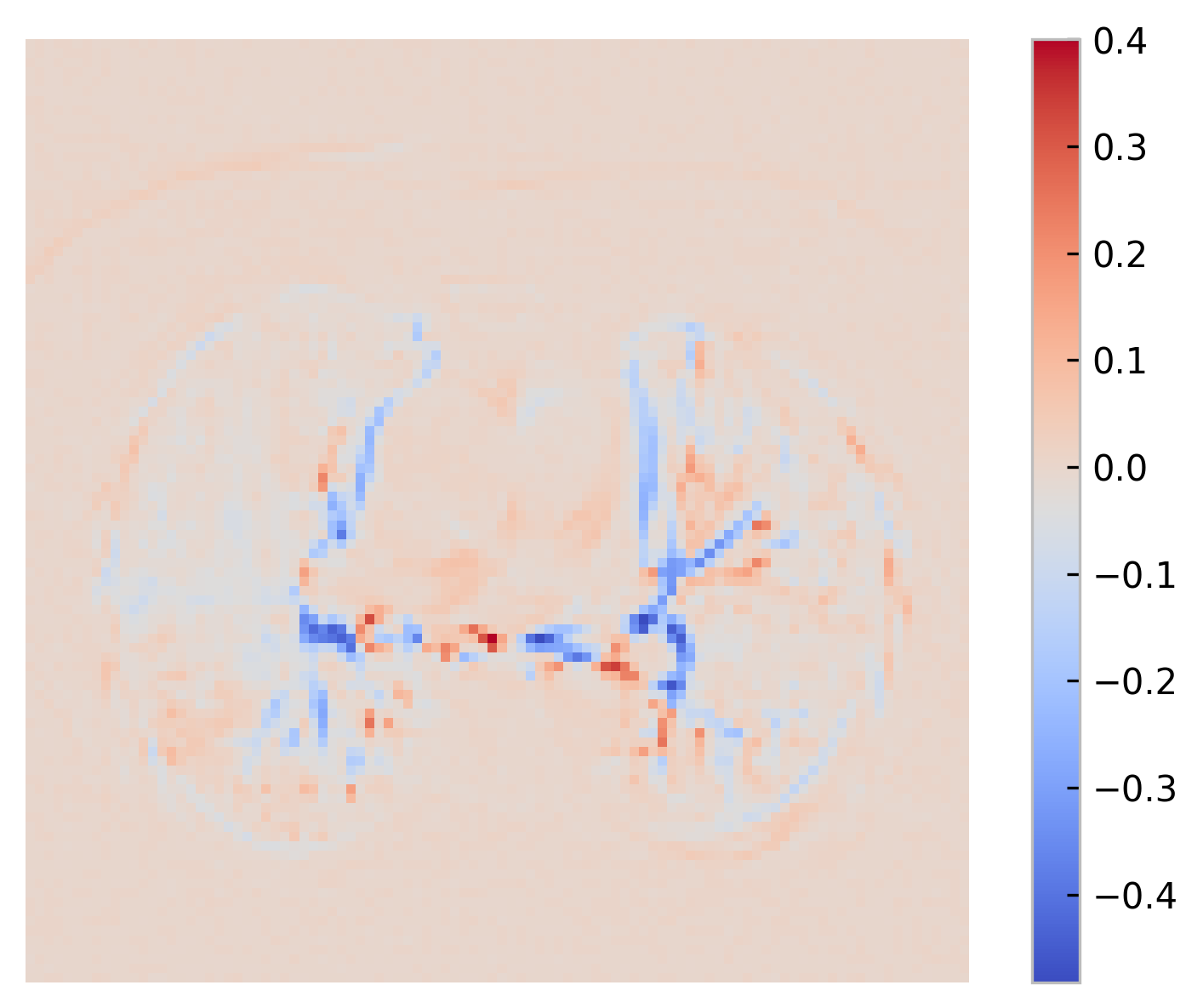}
\label{fig: lung z diff before}}
\subfigure[]{
\includegraphics[height = 1.9cm]{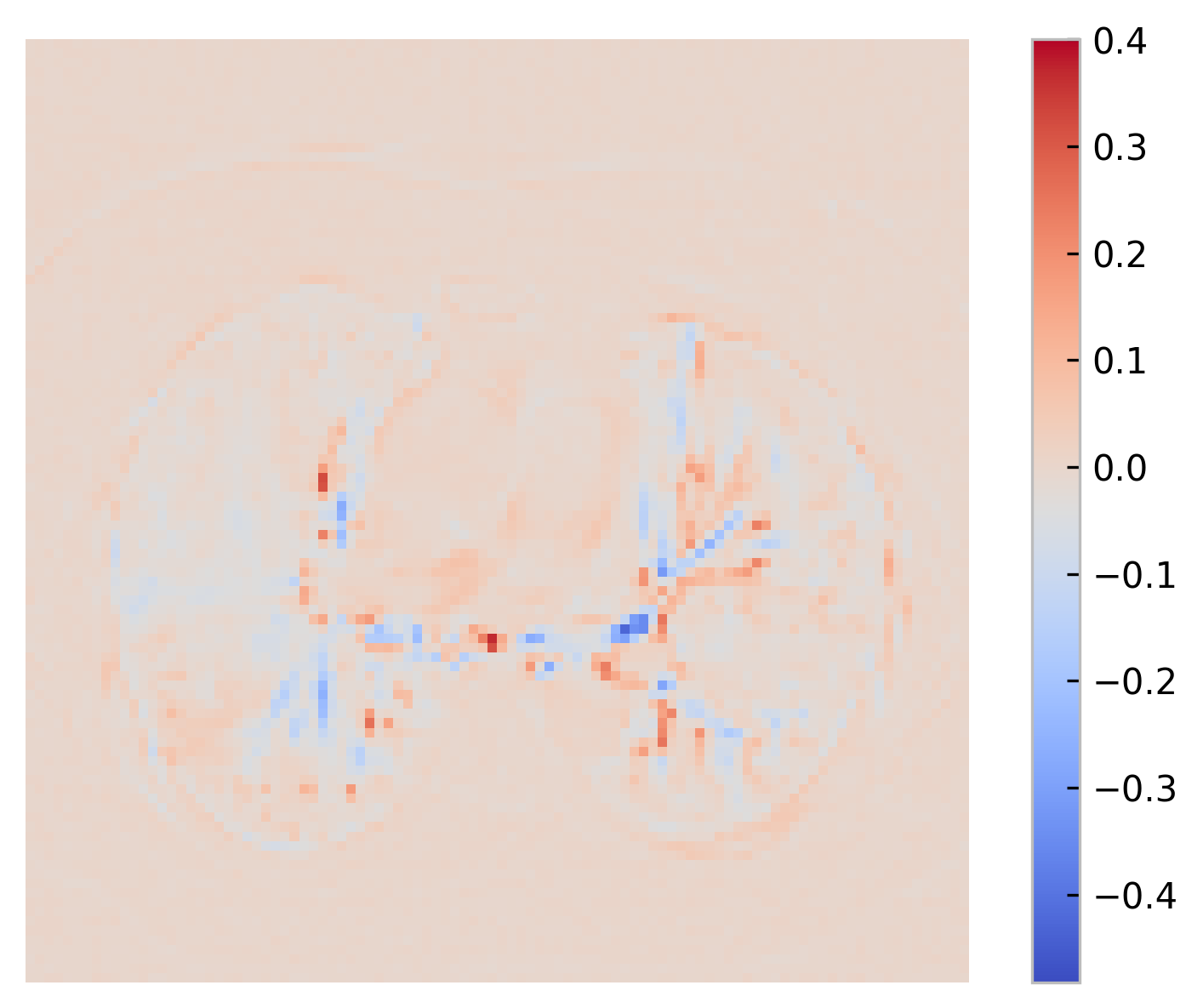}
\label{fig: lung z diff after}}
\subfigure[]{
\includegraphics[height = 1.9cm]{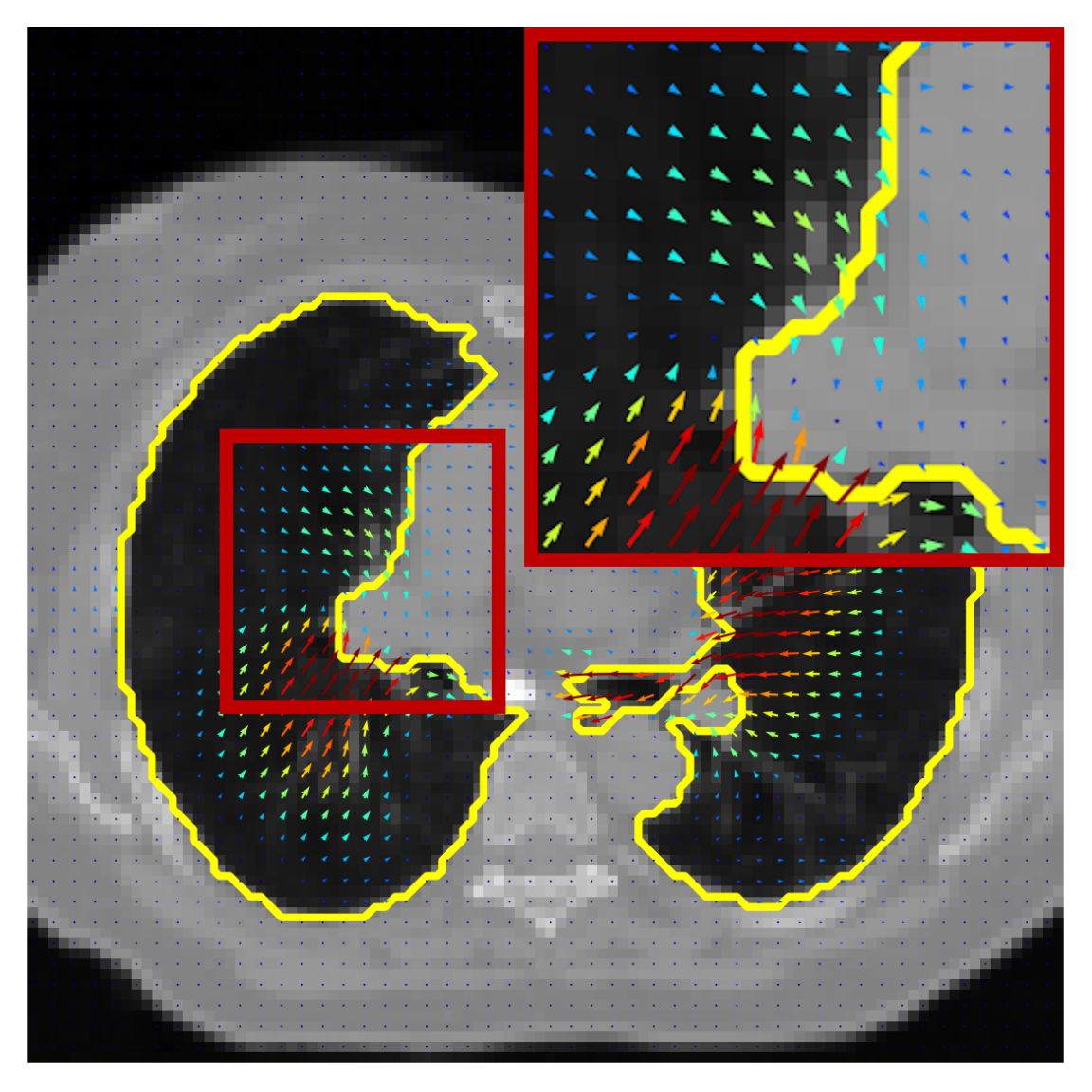}
\label{fig: lung z quiver}}
% \subfigure[]{
% \includegraphics[height = 1.9cm]{figure/lung/Vortexsheet_case5_z42_mag.png}
% \label{fig: lung z mag}}

\subfigure[]{
\includegraphics[height = 1.9cm]{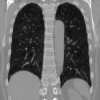}
\label{fig: lung y template}}
\subfigure[]{
\includegraphics[height = 1.9cm]{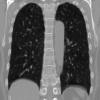}
\label{fig: lung y reference}}
\subfigure[]{
\includegraphics[height = 1.9cm]{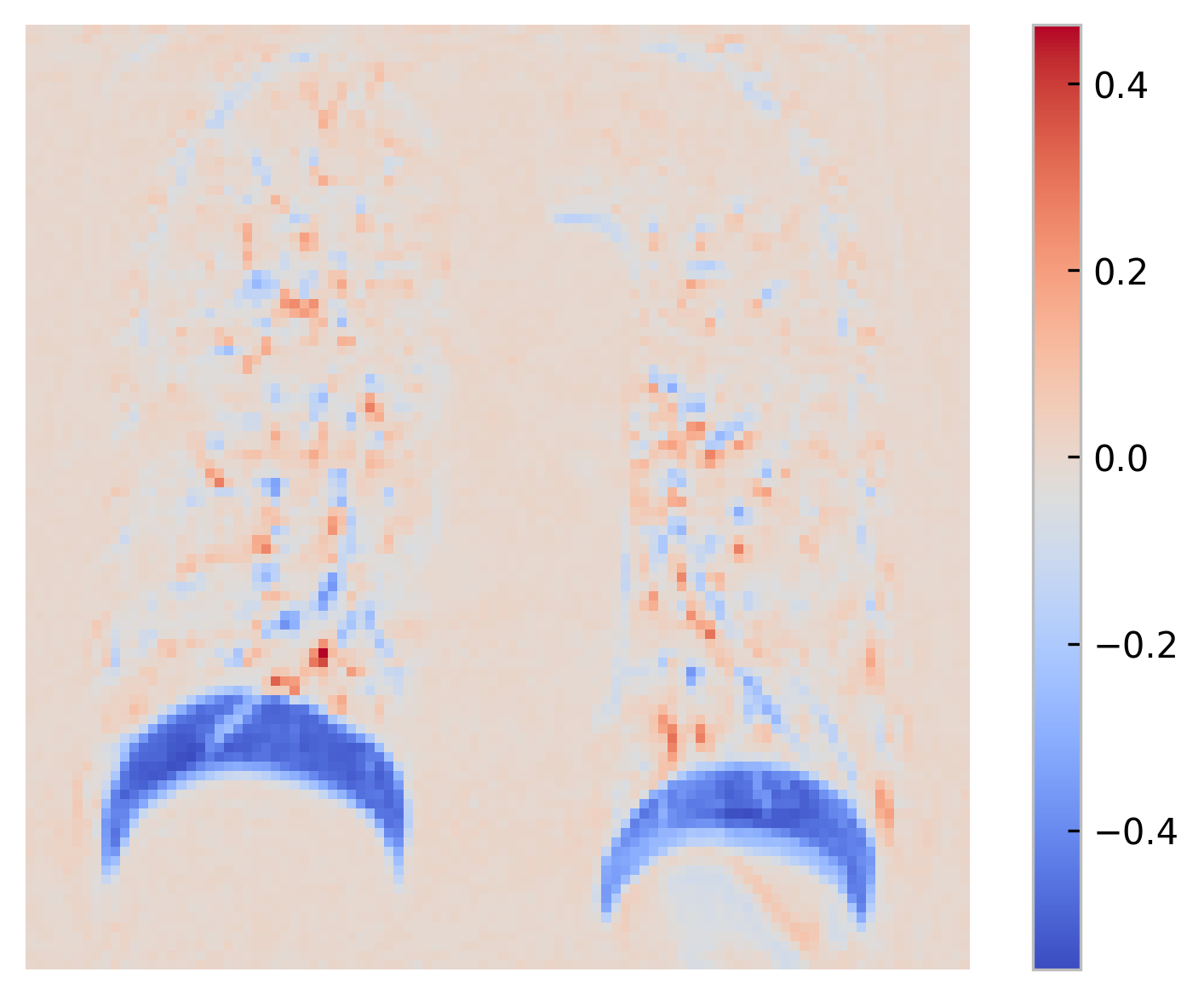}
\label{fig: lung y diff before}}
\subfigure[]{
\includegraphics[height = 1.9cm]{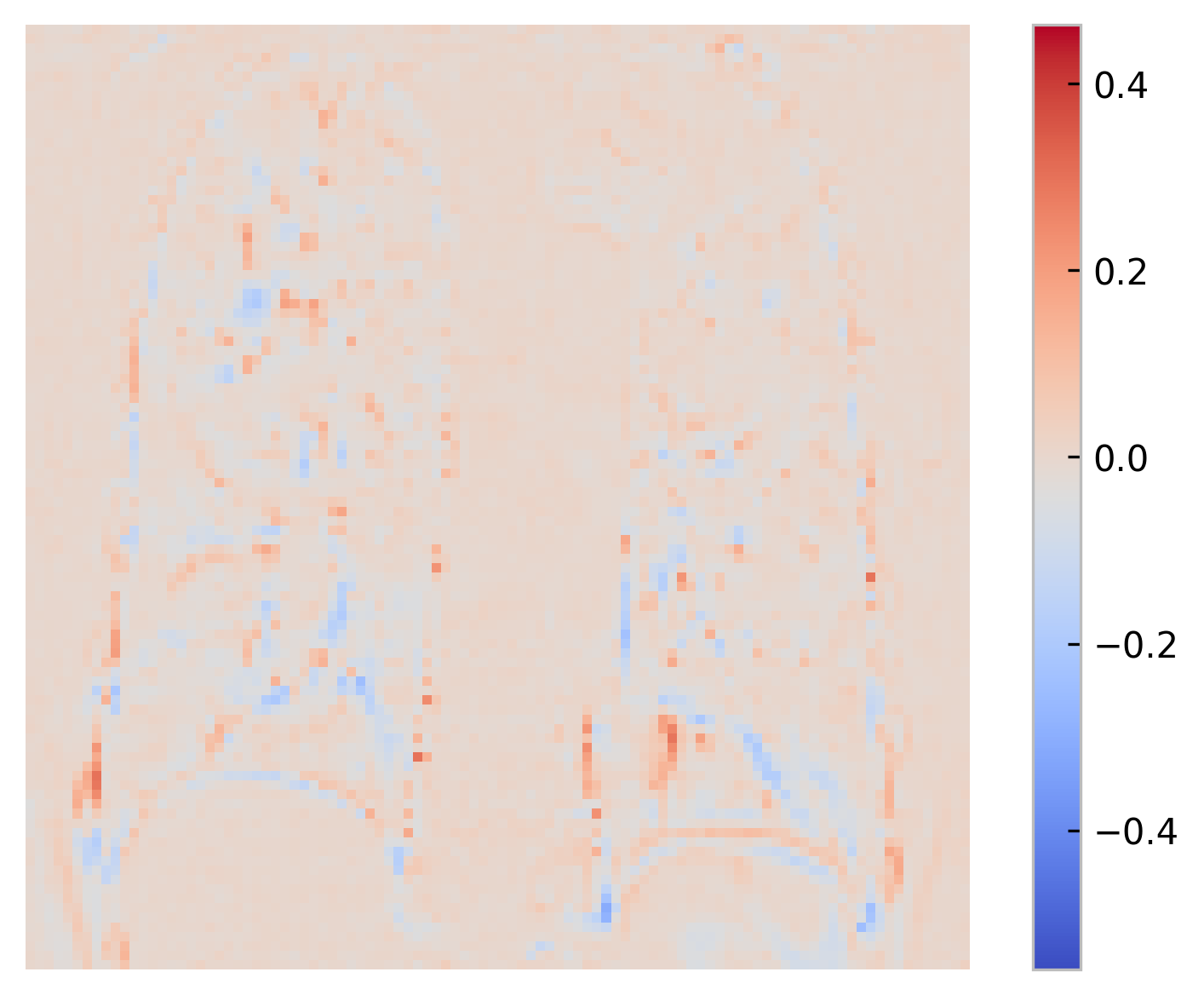}
\label{fig: lung y diff after}}
\subfigure[]{
\includegraphics[height = 1.9cm]{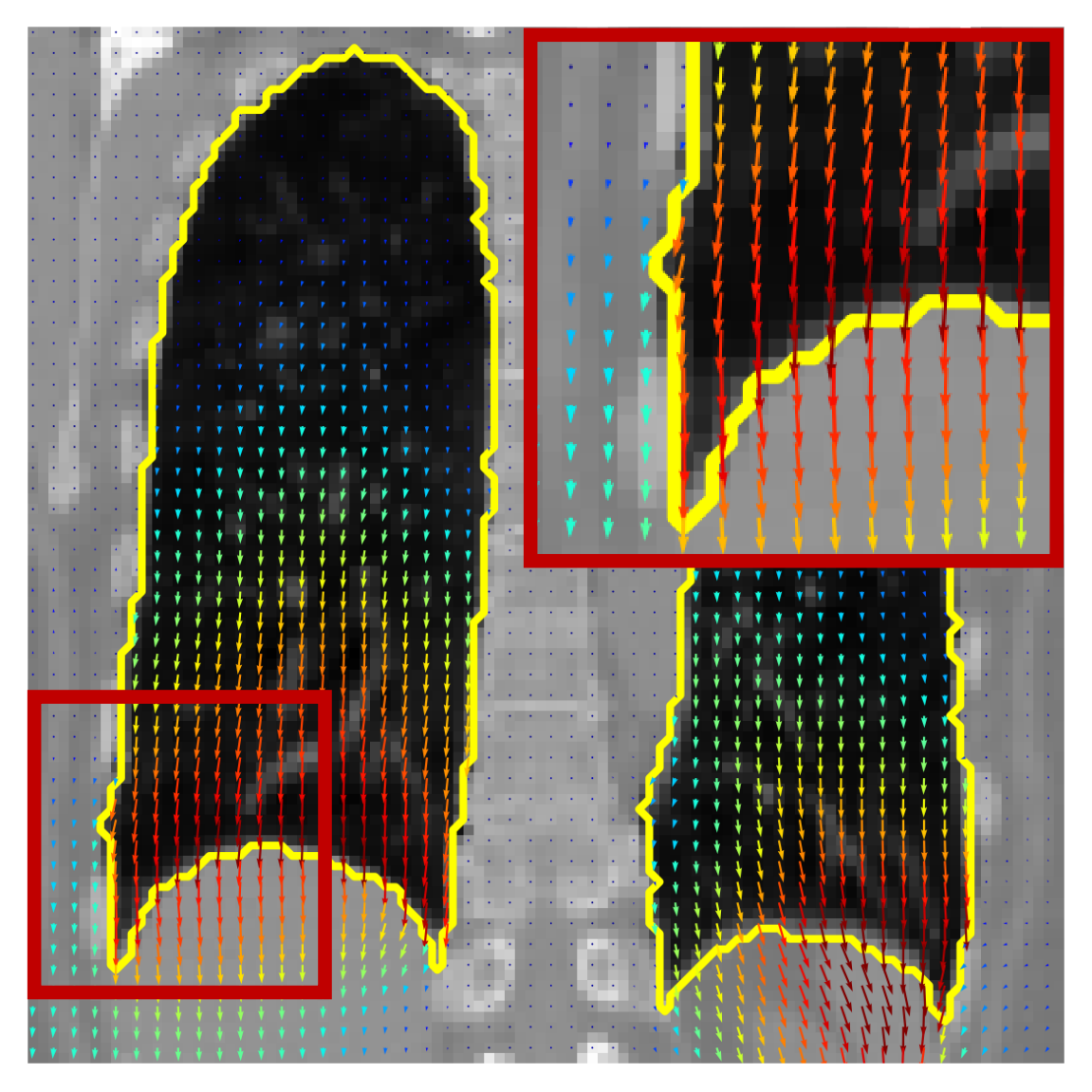}
\label{fig: lung y quiver}}
% \subfigure[]{
% \includegraphics[height = 1.9cm]{figure/lung/Vortexsheet_case5_y182_mag.png}
% \label{fig: lung y mag}}

\subfigure[]{
\includegraphics[height = 1.9cm]{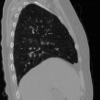}
\label{fig: lung x template}}
\subfigure[]{
\includegraphics[height = 1.9cm]{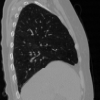}
\label{fig: lung x reference}}
\subfigure[]{
\includegraphics[height = 1.9cm]{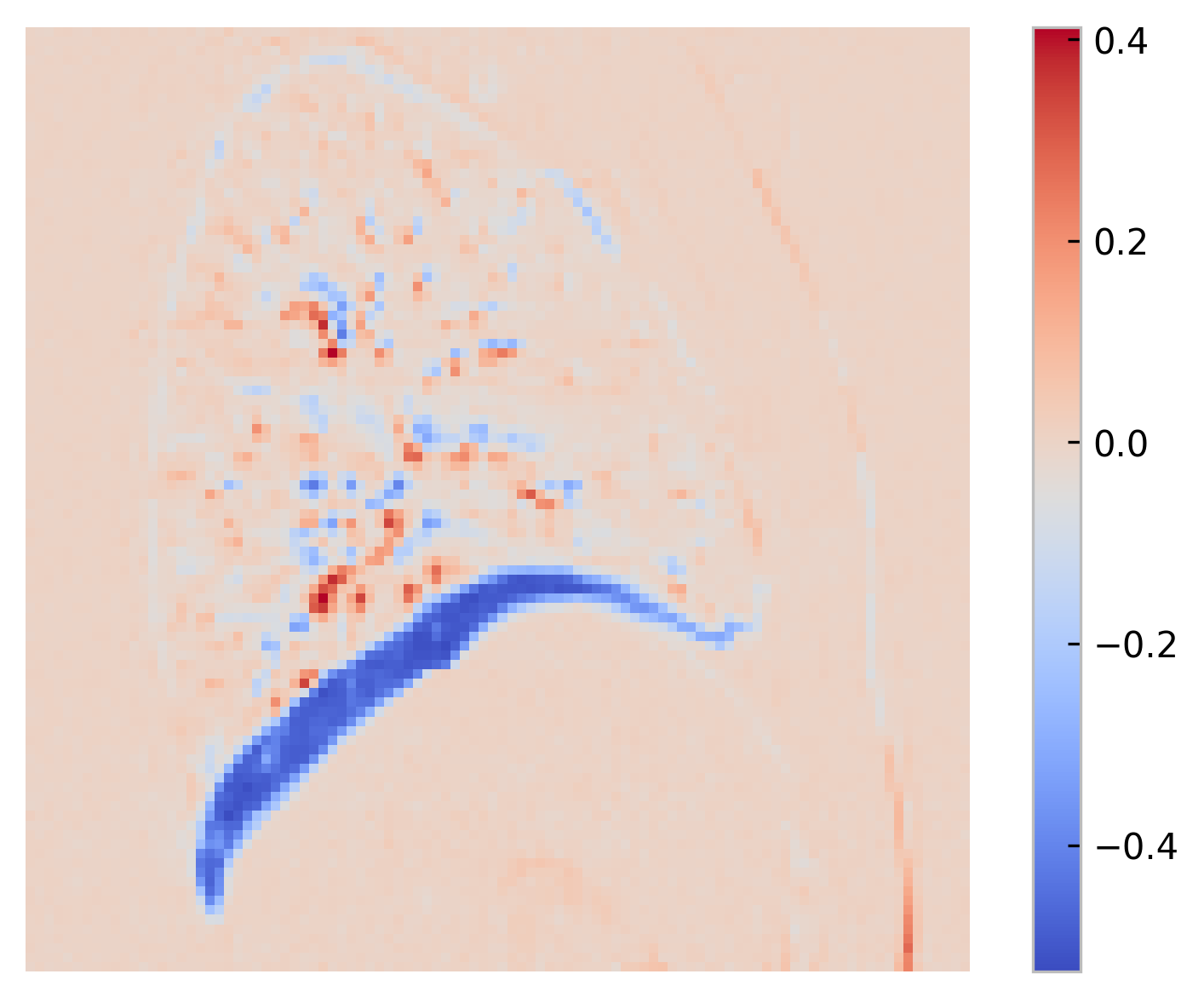}
\label{fig: lung x diff before}}
\subfigure[]{
\includegraphics[height = 1.9cm]{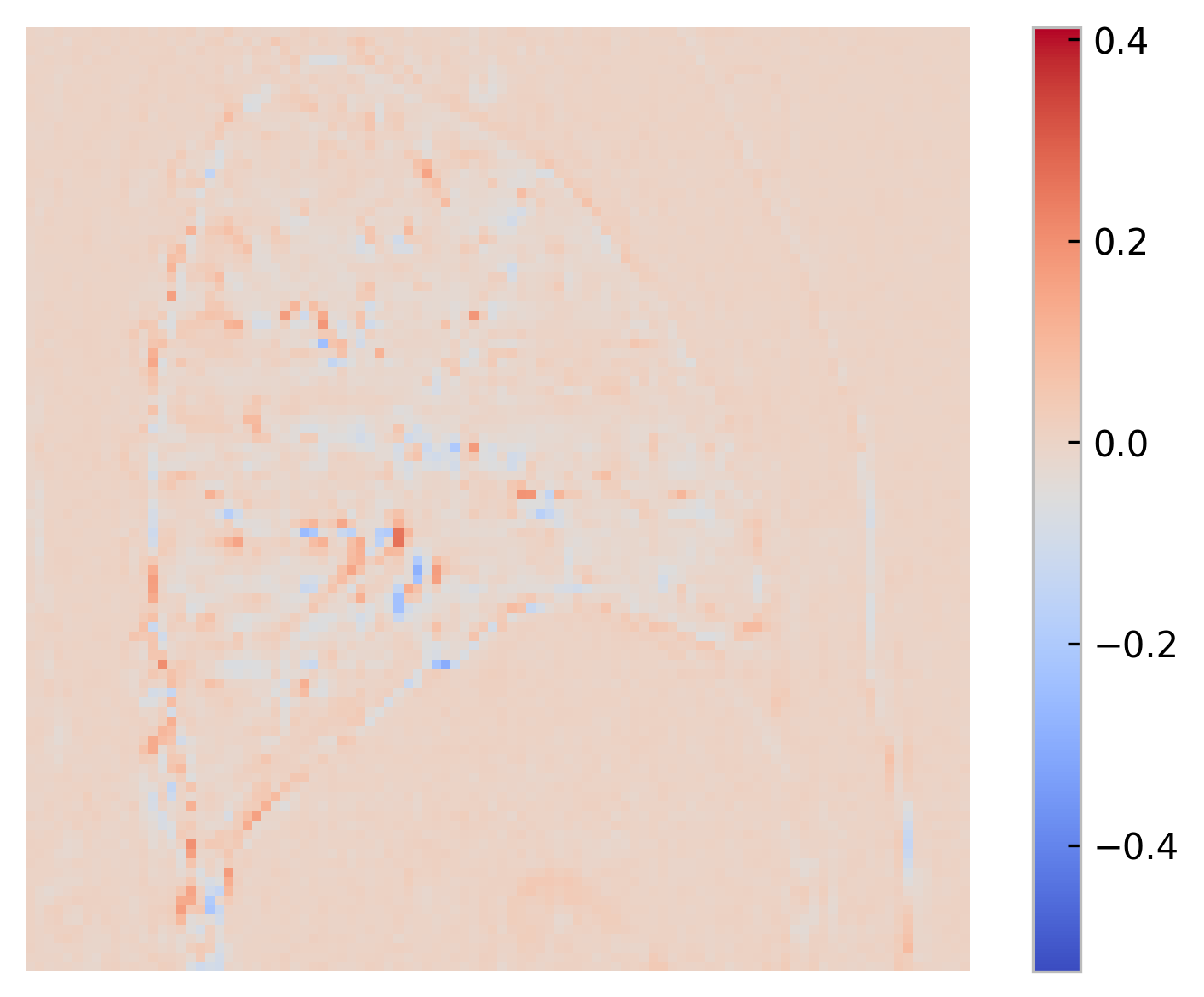}
\label{fig: lung x diff after}}
\subfigure[]{
\includegraphics[height = 1.9cm]{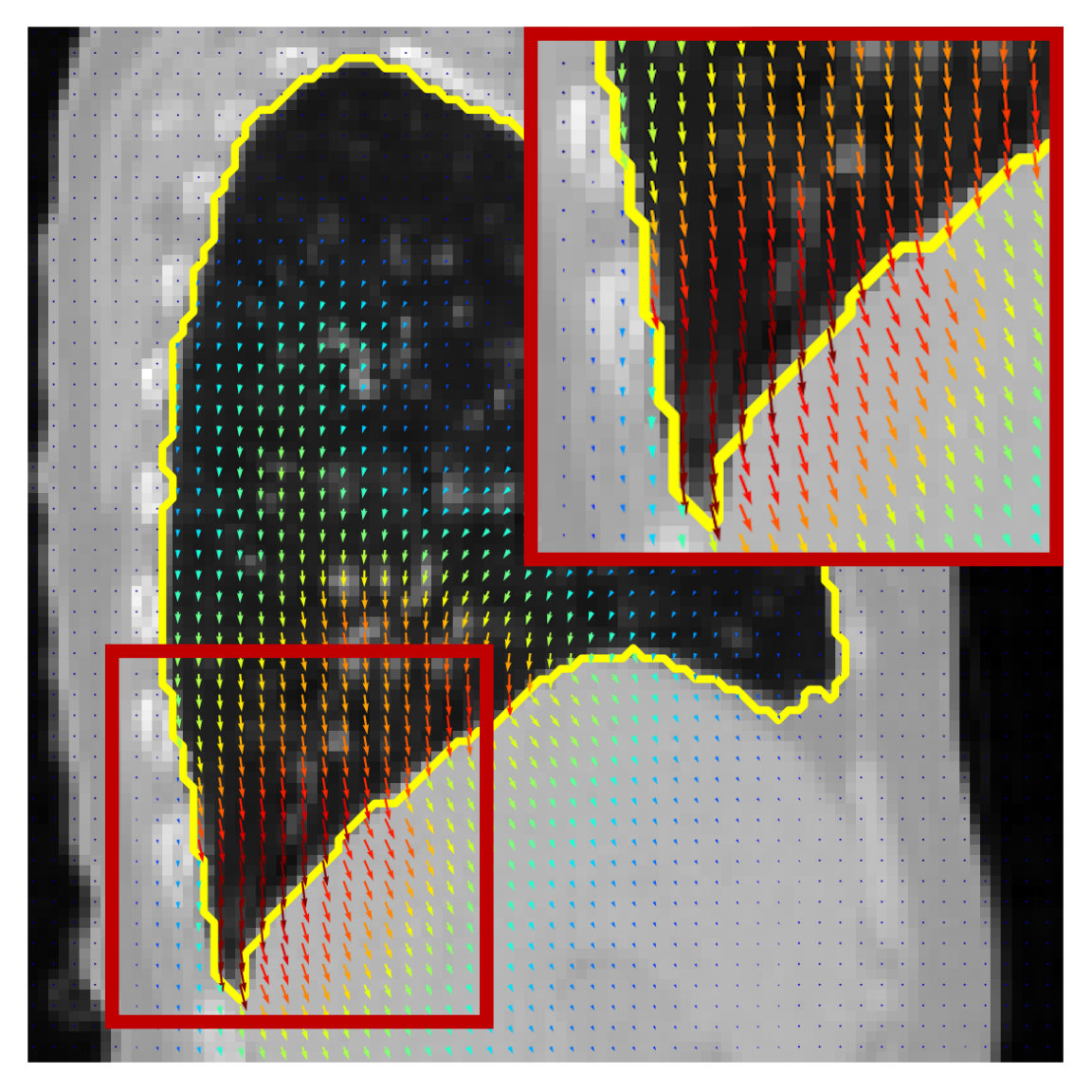}
\label{fig: lung x quiver}}
% \subfigure[]{
% \includegraphics[height = 1.9cm]{figure/lung/Vortexsheet_case5_x70_mag.png}
% \label{fig: lung x mag}}

\caption{\label{fig: lung result}Registration results on real lung images with sliding motion. Rows from top to bottom display axial, coronal, and sagittal views. The first and second columns present the moving and fixed images, respectively. The third and fourth columns depict the residual difference images before and after registration obtained by the proposed method, respectively. The last column shows the deformed moving images by the proposed method with deformation fields overlaid, respectively, in which the sliding boundaries are shown in yellow.}
\end{figure}

We summarize the evaluation results in Table \ref{tab: ReSSD} in a column in the order of $Re\_SSD$, $NCC$, and $SSIM$ values for each method. 
From Table \ref{tab: ReSSD}, we can see that the proposed method outperforms the LDDMM in terms of $Re\_SSD$, $NCC$, and $SSIM$. This shows the effectiveness of the proposed discontinuous diffeomorphism groupoid method.

Finally, we evaluate the proposed method on real lung images, where sliding motion naturally occurs between lung tissues and the surrounding structures during breathing. For this experiment, the extreme inhale images are used as the fixed images, while the extreme exhale images serve as the moving images. 
The evaluation is performed on axial, coronal, and sagittal views, as shown in Figure \ref{fig: lung result}, presented in rows from top to bottom. 
The fixed images are first segmented using the ITK SNAP program (\url{https://www.itksnap.org/}) to define the sliding boundaries, which serve as priors for the registration, and the proposed method is then used to estimate the deformation fields. 
The first and second columns in Figure \ref{fig: lung result} display the moving and fixed images, respectively, while the third and fourth columns show the residual difference images before and after registration using our proposed method. The last column shows the deformed moving images with deformation fields overlaid, in which the sliding boundaries are highlighted in yellow.
The results demonstrate that our model can handle the sliding motion effectively while preserving the structural consistency in homogeneous regions. 
This is particularly evident in the deformation fields shown within the red region of interest in Figures \ref{fig: lung z quiver}, \ref{fig: lung y quiver}, and \ref{fig: lung x quiver}, which demonstrate accurate preservation of discontinuities along the sliding boundaries. 

\section{Conclusions}
\label{sec: Conclusion}
In this paper, we have proposed a framework for discontinuous image registration by extending the classical diffeomorphism Lie group structure to a discontinuous diffeomorphism Lie groupoid. 
This new framework effectively handles discontinuities along the sliding boundary, addressing the limitation of traditional diffeomorphic registration methods.
Moreover, a rigorous mathematical analysis of the associated structures, including Lie algebroids and their duals, is also provided. The extremal Euler-Arnold equations are derived to determine optimal flows for discontinuous deformations. Numerical tests show that the proposed model can achieve a satisfactory image registration result.

An accurate segmentation of the sliding boundary is crucial for our method, so developing an integrated method that simultaneously determines the location of the sliding boundary and performs registration is one of the main outlooks of this work. Moreover, we plan to focus on extending the proposed model to 3D image registration for future research.

\bibliography{ref}

@article{izosimov2018vortex,
  title={Vortex sheets and diffeomorphism groupoids},
  author={Izosimov, Anton and Khesin, Boris},
  journal={Advances in Mathematics},
  volume={338},
  pages={447--501},
  year={2018},
  publisher={Elsevier}
}

@article{khesin2021geometric,
  title={Geometric hydrodynamics and infinite-dimensional Newton’s equations},
  author={Khesin, Boris and Misio{\l}ek, Gerard and Modin, Klas},
  journal={Bulletin of the American Mathematical Society},
  volume={58},
  number={3},
  pages={377--442},
  year={2021}
}

@book{arnold2014topological,
  title={Topological methods in hydrodynamics},
  author={Arnold, Vladimir Igorevich and Khesin, Boris A},
  volume={19},
  year={2009},
  publisher={Springer}
}

@article{izosimov2024geometry,
  title={Geometry of generalized fluid flows},
  author={A. Izosimov and B. Khesin},
  journal={Calculus of Variations and Partial Differential Equations},
  volume={63},
  number={1},
  pages={3},
  year={2024},
  publisher={Springer}
}

@article{vizman2008geodesic,
  title={Geodesic equations on diffeomorphism groups},
  author={Vizman, Cornelia and others},
  journal={Symmetry, Integrability and Geometry: Methods and Applications},
  volume={4},
  pages={030},
  year={2008},
  publisher={SIGMA. Symmetry, Integrability and Geometry: Methods and Applications}
}

@article{bauer2023liouville,
  title={Liouville comparison theory for blowup of {Euler-Arnold} equations},
  author={Bauer, Martin and Preston, Stephen C and Valletta, Justin},
  journal={arXiv preprint arXiv:2306.09748},
  year={2023}
}

@book{younes_shapes_2019,
	title = {{Shapes and Diffeomorphisms}},
	author = {Younes, Laurent},
	year = {2019},
	publisher = {Springer},
        address= {Berlin, Heidelberg}
}

@article{micheli2013sobolev,
  title={Sobolev metrics on diffeomorphism groups and the derived geometry of spaces of submanifolds},
  author={Micheli, Mario and Michor, Peter W and Mumford, David},
  journal={Izvestiya: Mathematics},
  volume={77},
  number={3},
  pages={541},
  year={2013},
  publisher={IOP Publishing}
}

@article{mumford2013euler,
  title={On {Euler's} equation and {`EPDiff'}},
  author={Mumford, David and Michor, Peter W},
  journal={Journal of Geometric Mechanics},
  volume={5},
  number={3},
  pages={319--344},
  year={2013},
  publisher={Journal of Geometric Mechanics}
}

@article{oliveira2014medical,
  title={Medical image registration: a review},
  author={Oliveira, Francisco PM and Tavares, Joao Manuel RS},
  journal={Computer Methods in Biomechanics and Biomedical Engineering},
  volume={17},
  number={2},
  pages={73--93},
  year={2014},
  publisher={Taylor \& Francis}
}

@article{viergever2016survey,
  title={A survey of medical image registration-under review},
  author={Viergever, Max A and Maintz, JB and Klein, Stefan and Murphy, Keelin and Staring, Marius and Pluim, JPW},
  journal={Medical Image Analysis},
  volume={33},
  pages={140--144},
  year={2016},
  publisher={Elsevier}
}

@book{pennec2019riemannian,
  title={Riemannian Geometric Statistics in Medical Image Analysis},
  author={Pennec, Xavier and Sommer, Stefan and Fletcher, Tom},
  year={2019},
  publisher={Academic Press},
  address = {United Kingdom}
}

@inproceedings{sommer2012kernel,
  title={Kernel bundle {EPDiff}: evolution equations for Multi-Scale diffeomorphic image registration},
  author={Sommer, Stefan and Lauze, Fran{\c{c}}ois and Nielsen, Mads and Pennec, Xavier},
  booktitle={International Conference on Scale Space and Variational Methods in Computer Vision},
  pages={677--688},
  year={2011},
  organization={Springer}
}

@article{beg2005computing,
  title={Computing large deformation metric mappings via geodesic flows of diffeomorphisms},
  author={Beg, M Faisal and Miller, Michael I and Trouv{\'e}, Alain and Younes, Laurent},
  journal={International Journal of Computer Vision},
  volume={61},
  pages={139--157},
  year={2005},
  publisher={Springer}
}

@book{modersitzki2004numerical,
  title={{Numerical Methods for Image Registration}},
  author={Modersitzki, Jan},
  year={2004},
  publisher={Oxford University Press},
  address = {New York}
}

@article{frohn-schauf_multigrid_2008,
	title = {Multigrid based total variation image registration},
	volume = {11},
	number = {2},
	journal = {Computing and Visualization in Science},
	author = {Frohn-Schauf, Claudia and Henn, Stefan and Witsch, Kristian},
	year = {2008},
	pages = {101--113},
	publisher={Springer}
}

@INPROCEEDINGS{Baig_lncc_2012,
  title={Local normalized cross correlation for geo-registration},
  author={Baig, Asim and Chaudhry, M Ali and Mahmood, Azhar},
  booktitle={International Bhurban Conference on Applied Sciences \& Technology},
  pages={70--74},
  year={2012},
  organization={IEEE}
}

@article{wang2004image,
  title={Image quality assessment: from error visibility to structural similarity},
  author={Wang, Zhou and Bovik, Alan C and Sheikh, Hamid R and Simoncelli, Eero P},
  journal={IEEE transactions on image processing},
  volume={13},
  number={4},
  pages={600--612},
  year={2004},
  publisher={IEEE}
}

@incollection{marsland2020riemannian,
  title={Riemannian geometry on shapes and diffeomorphisms: Statistics via actions of the diffeomorphism group},
  author={Marsland, Stephen and Sommer, Stefan},
  booktitle={Riemannian Geometric Statistics in Medical Image Analysis},
  pages={135--167},
  year={2020},
  publisher={Elsevier}
}

@article{bao2024sliding,
  title={Sliding at First-Order: Higher-Order Momentum Distributions for Discontinuous Image Registration},
  author={Bao, Lili and Lu, Jiahao and Ying, Shihui and Sommer, Stefan},
  journal={SIAM Journal on Imaging Sciences},
  volume={17},
  number={2},
  pages={861--887},
  year={2024},
  publisher={SIAM}
}

@article{bao2024time,
  title={Time multiscale regularization for nonlinear image registration},
  author={Bao, Lili and Chen, Ke and Kong, Dexing and Ying, Shihui and Zeng, Tieyong},
  journal={Computerized Medical Imaging and Graphics},
  volume={112},
  pages={102331},
  year={2024},
  publisher={Elsevier}
}

@book{khesin2008geometry,
  title={The geometry of infinite-dimensional groups},
  author={Khesin, Boris and Wendt, Robert},
  volume={51},
  year={2008},
  publisher={Springer Science \& Business Media}
}

@article{boucetta2011riemannian,
  title={Riemannian geometry of Lie algebroids},
  author={Boucetta, Mohamed},
  journal={Journal of the Egyptian Mathematical Society},
  volume={19},
  number={1-2},
  pages={57--70},
  year={2011},
  publisher={Elsevier}
}

@book{dufour2006poisson,
  title={Poisson structures and their normal forms},
  author={Dufour, Jean-Paul and Zung, Nguyen Tien},
  volume={242},
  year={2006},
  publisher={Springer Science \& Business Media}
}

@article{khesin2023geometric,
  title={Geometric hydrodynamics in open problems},
  author={Khesin, Boris and Misio{\l}ek, Gerard and Shnirelman, Alexander},
  journal={Archive for Rational Mechanics and Analysis},
  volume={247},
  number={2},
  pages={15},
  year={2023},
  publisher={Springer}
}

@article{dupuis1998variational,
  title={Variational problems on flows of diffeomorphisms for image matching},
  author={Dupuis, Paul and Grenander, Ulf and Miller, Michael I},
  journal={Quarterly of applied mathematics},
  pages={587--600},
  year={1998},
  publisher={JSTOR}
}

@article{TrouveA,
  title={An infinite dimensional group approach for physics based models in patterns recognition},
  author={A.~Trouv\'{e}},
  journal={Preprint},
  year={1995}}

@article{miller2006geodesic,
  title={Geodesic shooting for computational anatomy},
  author={Miller, Michael I and Trouv{\'e}, Alain and Younes, Laurent},
  journal={Journal of Mathematical Imaging and Vision},
  volume={24},
  pages={209--228},
  year={2006},
  publisher={Springer}
}

@book{holm2009geometric,
  title={{Geometric Mechanics and Symmetry: From Finite to Infinite Dimensions}},
  author={Holm, Darryl D and Schmah, Tanya and Stoica, Cristina},
  year={2009},
  publisher={Oxford University Press},
  address={USA}
}

@article{holm2005momentum,
  title={Momentum maps and measure-valued solutions (peakons, filaments, and sheets) for the {EPDiff} equation},
  author={Holm, Darryl D and Marsden, Jerrold E},
  journal={The Breadth of Symplectic and Poisson Geometry: Festschrift in Honor of Alan Weinstein},
  pages={203--235},
  year={2005},
  publisher={Springer}
}

@article{Chumchob2010A,
  title={A variational approach for discontinuity-preserving image registration},
  author={Chumchob, Noppadol and Ke, Chen},
  journal={East-west Journal of Mathematics},
  volume={2010},
  number={},
  pages={266--282},
  year={2010},
}

@article{fu_adaptive_2018,
  title={An adaptive motion regularization technique to support sliding motion in deformable image registration},
  author={Fu, Yabo and Liu, Shi and Li, H Harold and Li, Hua and Yang, Deshan},
  journal={Medical Physics},
  volume={45},
  number={2},
  pages={735--747},
  year={2018},
  publisher={Wiley Online Library}
}

@article{schmidt-richberg_estimation_2012,
	title = {Estimation of slipping organ motion by registration with direction-dependent regularization},
	volume = {16},
	number = {1},
	journal = {Medical Image Analysis},
	author = {Schmidt-Richberg, Alexander and Werner, Ren{\'e} and Handels, Heinz and Ehrhardt, Jan},
	year = {2012},
	pages = {150--159},
	publisher={Elsevier}
}

@article{mang2016constrained,
  title={Constrained H\^{}1-regularization schemes for diffeomorphic image registration},
  author={Mang, Andreas and Biros, George},
  journal={SIAM Journal on Imaging Sciences},
  volume={9},
  number={3},
  pages={1154--1194},
  year={2016},
  publisher={SIAM}
}

@article{risser_piecewise-diffeomorphic_2013,
    title={Piecewise-diffeomorphic image registration: Application to the motion estimation between {3D CT} lung images with sliding conditions},
    author={Risser, Laurent and Vialard, Fran{\c{c}}ois-Xavier and Baluwala, Habib Y and Schnabel, Julia A},
    journal={Medical Image Analysis},
    volume={17},
    number={2},
    pages={182--193},
    year={2013},
    publisher={Elsevier}
}

@article{risser_dieomorphic_2011,
        title={Diffeomorphic registration with sliding conditions: Application to the registration of lungs CT images},
        author={Risser, Laurent and Baluwala, Habib and Schnabel, Julia A},
        journal={Fourth International Workshop on Pulmonary Image Analysis, MICCAI},
        pages={79--90},
        year={2011}
}

@book{langtangen2017solving,
  title={Solving PDEs in python: the FEniCS tutorial I},
  author={Langtangen, Hans Petter and Logg, Anders},
  year={2017},
  publisher={Springer Nature}
}
\end{document}